%% file: 181105QE_revised.tex
\documentclass[a4paper,11pt]{amsart} 
\usepackage{amsmath,amsxtra,amssymb,latexsym, amscd,amsthm}
\usepackage[mathscr]{eucal}
\usepackage{mathrsfs}
\usepackage{bbm}
\usepackage{enumerate}
\usepackage{pict2e}
\usepackage{tikz}
\usepackage{graphicx}
\usepackage[a4paper]{geometry}
\usepackage{hyperref}
\numberwithin{equation}{section}

\theoremstyle{plain}
\newtheorem{thm}{Theorem}[section]
\newtheorem{prp}[thm]{Proposition}
\newtheorem{cor}[thm]{Corollary}
\newtheorem{lem}[thm]{Lemma}
\newtheorem*{euc*}{Euclidean division}
\newtheorem*{fek*}{Fekete's Lemma}
\newtheorem*{kin*}{Kingman's Subadditive Ergodic Theorem}
\newtheorem*{fur*}{Furstenberg-Kesten Theorem}
\theoremstyle{definition}
\newtheorem{defa}[thm]{Definition}

\newtheorem{rem}[thm]{Remark}

\newtheorem*{rem*}{Remark}

\newcommand{\dd}{\mathrm{d}}

\renewcommand{\Im}{\operatorname{Im}}
\renewcommand{\Re}{\operatorname{Re}}
\makeatletter
\newcommand*{\ov}[1]{%
  $\m@th\overline{\mbox{#1}}$%
}
\newcommand*{\ovA}[1]{%
  $\m@th\overline{\mbox{#1}\raisebox{3mm}{}}$%
}
\newcommand*{\ovB}[1]{%
  $\m@th\overline{\mbox{#1\rule{0pt}{3mm}}}$%
}
\newcommand*{\ovC}[1]{%
  $\m@th\overline{\mbox{#1\strut}}$%
}
\newcommand*{\ovD}[1]{%
  $\m@th\overline{\mbox{#1\vphantom{\"A}}}$%
}
\newcommand*{\ovE}[1]{%
  $\m@th\overline{\raisebox{0pt}[1.2\height]{#1}}$%
}
\newcommand*{\ovF}[1]{%
  $\m@th\overline{\raisebox{0pt}[\dimexpr\height+0.3mm\relax]{#1}}$%
}
\newcommand*{\ovG}[1]{%
  $\m@th\overline{\raisebox{0pt}[\dimexpr\height+1mm\relax]{#1\vphantom{A}}}$%
}
\makeatother
\newcommand{\N}{\mathbb{N}}

\newcommand{\R}{\mathbb{R}}
\newcommand{\C}{\mathbb{C}}

\newcommand{\tilGN}{{\widetilde{G_N}}}
\newcommand{\tilG}{{\widetilde{G}}}
\newcommand{\tilg}{{\tilde{g}}}
\newcommand{\tilWN}{{\widetilde{W_N}}}
\newcommand{\tilW}{{\widetilde{W}}}
\newcommand{\tilB}{{\widetilde{B}}}
\newcommand{\tilV}{{\widetilde{V}}}
\newcommand{\tilE}{{\widetilde{E}}}
\newcommand{\tilH}{{\widetilde{H}}}

\newcommand{\cs}{M}
\newcommand{\vertiii}[1]{{\left\vert\kern-0.25ex\left\vert\kern-0.25ex\left\vert #1 
    \right\vert\kern-0.25ex\right\vert\kern-0.25ex\right\vert}}
\DeclareMathOperator{\expect}{\mathbb{E}}

\DeclareMathOperator{\prob}{\mathbb{P}}

\DeclareMathOperator{\varnbi}{Var_{nb, \eta_0}^I}
\DeclareMathOperator{\varnbia}{Var_{nb, \eta_0}^{J_1}}
\DeclareMathOperator{\varnbib}{Var_{nb, \eta_0}^{J_2}}

\DeclareMathOperator{\varisigma}{Var_{\eta_0}^{J_j^{\varsigma}}}
\DeclareMathOperator{\varis}{Var_{\eta_0}^{ I_1\setminus \cup_{j=1}^M J^\varsigma_j}}
\DeclareMathOperator{\vari}{Var_{\eta_0}^I}
\DeclareMathOperator{\varib}{Var_{\eta_0}^{I_1}}

\DeclareSymbolFont{extraup}{U}{zavm}{m}{n}
\DeclareMathSymbol{\varheart}{\mathalpha}{extraup}{86}
\DeclareMathSymbol{\vardiamond}{\mathalpha}{extraup}{87}
\input{def.tex}

\title[From Spectral to Spatial Delocalization]{Quantum Ergodicity on Graphs~: from Spectral to Spatial Delocalization}
\author{Nalini Anantharaman and Mostafa \textsc{Sabri}}
\address{Universit\'e de Strasbourg, CNRS, IRMA UMR 7501, F-67000 Strasbourg, France.}
\email{anantharaman@math.unistra.fr}
\address{Department of Mathematics, Faculty of Science, Cairo University, Cairo 12613, Egypt.}
\address{Universit\'e de Strasbourg, CNRS, IRMA UMR 7501, F-67000 Strasbourg, France.}
\email{mmsabri@sci.cu.edu.eg}
\usepackage{calc}
\usepackage{yhmath}
\usepackage{graphicx}

\makeatletter
\newlength{\temp@wc@width}
\newlength{\temp@wc@height}
\newcommand{\widecheck}[1]{%
  \setlength{\temp@wc@width}{\widthof{$#1$}}%
  \setlength{\temp@wc@height}{\heightof{$#1$}}%
  #1\hspace{-\temp@wc@width}%
  \raisebox{\temp@wc@height+2pt}[\heightof{$\widehat{#1}$}]%
     {\rotatebox[origin=c]{180}{\vbox to 0pt{\hbox{$\widehat{\hphantom{#1}}$}}}}%
}
\makeatother

\begin{document}

\begin{abstract}
We prove a quantum-ergodicity theorem on large graphs, for eigenfunctions of Schr\"odinger operators in a very general setting. We consider a sequence of finite graphs endowed with discrete Schr\"odinger operators, assumed to have a local weak limit. We assume that our graphs have few short loops, in other words that the limit model is a random rooted {\em{tree}} endowed with a random discrete Schr\"odinger operator. We show that absolutely continuous spectrum for the infinite model, reinforced by a good control of the moments of the Green function, imply
``quantum ergodicity'', a form of spatial delocalization for eigenfunctions of the finite graphs approximating the tree. This roughly says that the eigenfunctions become equidistributed in phase space. Our result applies in particular to graphs converging to the Anderson model on a regular tree, in the r\'egime of extended states studied by Klein and Aizenman--Warzel.

%Note that we test the equidistribution of eigenfunctions by measuring some {\em a priori given} macroscopic observable. Thus our results are not incompatible with the numerical results of \cite{alts1, alts2} for the Anderson model on random regular graphs, that seem to indicate that each eigenfunction is supported on a relatively small portion of phase space.
\end{abstract}

\maketitle

\section{Introduction}          \label{sec:introd}

\subsection{The problem}
Consider a very large, but finite, graph $G=(V, E)$. Are the eigenfunctions of its adjacency matrix {\em{localized}}, or {\em{delocalized}}~? These words are used in a variety of contexts, with several different meanings.

For discrete Schr\"odinger operators on infinite graphs (e.g. for the celebrated Anderson model describing the metal-insulator transition), localization can be understood in a spectral, spatial or dynamical sense. Given an interval $I\subset \R$, one can consider
\begin{itemize}
\item \emph{spectral localization :} pure point spectrum in $I$,
\item \emph{exponential localization :} the corresponding eigenfunctions decay exponentially,
\item \emph{dynamical localization :} an initial state with energy in $I$ which is localized in a bounded domain essentially stays in this domain as time goes on.
\end{itemize}

On the opposite, delocalization may be understood at different levels :
\begin{itemize}
\item \emph{spectral delocalization :} purely absolutely continuous spectrum in $I$,
%\item \emph{spatial delocalization :} the corresponding (generalized) eigenfunctions do not concentrate on small regions. Ideally, they are uniformly distributed,
\item \emph{ballistic transport :} wave packets with energies in $I$ spread on the lattice at a specific (ideally, linear) rate as time goes on.
\end{itemize}
In this paper we want to discuss a notion of \emph{spatial delocalization}. Since the wavefunctions corresponding to absolutely continuous spectrum are not square summable, a natural interpretation of spatial delocalization is to consider a sequence of growing ``boxes'' or finite graphs $(G_N)$ approximating the infinite system in some sense, and ask if the eigenfunctions on $(G_N)$ become delocalized as $N\to \infty$. Can they concentrate on small regions, or, on the opposite, are they uniformly distributed over $(G_N)$~? Large, finite graphs are also a subject of interest on their own. Actually, an infinite system is often an idealized version of a large finite one.

Localization/delocalization of eigenfunctions is believed to bear some relation with {\em{spectral statistics}}~: localization is supposedly associated with Poissonian spectral statistics, whereas delocalization should be associated with Random Matrix statistics (GOE/GUE). In the field of quantum chaos, the former notion is often associated with {\em{integrable dynamics}} and the latter with {\em{chaotic dynamics}} \cite{BT, BGS1, BGS2}. However, specific examples show that the relation is not so straightforward \cite{SarnakSchur, SarnakPoisson, Mark} Understanding how far one can push these ideas is one amongst many reasons for studying models of large graphs \cite{Keating, Smi07, Smi10}.

Recently, the question of delocalization of eigenfunctions of large matrices or large graphs has been a subject of intense activity. Let us mention several ways of testing delocalization that have been used. Let $M_N$ be a large symmetric matrix of size $N\times N$, and let $(\psi_j)_{j=1}^N$ be an orthonormal basis of eigenfunctions. The eigenfunction $\psi_j$ defines a probability measure
$\sum_{x=1}^N|\psi_j(x)|^2\delta_x$. The goal is to compare this probability measure with the uniform measure, which puts mass $1/N$ on each point.
\begin{itemize}
\item \emph{$\ell^\infty$ norms~:} Can we have a pointwise upper bound on $|\psi_j(x)|$, in other words, is $\Vert\psi_j\Vert_\infty$ small, and how small compared with $1/{\sqrt{N}}$~?
\item \emph{$\ell^p$ norms:} Can we compare $\Vert\psi_j\Vert_p$ with $N^{1/p-1/2}$~? In \cite{Alt1}, a state $\psi_j$ is called {\emph{non-ergodic}} (and {\emph{multi-fractal}}) if $\Vert\psi_j\Vert_p$ behaves like
$N^{f(p)}$ with $f(p)\not=1/p-1/2$. Related criteria appear in \cite{ASW}.
\item \emph{Scarring~: } Can we have full concentration ($\sum_{x\in \Lambda}|\psi_j(x)|^2\geq 1-\eps$) or partial concentration ($\sum_{x\in \Lambda}|\psi_j(x)|^2\geq \eps$) with $\Lambda$ a set of ``small'' cardinality~? We borrow the term ``scarring'' from the term used in the theory of quantum chaos \cite{SarnakSchur}.
\item \emph{Quantum ergodicity~:} Given a function $a: \{1, \ldots, N\}\To \IC$, can we compare $\sum_{x}a(x)|\psi_j(x)|^2$ with $\frac1N\sum_x a(x)$~? This criterion, borrowed again from quantum chaos, was applied to discrete regular graphs in \cite{ALM, A}. Quantum ergodicity means that the two averages are close for {\emph{most}} $j$. If they are close for {\emph{all}} $j$, one speaks of \emph{quantum unique ergodicity}.
\end{itemize}
As was demonstrated in a recent series of papers, adding some randomness may allow to settle the problem completely. For instance {\em{almost sure}} optimal $\ell^\infty$-bounds and quantum unique ergodicity for various models of {\emph{random}} matrices and {\emph{random}} graphs, such as Wigner matrices, sparse Erd\"os-R\'enyi graphs, random regular graphs of slowly increasing or bounded degrees were obtained in \cite{ESY09, ESY09-2, BourgadeYau13, EKYY, BKY, BHKY, BHY}. The invariance of the probability distribution under certain elementary transformations plays an important role. The completely different point of view that we adopt is to consider deterministic graphs and to prove delocalization as resulting directly from the geometry of the graphs. Up to now, in this deterministic setting, only eigenfunctions of the adjacency matrix of regular graphs have been treated, taking advantage of the completely explicit Fourier analysis on regular trees. The papers \cite{ALM, BLML, A} give various proofs of quantum ergodicity; the paper \cite{BL} proves the absence of scarring on sets of cardinality $N^{1-\eps}$ and also contains (although not stated) a logarithmic upper bound on the $\ell^\infty$ norms.

The aim of this paper is to prove a \emph{quantum ergodicity theorem} for eigenfunctions of discrete Schr\"odinger operators on quite general large graphs. As we will see, a particularly interesting point of our result is that it gives a direct relation between {\emph{spectral delocalization}} of infinite systems and  {\emph{spatial delocalization}} of large finite system. Our result may be summarized as follows (with proper additional assumptions to be described later)~:

\medskip

{ \emph{``If a large finite system is close (in the Benjamini-Schramm topology) to an infinite system having purely absolutely continuous spectrum in an interval $I$, then the eigenfunctions (with eigenvalues lying in $I$) of the finite system satisfy quantum ergodicity.''}}

%Roughly speaking, the quantum ergodicity property means that most eigenfunctions get evenly distributed in phase space when the graph is sufficiently large; a form of \emph{delocalization}. We will assume that the finite graphs approximate a tree in the Benjamini--Schramm sense, and we will prove that {\em purely absolutely continuous} spectrum for the limiting Schr\"odinger operator on the tree, reinforced by a good control of the moments of the Green function, imply quantum ergodicity for the finite size approximations.

\subsection{The results}       \label{sec:results}
Consider a sequence of connected graphs without self-loops and multiple edges $(G_N)_{N \in \N}$. We assume each vertex has at least $3$ neighbours. It will be convenient to write $G_N$ as a quotient of a tree ${\widetilde{G_N}}$ by a group of automorphisms $\Gamma_N$, that is, $G_N =\Gamma_N\backslash \tilGN$, where $\Gamma_N$ acts freely on the vertices of $\tilGN$, i.e. given $v \in\tilGN$, $\gamma_1 v = \gamma_2 v$ implies $\gamma_1=\gamma_2$. In other words, 
${\widetilde{G_N}}$ is the ``universal cover'' of $G_N$.
We will work under the assumption that the degree of $\tilGN$ is everywhere smaller than some fixed $D$. 

%que se passe-t-il s'il y a des feuilles~?

We denote by $\widetilde{V_N}$ and $\widetilde{E_N}$ the set of vertices and edges of $\tilGN$, respectively.
We denote by $V_N$ and $E_N$ the vertices and edges of $G_N$, respectively. We assume $|V_N| = N$ and work in the limit $N \To \infty$.

Define the adjacency operator $\widetilde{\mathcal{A}}_N:\IC^{\tilGN}\to\IC^{\tilGN}$ by
\[
(\widetilde{\mathcal{A}}_Nf)(v) = \sum_{w \sim v} f(w) \, ,
\]
where $v \sim w$ means $v$ and $w$ are nearest neighbours.
The operator $ \widetilde{\mathcal{A}}_N$
is bounded on $\ell^2(\tilGN)$. It also preserves the space of $\Gamma_N$-invariant functions on $\widetilde{V_N}$, in other words it defines an operator on $\ell^2(V_N)$, that we denote by $\cA_N$
(we will drop the index $N$ and write $ \widetilde{\mathcal{A}}, \cA$ when no confusion may arise).
Consider a bounded function $\widetilde{W_N} :\widetilde{V_N}\To \IR$ such that $\tilWN (\gamma\cdot v)=\tilWN(v)$ for all $\gamma\in \Gamma_N$. The operator of multiplication by $  \tilWN$
is bounded on $\ell^2(\tilGN)$; it also preserves the space of $\Gamma_N$-invariant functions on $\widetilde{V_N}$, thus it defines an operator on $\ell^2(V_N)$, that we denote by $W_N$. We define the discrete Schr\"odinger operators
$\widetilde{H}_N=\widetilde{\mathcal{A}}_N+ \widetilde{W_N}$ and $H_N=\cA_N + W_N$. The central object of our study are the eigenfunctions of $H_N$, and their behaviour (localized/delocalized) as $N\To +\infty$. The fact that  $\Gamma_N$ acts freely implies that $H_N$ is symmetric (self-adjoint) on $\ell^2(V_N)$.

For comfort, we will always work under the assumption that $W_N$ takes values in some fixed interval $[-A, A]$. This implies that the spectrum of all operators we will encounter is contained in some fixed interval $I_0=[-A-D, A+D]$.

We define the Laplacian $P_N:\C^{V_N}\to\C^{V_N}$ by
\begin{equation}\label{e:lapl}
(P_Nf)(x) = \frac{1}{d_N(x)}\sum_{y\sim x} f(y) \, ,
\end{equation}
where $d_N(x)$ stands for the number of neighbours of $x$. If we introduce the positive measure on $V_N$ assigning to $x$ the weight $d_N(x)$, then $P_N$ is self-adjoint on $\ell^2(V_N, d_N)$.

We shall assume the following conditions on our sequence of graphs:

\medskip

\textbf{(EXP)} The sequence $(G_N)$ forms an expander family. By this we mean that the Laplacian $P_N$ has a uniform spectral gap in $\ell^2(V_N, d_N)$. More precisely, the eigenvalue $1$ of $P_N$ is simple, and the spectrum of $P_N$ is contained in $[-1+\beta, 1-\beta]\cup \{1\}$, where $\beta>0$ is independent of $N$.

\medskip

Note that $1$ is always an eigenvalue, corresponding to constant functions. Our assumption implies in particular that each $G_N$ is connected and non-bipartite. 
It is well-known that a uniform spectral gap for $P_N$ is equivalent to a Cheeger constant bounded away from $0$ (see for instance \cite{Diac91}, \S 3). 

Our second assumption is that $(G_N)$ has few short loops:

\medskip
{\bf{(BST)}}
For all $r>0$,
\[
\lim_{N \to \infty} \frac{|\{x \in V_N : \rho_{G_N}(x)<r\}|}{N} = 0 \, ,
\]
where $\rho_{G_N}(x)$ is the \emph{injectivity radius} at $x$, i.e. the largest $\rho$ such that the ball $B_{G_N}(x,\rho)$ is a tree. 

\medskip

The general theory of Benjamini-Schramm convergence (or local weak convergence), briefly recalled in Appendix~\ref{s:BSCT}, allows us to assign a limit object to the sequence $(G_N,W_N)$, which is a probability distribution carried on \emph{trees}. More precisely, up to passing to a subsequence, assumption \textbf{(BST)} above is equivalent to the following assumption.

\medskip
 
\textbf{(BSCT)} The sequence $(G_N, W_N)$ has a local weak limit $\IP$ which is concentrated on the set of (isomorphism classes of) coloured rooted \emph{trees}, denoted $\mathscr{T}_{\ast}^{D,A}$.

\medskip

Assumption \textbf{(BSCT)} says that $(G_N,W_N)$ converges in a distributional sense to a random system of rooted trees $\{[\mathcal{T},o]\}$, endowed with a map $\cW :\mathcal{T}\To \IR$. More precisely, the empirical measure of $(G_N,W_N)$, defined by choosing a root $x\in V_N$ uniformly at random, converges weakly to a probability measure $\prob$ concentrated on trees.

If $[\mathcal{T},o,\mathcal{W}]\in \mathscr{T}_{\ast}^{D,A}$ and $\cA$ is the adjacency matrix of $\cT$, we denote by $\cH=\cA+\cW$ the limiting random Schr\"odinger operator, which is self-adjoint on $\ell^2(\cT)$.

Call $(\lambda^{(N)}_j)_{j=1}^N$ the eigenvalues of $H_N$ on $\ell^2(V_N)$.
Assumption \textbf{(BSCT)} implies the convergence of the empirical law of eigenvalues~: for any continuous $\chi:\IR\To \IR$, we have
\begin{equation}\label{e:IDS1}
\frac1N\sum_{j=1}^N\chi(\lambda^{(N)}_j)\Lim_{N\To +\infty}  \IE\left(\la \delta_o, \chi(\cH)\delta_o\ra\right)=:\rho(\chi) \, ,
\end{equation}
see Remark~\ref{rem:IDS1}. Here $\expect$ is the expectation with respect to $\prob$, that is,
\[
\expect(f) = \int_{\mathscr{T}_{\ast}^{D,A}} f([\mathcal{T},o,\mathcal{W}])\,\dd\prob([\mathcal{T},o,\mathcal{W}]) \, .
\]
The measure $\rho$ is called the \emph{integrated density of states} in the theory of random Schr\"odinger operators.

We need some notation for our last assumption. Let $[\mathcal{T},o,\mathcal{W}]\in \mathscr{T}_{\ast}^{D,A}$. Given $x, y\in \cT$, and $\gamma\in \IC\setminus \IR$, we introduce the Green function
\[
\cG^{\gamma}(x, y)=\la \delta_x, (\cH-\gamma)^{-1}\delta_y\ra_{\ell^2(\cT)} \, .
\]
Given $v,w \in \cT$ with $v \sim w$, we denote by ${\cT}^{(v|w)}$ the tree obtained by removing from the tree ${\cT}$ the branch emanating from $v$ that passes through $w$. We define the restriction $\mathcal{H}^{(v|w)}(u,u') = \mathcal{H}(u,u')$ if $u,u'\in \mathcal{T}^{(v|w)}$ and zero otherwise. The corresponding Green function is denoted by $\cG^{(v|w)}(\cdot, \cdot;\gamma)$. We then put $\hat{\zeta}_w^{\gamma}(v) := -\cG^{(v|w)}(v,v;\gamma)$.

\medskip

\textbf{(Green)} There is a non-empty open set $I_1$, such that for all $s>0$ we have
\[
\sup_{\lambda\in I_1, \eta_0\in(0,1)}\IE\left(\sum_{y : y\sim o} |\Im \hat{\zeta}^{\lambda+i\eta_0}_o(y)|^{-s}\right) < \infty \, .
\]

To understand \textbf{(Green)}, define the (rooted) spectral measure of $[\mathcal{T},o,\mathcal{W}]\in \mathscr{T}_{\ast}^{D,A}$ by
\begin{equation}         \label{e:muo}
\mu_o(J) = \langle \delta_o, \chi_J(\mathcal{H})\delta_o\rangle \qquad \text{for Borel } J \subseteq \R \, .
\end{equation}
Assumption \textbf{(Green)} implies that $\sup_{\lambda\in I_1,\eta_0>0}\expect(|\mathcal{G}^{\gamma}(o,o)|^2)<\infty$; see Remark~\ref{rem:IDS2}. As shown in \cite{Klein}, this implies that for $\prob$-a.e. $[\mathcal{T},o,\mathcal{W}]\in \mathscr{T}_{\ast}^{D,A}$, the spectral measure $\mu_o$ is absolutely continuous in $I_1$, with density $\frac{1}\pi \Im \cG^{\lambda+i0}(o, o)$. Hence, \textbf{(Green)} implies that $\prob$-a.e. operator $\mathcal{H}$ has purely absolutely continuous spectrum in $I_1$. This is a natural assumption since our aim is to prove delocalization properties of eigenfunctions.

 % \begin{rem} An example of a measure $\IP$ for which (Green) is satisfied (and our main motivation for writing this paper) is the case of the Anderson model on a $(q+1)$-regular tree. More precisely, $\cT$ is the (deterministic) rooted $(q+1)$-regular tree $\cT=\IT_q$, with a random colouring $\cW : V(\IT_q)\To \IR$ such that the random variables $\cW(v), v\in V(\IT_q)$ are i.i.d. The results of Klein, Froese--Hassler--Spitzer, Aizenman--Warzel show that...

% \end{rem}

Now let $(\psi^{(N)}_j)_{j=1}^N$ be an orthonormal basis of $\ell^2(V_N)$ consisting of eigenfunctions of $H_N$. Pick $j\in \{1,\dots, N\}$. The problem of quantum ergodicity is to understand if the probability measure $\sum_{x\in V_N}| \psi^{(N)}_j(x)|^2\delta_x$ on $V_N$
is ``localized'' (essentially carried by $o(N)$ vertices) or ``delocalized'' (ideally, close to the uniform measure on $V_N$, or maybe, to some other natural measure on $V_N$, comparable to the uniform measure). More generally, we want to know if the correlations $\overline{\psi^{(N)}_j(x)}\psi^{(N)}_j(y)$, for $x$ and $y\in V_N$ at some fixed distance, approach some limiting object.
From a mathematical point of view, the question was addressed in \cite{ALM, BLML} for eigenfunctions of the adjacency matrix of large deterministic {\em{regular}} graphs, and for the adjacency matrix of \emph{random} regular graphs or Erd\"os-R\'enyi graphs in the recent works \cite{EKYY, BKY, BHKY, BHY}. The main motivation of our paper is to extend the results of \cite{ALM} to disordered systems, that is, to non-regular graphs, possibly with a potential on the vertices or weights on the edges. This necessarily requires a different method from that of \cite{ALM}, that was specific to regular graphs. New methods to prove quantum ergodicity were already explored in \cite{A}. We insist on the fact that, contrary to \cite{EKYY, BKY, BHKY, BHY, Geisinger}, our sequence of graphs and potentials are deterministic. The results may in particular be applied to random graphs and/or random potentials, provided one knows that Assumptions \textbf{(EXP)}, \textbf{(BSCT)} and \textbf{(Green)} hold true for some realizations. We discuss the relation with existing work more extensively in Section \ref{s:other}.

Let us state the main abstract result; its concrete meaning will be explored afterwards. For $x, y\in \tilV_N$, and $\gamma\in \IC\setminus \IR$, we introduce the lifted Green function
\begin{equation}\label{e:tilg}\tilde g_N^\gamma(x, y)=\la \delta_x, (\tilH_N-\gamma)^{-1}\delta_y\ra_{\ell^2(\tilV_N)}.\end{equation}
Recall that we write $G_N$ as a quotient $\Gamma_N\backslash \tilG_N$ where $\tilG_N$ is a tree. We denote by $\cD_N$ a fundamental domain of the action of $\Gamma_N$ on the vertices of $\tilG_N$. Thus $\cD_N$ contains $N$ vertices of $\tilG_N$, each of them projecting to a distinct vertex of $G_N$.

Let $I_1$ be the open set of Assumption \textbf{(Green)}, and let us fix an interval $I$ (or finite union of intervals) such that $\bar I\subset I_1$.

\begin{thm} \label{thm:2}
Assume that the graphs $G_N$ and the potentials $W_N$ satisfy \textbf{\emph{(BSCT)}}, \textbf{\emph{(EXP)}} and \textbf{\emph{(Green)}}.

Call $(\lambda^{(N)}_j)_{j=1}^N$ the eigenvalues of the Schr\"odinger operator $H_N$ on $\ell^2(V_N)$, and let $(\psi^{(N)}_j)_{j=1}^N$ be a corresponding orthonormal eigenbasis.

For each $N$, let $a=a_N$ be a function on $V_N$ with $\sup_N\sup_{x\in V_N}|a_N(x)| \leq 1$. For $\gamma\in \C\setminus \IR$, define $ \la a\ra_{\gamma}=  \sum_{x\in V_N }a(x)\Phi_\gamma^N(\tilde{x},\tilde{x})$, where $\Phi_{\gamma}^N(\tilde{x},\tilde{x}) = \frac{\Im {\tilde g^{\gamma}_N}(\tilde x, \tilde x)}{\sum_{\tilde x\in \cD_N} \Im {\tilde g^{\gamma}_N}(\tilde x,\tilde x)}$. Then
\[
\lim_{\eta_0\downarrow 0}\lim_{N\to +\infty }\frac1{N}  \sum_{\lambda^{(N)}_j\in I} \left| \sum_{x\in V_N} a(x)|\psi^{(N)}_j(x)|^2 - \la a \ra_{\lambda^{(N)}_j+i\eta_0}
\right| =0 \, .
\]
\end{thm}
Here, $\tilde{x}$ is a lift of $x\in V_N$ in the universal cover $\widetilde{V}_N$.

\begin{cor}\label{c:firstcor}
Under the same assumptions, for any $\eps>0$, we have
\[
\frac1N \#\left\{\lambda^{(N)}_j\in I : \left| \sum_{x\in V_N} a(x)|\psi^{(N)}_j(x)|^2 - \la a \ra_{\lambda^{(N)}_j+i\eta_0}
\right| >\eps \right\} \Lim_{N\to +\infty,\, \eta_0\downarrow 0} 0 \, .
\]
\end{cor}

More generally, we have the following result on eigenfunction correlators, which says that $\overline{\psi_j(x)}\psi_j(y)$ ``approaches'' the function $\Phi_{\lambda_j+i0}^N(\tilde{x},\tilde{y})$ defined in \eqref{e:Klambda}. For technical reasons we have to assume the $(\psi_j)$ are real-valued. More precisely, we need $\overline{\psi_j(x)}\psi_j(y)$ to be real for any $j=1,\dots,N$ and $x,y\in V_N$ with $x\sim y$.

\begin{thm} \label{thm:1} 
Assume that $(G_N, W_N)$ satisfies \textbf{\emph{(BSCT)}}, \textbf{\emph{(EXP)}} and \textbf{\emph{(Green)}}.

Call $(\lambda^{(N)}_j)_{j=1}^N$ the eigenvalues of $H_N$ on $\ell^2(V_N)$, and let $(\psi^{(N)}_j)_{j=1}^N$ be a corresponding orthonormal eigenbasis. Assume the $(\psi_j)_{j=1}^N$ are real-valued.

Fix $R\in \N$. For each $N$, let ${\mathbf K}={\mathbf K}_N$ be an operator on $\ell^2(V_N)$ whose kernel $K=K_N : V_N \times V_N\To \C$ is such that $K(x, y)=0$ for $d(x, y)>R$ (in other words $K$ is supported at distance $\leq R$ from the diagonal). Assume that $\sup_N \sup_{x, y\in V_N}|K_N(x, y)| \leq 1.$

For $\gamma\in \C\setminus \IR$, define 
\begin{equation}\label{e:Klambda}
\la {\mathbf K}\ra_{\gamma}=  \sum_{\tilde x\in \cD_N, \tilde y\in \tilV_N}K(\tilde x, \tilde y) \Phi_{\gamma}^N(\tilde{x},\tilde{y}) \quad \text{where} \quad \Phi_{\gamma}^N(\tilde{x},\tilde{y}) = \frac{\Im {\tilde g^{\gamma}_N}(\tilde x, \tilde y)}{\sum_{\tilde x\in \cD_N} \Im {\tilde g^{\gamma}_N}(\tilde x,\tilde x)} \, .
\end{equation}
Then
\[
\lim_{\eta_0\downarrow 0}\lim_{N\to +\infty }\frac1{N} \sum_{\lambda^{(N)}_j\in I} \left| \la \psi^{(N)}_j, {\mathbf K}\psi^{(N)}_j\ra_{\ell^2(V_N)} - \la {\mathbf K}\ra_{\lambda^{(N)}_j+i\eta_0}
\right| =0 \, .
\]
\end{thm}
The ``kernel'' above is the matrix of ${\mathbf K}$ in the basis $(\delta_x)$, i.e. $K(x, y)=\la \delta_x, {\mathbf K}\delta_y\ra_{\ell^2(V_N)}$. To define \eqref{e:Klambda} properly, we lift $K$ to $\tilV_N \times \tilV_N$ by letting
\begin{equation}
K(\tilde x, \tilde y)=K(x, y)\bbbone_{dist_{\tilG_N}(\tilde x,  \tilde y)\leq R}\label{e:lift}
\end{equation}
if $x, y\in V_N=\Gamma_N\backslash \tilV_N$ are the projections of $\tilde x, \tilde y\in\tilV_N$.

If we know in addition that $\rho(\partial I_1)=0$, where $\rho$ is the integrated density of states measure \eqref{e:IDS1}, then our main theorems hold with $I$ replaced by $I_1$; see the end of Section~\ref{sec:retour}. Note that if \textbf{(Green)} holds on $\overline{I_1}$, then $\rho(\partial I_1)=0$.

Although we tend to skip it from the notation, the ``observables'' $\mathbf{K}$ and $a$ necessarily depend on $N$. On the other hand, they do not depend on $j$, the index of the eigenfunction (they are actually allowed to depend on $\lambda^{(N)}_j$ in the proof, but this dependence cannot be wild, it has to be at least continuous).
We interpret Corollary \ref{c:firstcor} as follows~: for a given observable $a$, the average $\sum_{x\in V_N} a(x)|\psi^{(N)}_j(x)|^2$ is close to $\la a \ra_{\lambda^{(N)}_j+i\eta_0}$ for most indices $j$. It follows similarly from Theorem \ref{thm:1} that
 $\sum_{x,y\in V_N}K(x, y)\overline{\psi^{(N)}_j(x)} \psi^{(N)}_j(y)$ is close to $\la {\mathbf K}\ra_{\lambda^{(N)}_j+i\eta_0}$ for most $j$. One of the subtleties of the result is that the indices $j$
 for which this holds may {\em{a priori}} depend on the observables $a$, $\mathbf{K}$. If we wanted to have a common set of indices $j$ that do the job for all observables (whose number is exponential in $N$), we would need to have an exponential rate of convergence in Theorems \ref{thm:2} and \ref{thm:1}. As is seen in the case of regular graphs and $W=0$ \cite{A}, our proof gives a rate that is at best a negative power of the girth, which is itself typically of order $\log N$. So, the result is far from showing that $|\psi^{(N)}_j(x)|^2$ is close to the uniform measure in total variation.
 
Note the presence of the extra parameter $\eta_0$, in comparison with the case of regular graphs \cite{ALM, A}. This is due to the fact that, generally speaking, the quantities $\la a \ra_{\lambda^{(N)}_j+i\eta_0}$ and $\la {\mathbf K}\ra_{\lambda^{(N)}_j+i\eta_0}$ are not necessarily bounded as $\eta_0\downarrow 0$ for fixed $N$. They will however stay bounded in the limits $N\to +\infty$ followed by $\eta_0\downarrow 0$ (as a result of \eqref{e:cool3} and \textbf{(Green)}).  
  
\subsection{Understanding the weighted averages.}          \label{sec:weightedav}
In order to clarify the relevance of Theorems \ref{thm:2} and \ref{thm:1}, we now investigate the meaning of the quantities $\la a \ra_{\lambda+i\eta_0}$ and $\langle \mathbf{K}\rangle_{\lambda_j+i\eta_0} $. Let us start with Theorem~\ref{thm:2}. A good illustration is to choose $a_N=\bbbone_{\Lambda_N}$, the characteristic function of a set $\Lambda_N\subset V_N$ of size $\approx \alpha N$ for some $\alpha \in (0,1)$, say $\alpha=\frac{1}{2}$.

In the special case where $(G_N)$ is regular and $H_N=\mathcal{A}_N$, and also for the anisotropic model treated in \cite{A}, the Green function $\tilde g^{\gamma}_N(\tilde{x},\tilde{y})$ does not depend on $N$, as it coincides with the limiting Green function $\mathcal{G}^{\gamma}(\tilde{x},\tilde{y})$. Moreover, $\mathcal{G}^{\gamma}(\tilde{x},\tilde{x})=\mathcal{G}^{\gamma}(o,o)$ for all $\tilde{x}\in \mathcal{D}_N$. It follows that $\langle \bbbone_{\Lambda_N} \rangle_{\lambda_j +i\eta_0}  = \sum_{x\in \Lambda_N} \frac{\mathcal{G}^{\lambda_j+i\eta_0}(o,o)}{N \mathcal{G}^{\lambda_j+i\eta_0}(o,o)}=\alpha$. So Corollary~\ref{c:firstcor} implies that $\| \bbbone_{\Lambda_N} \psi_j^{(N)}\|^2 \approx \alpha$ for most $\psi_j^{(N)}$. This shows that most $\psi_j^{(N)}$ are uniformly distributed, in the sense that if we consider any $\Lambda_N\subset V_N$ containing half the vertices, we find half the mass of $\|\psi_j^{(N)}\|^2$. As we show in the next subsection, such interpretation is also valid for the Anderson model.

For general models, we cannot assert that $\langle \bbbone_{\Lambda_N} \ra_{\lambda+i\eta_0} = \alpha$. Still, we prove in Section~\ref{sec:bsctaux} that there exists $c_{\alpha}>0$ such that for any $\Lambda_N\subset V_N$ with $|\Lambda_N| \geq \alpha N$, we have
\begin{equation}       \label{eq:weightedav}
\inf_{\eta_0\in(0,1)}\liminf_{N\To\infty} \inf_{\lambda\in I_1} \la \bbbone_{\Lambda_N} \ra_{\lambda+i\eta_0}\geq 2c_{\alpha} \, .
\end{equation}
Combined with Corollary \ref{c:firstcor}, this implies

\begin{cor}\label{c:secondcor}
For any $\alpha \in (0,1)$, there exists $c_{\alpha}>0$ such that for any $\Lambda_N\subset V_N$ with $|\Lambda_N| \geq \alpha N$, we have
\[
\frac1N \#\left\{\lambda^{(N)}_j\in I : \left\| \bbbone_{\Lambda_N} \psi_j^{(N)}\right\|^2  < c_{\alpha} \right\} \Lim_{N\To +\infty } 0 \, .
\]
\end{cor}
Hence, while in the simple case we had $\|\bbbone_{\Lambda_N}\psi_j^{(N)}\|^2 \approx \alpha$ for most $\psi_j^{(N)}$, in the general case, we can still assert that $\|\bbbone_{\Lambda_N}\psi_j^{(N)}\|^2 \ge c_{\alpha}>0$ for most $\psi_j^{(N)}$. This indicates that our theorem can truly be interpreted as a delocalization theorem. The bad indices $j$ (for which $\| \bbbone_{\Lambda_N} \psi_j^{(N)}\|^2  <c_{\alpha}$) will a priori depend on $\Lambda_N$.
 
 We now turn to the general averages $\langle \mathbf{K}\rangle_{\gamma_j} $. Recall that $\Phi_{\gamma}^N(\tilde{x},\tilde{y}) = \frac{\Im {\tilde g^{\gamma}_N}(\tilde x, \tilde y)}{\sum_{\tilde x\in \cD_N} \Im {\tilde g^{\gamma}_N}(\tilde x,\tilde x)}$. We will show in Section~\ref{sec:bsctaux} that under assumption \textbf{(BSCT)}, we have
\begin{equation}        \label{e:conv2}
\frac1N \sum_{x\in V_N} \Im  {\tilde g^{\lambda+i\eta_0}_N}(x, x) \Lim_{N\To +\infty} \IE\left(\Im {\cG^{\lambda+i\eta_0}}(o, o)\right)
\end{equation}
uniformly in $\lambda\in I_0$. This already shows that $\Phi_{\gamma}^N(\tilde{x},\tilde{y})$ is of order $1/N$, since the denominator in its expression is of order $N$. We strengthen this observation by proving that for any continuous $F:\IR \to \IR$,
%and $\chi:\mathscr{G}_{\ast}^{D,A}\to \R$,
we have uniformly in $\lambda\in I_0$,
\begin{equation}        \label{eq:joint2}
\frac1N \sum_{x\in V_N}  \sum_{y, d(y, x)=k} F\left(N\Phi_{\lambda+i\eta_0}^N(\tilde{x},\tilde{y})\right)  \Lim_{N\To +\infty} \IE\left(   \sum_{v, d(v, o)=k} F\left(\frac{{\Im\cG^{\lambda+i\eta_0}}(o, v)}{\IE\left({\Im \cG^{\lambda+i\eta_0}}(o, o)\right)}\right)\right) \, .
\end{equation}
This says that the empirical distribution of $\left(N\Phi_{\gamma}^N(\tilde{x},\tilde{y})\right)$ (when $x$ is chosen uniformly at random in $V_N$ and $y$ is then chosen uniformly among the points at distance $k$ from $x$) converges to the law of $\left( \frac{{\Im\cG^{\gamma}}(o, v)}{\IE\left({\Im \cG^{\gamma}}(o, o)\right)}\right)$ ($v$ being chosen uniformly among the points at distance $k$ from the root $o$). This is a second way of saying that $\Phi_{\gamma}^N(\tilde{x},\tilde{y})$ is of order $1/N$~: when multiplied by $N$, it has a non-trivial limiting distribution.
  
\subsection{Case of the Anderson model\label{s:andersonthm}}
It is important to check that the models covered by the assumptions of our main theorems are not reduced to the case of the laplacian on regular graphs, already treated in \cite{ALM, BLML, A}. Here we consider the important case of the Anderson model on regular graphs, i.e. the laplacian with a random potential. We will show that, if the strength of the disorder is small enough, then the assumptions of Theorem~\ref{thm:2} and \ref{thm:1} are satisfied for almost every realization of the potential.

Let $\mathbb{T}_q$ be the $(q+1)$-regular tree.
Let $\nu$ be a probability measure on $\IR$, supported on a compact interval $[-A, A]$, and for every $\eps>0$ let $\nu_\eps$ be the image of $\nu$ under the homothety $x\mapsto \eps x$ ($\nu_\eps$ is now supported on $[-\eps A, \eps A]$).
Let $\Omega=\IR^{\mathbb{T}_q}$, and define $\mathbf{P}_\eps$ on $\Omega$ by $\mathbf{P}_\eps = \mathop \otimes_{v\in \mathbb{T}_q} \nu_\eps$. We shall denote by $\mathbf{E}_{\eps}$ the expectation with respect to $\mathbf{P}_\eps$. Given $\omega = (\omega_v) \in \Omega$, define $\cW^{\,\omega}(v) = \omega_v$ for $v\in \mathbb{T}_q$. Then the $\{\omega_v\}_{v\in \mathbb{T}_q}$ are i.i.d. random variables with common distribution $\nu_\eps$. Here $\eps\in \R$ is fixed and parametrizes the strength of the disorder.

Let $G_N=(V_N,E_N)$ be a (deterministic) sequence of $(q+1)$-regular graphs with $|V_N|=N$. This means that $\tilG_N= \mathbb{T}_q$ for all $N$. Let $\Omega_N = \IR^{V_N}$ and $\cP^\eps_N = \mathop\otimes_{x\in V_N} \nu_\eps$ on $\Omega_N$. We denote $\widetilde{\Omega} = \prod_{N\in \N} \Omega_N$ and let $\mathcal{P}_\eps$ be any probability measure on $\widetilde{\Omega}$ having $\cP^\eps_N$ as a marginal on the factor $\Omega_N$. Given $(\omega_N)_{N\in \N}\in \widetilde{\Omega}$, so that $\omega_N=(\omega_x)_{x\in V_N}\in \Omega_N$, we define $W^{\omega_N}(x) =  \omega_x$ for $x\in V_N$.

The results of this section are proved in a companion paper \cite{AS2}.

\begin{prp} \label{p:example}
Suppose $(G_N)$ satisfies \emph{\textbf{(BST)}}. Then
\emph{\textbf{(BSCT)}} holds for $\cP^\eps$-almost every realization of the potential. More precisely, for $\cP^{\eps}$-a.e. $(\omega_N)\in\widetilde{\Omega}$, the sequence $(G_N,W^{\omega_N})$ has a local weak limit $\prob_{\eps}$ which is concentrated on $\{[\mathbb{T}_q, o, \cW^{\,\omega}]:\omega\in\Omega\}$, where $o\in \mathbb{T}_q$ is fixed and arbitrary. The measure $\prob_{\eps}$ acts by taking the expectation w.r.t. $\mathbf{P}_{\eps}$, that is, if $D=q+1$, then
\[
\int_{\mathscr{G}_{\ast}^{D,\eps A}} f([G,v,W])\,\dd \prob_{\eps}([G,v,W]) = \int_{\Omega}f([\mathbb{T}_q,o,\cW^{\,\omega}])\,\dd\mathbf{P}_{\eps}(\omega) = \mathbf{E}_{\eps}[f([\mathbb{T}_q,o,\cW^{\,\omega}])]\,.
\]
\end{prp}
  
We make the following assumption on the random variables:

\medskip

\textbf{(POT)} The measure $\nu$ is H\"older continuous, i.e. there exist $C_{\nu}>0$ and $b\in (0,1]$ such that $\nu(I) \le C_{\nu} |I|^b$ for all bounded $I\subset \R$.

The following proposition is by no means trivial, it comes from the results of \cite{Klein, AW2}.
\begin{prp}  \label{p:AW}
Fix $0<\lambda_0<2\sqrt{q}$. There exists $\eps(\lambda_0)$ such that if $|\eps|<\eps(\lambda_0)$, then assumption \emph{\textbf{(Green)}} holds for the measure $\IP_{\eps}$ of Proposition~\ref{p:example} on $I_1= (-\lambda_0,\lambda_0)$.
\end{prp} 
  
\begin{cor} If the graphs $G_N$ form an expander family and satisfy \emph{\textbf{(BST)}} and if the disorder $\eps$ is small enough, the conclusions of Theorems \ref{thm:2} and \ref{thm:1} hold true for $\cP_\eps$-a.e. realization $(\omega_N)\in \widetilde{\Omega}$, with $I_1=  (-\lambda_0,\lambda_0)$.
 \end{cor}
 
This gives a rich enough family of examples where the assumptions of Theorems \ref{thm:2} and \ref{thm:1} hold true. Thus the conclusions of the theorems hold for any observables $a_N, K_N$. If in addition $a_N$ or $K_N$ are independent of the disorder, some extra averaging takes place, and we may replace $\langle\mathbf{K}\rangle_{\lambda +i\eta_0} $ by a simpler average as follows.

\begin{thm}                    \label{t:thm3}
Assume that \textbf{\emph{(POT)}}, \textbf{\emph{(EXP)}} and \textbf{\emph{(BST)}} hold. Given $(\omega_N)\in \widetilde{\Omega}$, let $(\psi_i^{\omega_N})_{i=1}^N$ be an orthonormal basis of eigenfunctions of $H_N^{\omega}=\mathcal{A}_N+W^{\omega_N}$ in $\ell^2(V_N)$, with corresponding eigenvalues $(\lambda_i^{\omega_N})_{i=1}^N$.

Let $K_N:V_N\times V_N \to \C$, $\sup_N\sup_{x,y \in V_N} |K_N(x,y)| \le 1$, $K_N(x,y)=0$ if $d(x,y)>R$, and assume $K_N$ is independent of $(\omega_N)$. Fix $0<\lambda_0<2\sqrt{q}$. If $|\eps|<\eps(\lambda_0)$, we have for $\mathcal{P}_\eps$-a.e. $(\omega_N)$,
\[
\lim_{\eta_0\downarrow 0} \lim_{N \to \infty} \frac{1}{N} \sum_{\lambda_i^{\omega_N}\in (-\lambda_0,\lambda_0)} \big|\langle \psi_i^{\omega_N},K_N \psi_i^{\omega_N} \rangle - \langle K_N \rangle_{\lambda_i^{\omega_N}}^{\eta_0} \big|  = 0 \, ,
\]
where for $\gamma\in \IC\setminus \IR$
\begin{equation}\label{e:Klambda2}
\langle K \rangle_{\lambda}^{\eta_0} = \sum_{x,y \in V_N} K(\tilde{x},\tilde{y}) \widetilde{\Phi}_{\gamma}(\tilde{x},\tilde{y})  \quad \text{and} \quad \widetilde{\Phi}_{\gamma}(\tilde{x},\tilde{y}) = \frac{1}{N} \cdot \frac{\mathbf{E}_{\eps}[\Im \mathcal{G}^{\gamma}(\tilde{x},\tilde{y})]}{\mathbf{E}_{\eps}[\Im \mathcal{G}^{\gamma}(o,o)]} \, .
\end{equation}
\end{thm}

As in the previous theorems, if $R=0$, the $\psi_j$ are arbitrary, while if $R>0$, we assume the $\psi_j$ are real-valued.
 
For the Anderson model, $\mathbf{E}_\eps\left( \Im \cG^{\gamma}(v,w) \right)$ depends only on $d(v,w)$ : $\mathbf{E}_\eps\left( \Im \cG^{\gamma}(v,w) \right)=\mathbf{E}_\eps\left( \Im \cG^{\gamma}(o, u) \right)$ where $u$ is any vertex of $\IT_q$ such that $d(o, u)=d(v,w)$.

In the special case $R=0$, we have $\langle a_N \rangle_{\lambda}^{\eta_0} = \frac{1}{N} \sum_{x\in V_N} a(x)$. So choosing $a_N=\bbbone_{\Lambda_N}$, Theorem~\ref{t:thm3} implies the strong form of delocalization given by the uniform distribution of $\psi_j^{(N)}$ on $V_N$, as explained in Section~\ref{sec:weightedav}.
            
\subsection{Relation with previous work\label{s:other}}
%First, let us note that, like in our previous papers \cite{ALM, A}, the term ``quantum ergodicity'' is borrowed from the field of quantum chaos and

Our main Theorem \ref{thm:1} holds for {\em{deterministic}} sequences of graphs and potentials. For any sequence $(G_N, W_N)$ satisfying the assumptions of the theorem, the conclusion holds for any observable $K$; in particular, $K$ may depend on the graphs. As already noted, the result only says something about the delocalization of ``most'' eigenfunctions, where the ``good'' eigenfunctions exhibiting delocalization may depend on the choice of the observable $K$.

In the past years, there has been tremendous interest in spectral statistics and delocalization of eigenfunctions of {\em{random}} sequences of graphs and potentials. Many papers consider {\em{random}} regular graphs, with degree going slowly to infinity \cite{TranVu, Dumitriu, BKY, BHKY} or fixed \cite{Geisinger, BHY}, sometimes adding a random i.i.d potential \cite{Geisinger}. In particular, the recent papers \cite{BKY, BHKY, BHY} show ``quantum unique ergodicity'' for the \emph{adjacency matrix} of random regular graphs~: given an observable $a_N : \{1, \ldots, N\}\To \IR$, for most $(q+1)$-regular graphs on the vertices $\{1, \ldots, N\}$ we have that
$\sum_{x=1}^N a_N(x)|\psi_j^{(N)}(x)|^2$ is close to $\la a_N\ra$ for {\em{all}} indices $j$. This is a considerable strengthening of Corollary \ref{c:firstcor} (or of the similar result in \cite{ALM}), that only says something for {\em{most}} indices $j$. This possibility to prove QUE is, of course, due to the fact that $a_N$ has to be independent of the choice of the graph and that results holds for almost all graphs.
%It might well be that a positive proportion of graphs contradicts QUE, if we were allowed to choose observables $a_N$ depending on the graph (this is a completely open question).

When ``ergodicity'' of eigenfunctions is tested numerically as in the physics papers \cite{Alt1, Alt2}, it is natural to first pick a realization of the graph and of the potential, and then test the eigenfunctions one by one to determine if they can be localized in small parts of the graph. It is then natural to allow the test-observables to depend on the graph and the potential (which our Theorem \ref{thm:1} does, but not the results of \cite{BKY, BHY}), but {\em{also}} on the index $j$ of the eigenfunction, which neither of the rigourous mathematical results achieves. The numerical results of \cite{Alt2} seem to indicate that, as soon as a random disorder is turned on, the eigenfunctions will be localized in small parts of the graph. This is not in contradiction with our results~: the region of localization of $\psi_j^{(N)}$ might depend on $j$, but our result does not allow to test this. Note also that the results of \cite{Alt1, Alt2} were recently questioned in \cite{TMS}, where the authors argue that $N$ has not been taken large enough to see the delocalization take place.
 
The paper \cite{BaSz16} proves a very important result, saying that if $\psi_j$ is an ``almost eigenvector'' of the adjacency matrix on a random regular graph $G$, then for almost all $G$ and all $j$, the value distribution of $\psi_j(x)$ as $x$ runs over $\{1, \ldots, N\}$ is close to a Gaussian $\cN(0, \sigma_j^2)$ with $\sigma_j\leq 1$. Proving that $\sigma_j=1$ is a challenge; it would amount to proving that eigenfunctions cannot be localized in small parts of the graph. Our result does not say this, again because we can only test one observable $a$ at a time.
The indices $j$ for which Corollary \ref{c:firstcor} proves delocalization depend on $a$. If we wanted to have a common set of indices $j$ that do the job for all observables (whose number is exponential in $N$), we would need to have an exponential rate of convergence in Theorems \ref{thm:2}, \ref{thm:1}. Our proof gives a rate that is at best a negative power of the girth (itself typically of order $\log N$).
 
 Finally we would also like to mention the paper \cite{Bor15}, where existence of absolutely continuous spectrum for percolation graphs on the $(q+1)$-regular tree is proven, if the percolation parameter is close enough to $1$. Since the absolutely continuous spectrum is mixed with purely discrete spectrum, one cannot expect a quantum ergodicity result that claims delocalization of most eigenfunctions, but only a ``partial delocalization'' result for a {\em{positive proportion}} of eigenfunctions. These are the contents of \cite[Theorem 9]{Bor15}. It would be nice to investigate what the methods of our paper would give for that model.
 
\subsection{Outline of the proof\label{s:outline}}
We borrowed the name ``Quantum Ergodicity'' from a result about laplacian eigenfunctions on Riemannian manifolds \cite{Sni, Zel87, CdV85, ZelC}.
The proof in the setting of laplacian eigenfunctions on manifolds is made of 4 steps, of unequal difficulty . These 4 steps are also present in our proof~:

{\bf{Step 0.}} Define the quantum variance. The goal is to show that this goes to $0$ as $N\to \infty$. A novelty of our proof is that we replace the usual quantum variance \eqref{e:usualvar} by a ``non-backtracking'' one \eqref{e:varnb}, where we replace the eigenfunctions $\psi_j$ by eigenfunctions $f_j, f_j^*$ of a non-backtracking random walk (Section \ref{s:nb}). These new $f_j, f_j^*$ are thus eigenfunctions of a non-selfadjoint problem. This causes new difficulties, that however will be compensated by the fact that the non-backtracking random walk has simpler trajectories than the ``simple'' random walk generated by the adjacency matrix $\cA$.

{\bf{Step 1.}} Show that the quantum variance is controlled by the Hilbert-Schmidt norm of $K$. Although this is obvious for the original quantum variance, this will be much harder for the ``non-backtracking quantum variance'' (Section \ref{s:proof1}). This uses \textbf{(BSCT)} and \textbf{(Green)}.

{\bf{Step 2.}} Due to the fact that $f_j, f_j^*$ satisfy an eigenfunction problem, the quantum variance is invariant under certain transformations (Section \ref{s:inv}).

{\bf{Step 3.}} One should see behind these transformations the emergence of a ``classical dynamical system''. In the setting of laplacian eigenfunctions on manifolds, this is the geodesic flow. Here, what we get is a family of stationary Markov chains on the set of infinite non-backtracking paths (Section \ref{s:Mark}, Remark \ref{r:compat}). This step has been called ``classicalization'' by U. Smilansky in a private conversation; this is supposed to mean the opposite of ``quantization''.

{\bf{Step 4.}} Iterate the classical dynamical system, use its ergodicity to show that the quantum variance is small (Section \ref{s:step4}). Here, the ergodicity of our Markov chains (more precisely, the fact that the mixing rate is independent on $N$) comes from the \textbf{(EXP)} condition. Assumption \textbf{(Green)} is also used to control the probability transitions.

There is an additional step that does not exist in the traditional setting~:

{\bf{Step 5.}} Translate the result for the ``non-backtracking quantum variance'' (involving $f_j, f_j^*$) into a result for the original one, involving the $\psi_j$ (Section \ref{sec:retour}).
Assumptions \textbf{(EXP)}, \textbf{(BSCT)} and \textbf{(Green)} are used here again to show that the transformation sending $\psi_j$ to $f_j, f_j^*$ is well-behaved in the limit $N\To +\infty$.

\section{Basic identities}

\subsection{``Quantization procedure'' on trees and their quotients}                        \label{sec:basic}
Let $G=G_N$, $G=(V,E)$. Most of the time we will drop the subscript $N$ in the notation. As in Section~\ref{sec:results}, we regard $G$ as a quotient: $G=\Gamma \backslash \widetilde{G}$, and let $\pi:\tilV\to V$ denote the projection. Fix a fundamental domain $\mathcal{D} \subset \tilV$ for the action of $\Gamma$ on $\tilV$. Then $|\mathcal{D}| = |V|$.

Each edge $\{x_0,x_1\}\in \tilE$, gives rise to two oriented edges $e = (x_0,x_1)$ and $\hat{e} = (x_1,x_0)$ in the reverse direction. We let $o_e$ and $t_e$ be the origin and terminus of $e$, respectively. We then let $\tilB_1$, or simply $\tilB$, be the set of all such oriented edges of $\tilG$. More generally, let $\tilB_k$ be the set of non-backtracking paths of length $k$ in $\tilG$. By convention, $\tilB_0 :=\tilV$. If $\omega = (x_0,\ldots,x_k)$ and $\omega' = (x_0',\ldots,x_k') \in \tilB_k$, we write $\omega \rightsquigarrow \omega'$ if $x_0'=x_1,\ldots,x_{k-1}'=x_k$ and $(x_0,\ldots,x_k,x_k') \in \tilB_{k+1}$. We also denote $o_{\omega}=x_0$, $t_{\omega}=x_k$.

These notions descend to the quotient. We denote by $B_k := \Gamma\backslash \widetilde{B}_k$ the set of non-backtracking paths of length $k$ in $G$. By convention, $B_0 :=V$. For $k=1$ we let $B=B_1$. 
The set $B_k$ is in bijection with the subset $\cD^{(k)}\subset \tilB_k$ of elements having their origin in $\cD$.

%We introduce the space $\mathscr{H}$ of matrices $K:\tilV\times \tilV \to \C$, that satisfy the condition
%\[
%K(\gamma v,\gamma w) = K(v,w) \quad \text{for all } \gamma \in \Gamma   \, .
%\]
 
%The space $\mathscr{H}$ may be decomposed into the direct sum $\mathop \oplus_{k=0}^{\infty} \mathscr{H}_k$, where $\mathscr{H}_k \subset \mathscr{H}$ is the subspace defined by
%\[
%K(v,w) \neq 0 \implies dist_{\tilG}(v,w) = k \, .
%\]
%An element $K\in \mathscr{H}_k$ is completely determined by the values of $K(v, w)$ with $v\in \cD$.
%Since $K(v,w) = 0$ if $dist(v,w)>k$, the space $\mathscr{H}_k$ is finite-dimensional.

Let $\mathscr{H}_k =\IC^{B_k}$ (the complex-valued functions on $B_k$),  $\mathscr{H} = \mathop \oplus_{k=0}^{\infty} \mathscr{H}_k$ and $\mathscr{H}_{\leq k}:=\mathop \oplus_{\ell=0}^{k} \mathscr{H}_\ell$. 

It will be convenient to identify $\IC^{B_k}$ with the $\Gamma$-invariant elements of $\IC^{\tilB_k}$ or with $\IC^{\cD^{(k)}}$. For $K\in \mathscr{H}_k$ and $(x_0, \ldots, x_k)\in\tilB_k$, we will sometimes use the short-hand notation $K(x_0; x_k)$ for $K(x_0, \ldots, x_k)$. This is justified by the fact than on $\tilG$, the endpoints $(x_0; x_k)$ determine the path $(x_0, \ldots, x_k)$ uniquely. We will also use this short-hand notation on $B_k$, although in that case one should keep in mind that $K(x_0; x_k)$ actually depends on the full path $(x_0, \ldots, x_k)$.

Any $K \in  \mathscr{H}_k$ (regarded as a $\Gamma$-invariant element of $\C^{\widetilde{B}_k}$) may be used to define an operator $\widehat{K}$ on the space of finitely supported functions on $\tilV$, with kernel
$\la \delta_v, \widehat{K}\delta_w\ra_{\ell^2(\tilV)}=K(v; w)$. It also defines an operator $\widehat{K}_G$ on $\IC^V$, with kernel\[ %\label{e:kg0}
K_G(x,y)=\sum_{\gamma\in\Gamma}K (\tilde x; \gamma\cdot \tilde y) \, ,
\]
where $\tilde x, \tilde y\in \tilV$ are representatives of $x, y\in V$. The map $K\in \mathscr{H}_k\mapsto K_G$ is a priori not one-to-one. However, if $\rho_G(x) \ge k$, then $K_G(x, \cdot)$ determines $K(\tilde x, \cdot)$ uniquely. To see that $K\in \mathscr{H}_k\mapsto K_G$ is surjective, consider $\mathbf{k}: V\times V\To \IR$ supported at distance $k$ from the diagonal, and let $K(\tilde x, \tilde y)= \mathbf{k}(\pi(\tilde x), \pi(\tilde y)) \bbbone_{dist(\tilde x, \tilde y) \leq k} (\sharp\{\gamma \in \Gamma, dist(\tilde x, \gamma\cdot \tilde y) \leq k\})^{-1}$. Then $K_G=\mathbf{k}$ and this coincides with the lift \eqref{e:lift} except at the few points where $\rho_G(x)\leq k$.

Define the non-backtracking adjacency operator $\cB:\C^\tilB\to \C^\tilB$ by
\begin{equation}    \label{eq:nonback}
(\cB f)(x_0,x_1) = \sum_{x_2 \in \mathcal{N}_{x_1}\setminus \{x_0\}} f(x_1,x_2) \, 
\end{equation}
where $\mathcal{N}_{x}$ means the set of neighbours of $x$.
Then an element $K\in \mathscr{H}_k$ may also be used to define an operator $\widehat{K}_{\tilB}$ on $\ell^2(\tilB)$, with kernel
\[%\label{e:defK}
\la \delta_{b_1}, \widehat{K}_{\tilB}\delta_{b_2}\ra_{\ell^2(\tilB)} = \begin{cases} K(o_{b_1}; t_{b_2})&\text{if } \cB^{k-1}(b_1, b_2)\neq 0, \\
0 & \text{otherwise.} \end{cases}
\]
Thus $\la \delta_{b_1}, \widehat{K}_{\tilB}\delta_{b_2}\ra_{\ell^2(\tilB)}\not= 0$ only if there is a non-backtracking path of length $k$ in $\tilG$, starting with the oriented edge $b_1$ and ending with $b_2$.

Finally, $K\in \mathscr{H}_k$ also defines an operator $\widehat{K}_B$ on $\IC^B$, with matrix $K_B: B \times B \to \C$ given by
\[%\label{e:defKB}
K_B(b_1,b_2)=\sum_{\gamma\in\Gamma}K({\tilde b_1}; \gamma\cdot  {\tilde b_2}) \, ,
\]
where $\tilde b_1,  \tilde b_2 \in \tilB$ are lifts of $b_1, b_2\in B$.
By linearity, this extends to $K\in \mathscr{H}_{\leq k}$.

%The matrices $K_G$ and $K_B$ naturally define operators on $\ell^2(V)$ and $\ell^2(B)$, respectively.

Note that if $K\in \mathscr{H}_k$, then $\langle \psi,K_G\phi \rangle_{\ell^2(V)} = \sum_{(x_0,\ldots, x_k)\in B_k} \overline{\psi(x_0)} K(x_0;x_k) \phi(x_k)$ for any $\psi, \phi\in \ell^2(V)$. Similarly, if $f, g\in \ell^2(B)$, we have
\begin{equation}   \label{eq:kb}
\langle f, K_B g \rangle_{\ell^2(B)} = \sum_{(x_0,\ldots, x_k)\in B_k} \overline{f(x_0,x_1)} K(x_0;x_k) g(x_{k-1},x_k) \, ,
\end{equation}
\begin{equation}          \label{eq:KB2}
\|K_Bf\|^2_{\ell^2(B)} = \sum_{(x_0,x_1)\in B} \Big|\sum_{_{x_{0,1}}(x_2;x_k)} K(x_0;x_k) f(x_{k-1},x_k)\Big|^2 \, ,
\end{equation}
where $\sum_{_{x_{0,1}} (x_2;x_k)}$ sums over all $(x_2;x_k)\in B_{k-2}$ such that $x_2\in \mathcal{N}_{x_1}\setminus \{x_0\}$. Alternatively, we may simply sum over $(x_2;x_k)\in B_{k-2}$ but decide that $K(x_0;x_k)=0$ if the path $(x_0,\ldots, x_k)$ back-tracks. %Similarly, we sometimes denote $\sum_{(x_{-k+1};x_{-1})_{x_{0,1}}}$ for the sum over all $(x_{-k+1};x_{-1})\in B_{k-2}$ such that $x_{-1}\in \cN_{x_0}\setminus \{x_1\}$.

\begin{rem}The maps $K\mapsto \widehat{K}$, $K\mapsto \widehat{K}_G$, $K\mapsto\widehat{K}_{\tilB}$ and $K\mapsto\widehat{K}_B$ associate an operator to a function on the set of paths. It is tempting to view this as a form of ``quantization procedure'' as those used for quantum ergodicity on manifolds.
\end{rem}

\subsection{Green functions on trees\label{s:ident}}

Assumption \textbf{(BST)} says that our graphs have few short loops, in other words, that most balls of a given radius look like {\em{trees}}. One of the ingredients of our proof is that the Green function
on trees satisfies certain algebraic relations, that follow from the fact that removing a vertex (or cutting an edge) from a tree suffices to disconnect it.

Here we recall some standard facts that hold for an arbitrary {\em{tree}} $T=(V(T), E(T))$, endowed with a discrete Schr\"odinger of the form $H=\cA+W$ acting on $\ell^2(V(T))$, where $\cA$ is the adjacency matrix and $W: V(T)\To \IR$ is a bounded function. Given $\gamma \in \C \setminus \R$ and $v,w\in T$, the Green function is denoted in this section by
\[
G(v,w;\gamma)=\la \delta_v, (H-\gamma)^{-1}\delta_w\ra_{\ell^2(V(T))} \, .
\]

If $v \sim w$, we denote by ${T}^{(v|w)}$ the tree obtained by removing from ${T}$ the branch emanating from $v$ that passes through $w$. We define the restriction $H^{(v|w)}(u,u') = H(u,u')$ if $u,u'\in {T}^{(v|w)}$ and zero otherwise. The corresponding Green function is denoted by  $\tilg^{(v|w)}(\cdot, \cdot;\gamma)$. We finally denote
\[
\frac{-1}{2m_v^{\gamma}}=G(v,v;\gamma)  \quad \text{and} \quad \zeta_w^{\gamma}(v) = -\tilde{g}^{(v|w)}(v,v;\gamma) \, .
\]

Later on, we will apply these results for $(T, W)=(\tilG_N, \tilW_N)$. In this case the (full) Green function will be denoted by $\tilde g_N^{\gamma}(x,y)$, and the restricted one by $\zeta_x^{\gamma}(y)$. In the case $(T, W)=(\cT, \cW)$ (the random coloured rooted trees of assumption \textbf{(BSCT)}), the Green function will be denoted by $\cG^{\gamma}(v,w)$, and the restricted one by $\hat{\zeta}_w^{\gamma}(v)$. As a general rule, the objects defined on the limit $(\cT, \cW)$ will wear a hat $\hat{\cdot}$ to distinguish them from similar objects defined on $(\tilG_N, \tilW_N)$ (see also Remark \ref{rem:IDS1}).

The Green functions on trees satisfy some classical recursive relations; the following lemma is proved for instance in \cite{AS}. Given $v\in V(T)$, we denote by $\mathcal{N}_v$ its set of nearest neighbours.

\begin{lem}                     \label{lem:zetapot}
For any $v \in {T}$ and $\gamma  \in \C\setminus \IR$, we have
\begin{subequations}
\begin{align}
\gamma &= W(v) + \sum_{u \sim v} \zeta_v^{\gamma}(u)+2m^{\gamma}_v \,, \label{eq:green3}\\
\gamma &= W(v) + \sum_{u \in \mathcal{N}_v \setminus \{w\}} \zeta_v^{\gamma}(u) + \frac{1}{\zeta_w^{\gamma}(v)} \, .\label{eq:green3'}
\end{align}
\end{subequations}
For any non-backtracking path $(v_0;v_k)$ in $T$,
\begin{equation}             \label{eq:multigreen2}
G(v_0,v_k;\gamma) = \frac{-\prod_{j=0}^{k-1} \zeta_{v_{j+1}}^{\gamma}(v_j)}{2m^{\gamma}_{v_k}} \, ,
\end{equation}
\begin{equation}          \label{eq:multigreen3}
G(v_0,v_k;\gamma) = \zeta_{v_1}^{\gamma}(v_0) G(v_1,v_k;\gamma) = \zeta_{v_{k-1}}^{\gamma}(v_k) G(v_0,v_{k-1};\gamma) \, .
\end{equation}
Also, for any $w\sim v$, we have
\begin{equation}           \label{eq:mv}
\zeta_w^{\gamma}(v) = \frac{m_w^{\gamma}}{m_v^{\gamma}} \,\zeta_v^{\gamma}(w) \quad \text{and} \quad 
\frac{1}{\zeta_w^{\gamma}(v)} - \zeta_v^{\gamma}(w) = 2m^{\gamma}_v \, .
\end{equation}
For any $v,w\in T$, we have
\begin{equation}              \label{eq:greensym}
G(v,w;\gamma) = G(w,v;\gamma) \, .
\end{equation}
Next, if $\gamma = \lambda\pm i\eta$ with $\lambda\in\R$, $\eta>0$, then
\begin{equation}             \label{eq:sumzeta}
\sum_{u\in \mathcal{N}_v\setminus \{w\}} |\Im \zeta_v^{\gamma}(u)| = \frac{|\Im \zeta_w^{\gamma}(v)|}{|\zeta_w^{\gamma}(v)|^2} - \eta \, .
\end{equation}
Finally, if $\Psi_{\gamma,v}(w) = \Im G(v,w;\gamma)$, then for any path $(v_0,\dots,v_k)$ in $ {T}$, $k \ge 1$,
\begin{equation}              \label{eq:idpsi}
 \Psi_{\gamma,v_0}(v_k) - \zeta_{v_{k-1}}^{\gamma}(v_k)\Psi_{\gamma,v_0}(v_{k-1}) = \Im \zeta_{v_{k-1}}^{\gamma}(v_k) \cdot \overline{G(v_0,v_{k-1};\gamma)} \, .
\end{equation}
\end{lem}

Note that $|\zeta_v^{\lambda+i\eta}(u)|\le \eta^{-1}$. It follows from \eqref{eq:green3'} that for any $\lambda\in [-(A+D),A+D]$ and $\eta\in (0,1)$,
\begin{equation}\label{e:green3cons}
\left|\frac{1}{\zeta_w^{\lambda+i\eta}(v)}\right| \le c_{D,A} \eta^{-1}\,,
\end{equation}
where $c_{D,A} = 2(A+D)+1$.

\begin{cor}          \label{cor:psiiden}
Given $\gamma\in \C\setminus\IR$, for any $v_0,v_1\in T$, $v_0\sim v_1$, we have
\begin{equation}     \label{eq:psiiden1}
\Psi_{\gamma,v_1}(v_1)-\overline{\zeta_{v_0}^{\gamma}(v_1)}\Psi_{\gamma,v_1}(v_0) - \zeta_{v_0}^{\gamma}(v_1)\Psi_{\gamma,v_0}(v_1)+|\zeta_{v_0}^{\gamma}(v_1)|^2\Psi_{\gamma,v_0}(v_0) =   |\Im \zeta_{v_0}^{\gamma}(v_1)| \, .
\end{equation}
Also, for any non-backtracking path $(v_0;v_k)$ in $T$, $k \ge 1$, we have
\begin{equation}     \label{eq:psiiden2}
\Psi_{\gamma,v_0}(v_k) - \overline{\zeta_{v_1}^{\gamma}(v_0)}\Psi_{\gamma,v_1}(v_k) - \zeta_{v_{k-1}}^{\gamma}(v_k)\Psi_{\gamma,v_0}(v_{k-1}) + \overline{\zeta_{v_1}^{\gamma}(v_0)}\zeta_{v_{k-1}}^{\gamma}(v_k)\Psi_{\gamma,v_1}(v_{k-1})=0 \, .
\end{equation}
\end{cor}
\begin{proof}
By (\ref{eq:idpsi}), $\Psi_{\gamma,v_0}(v_1) - \zeta_{v_0}^{\gamma}(v_1)\Psi_{\gamma,v_0}(v_0) = \Im \zeta_{v_0}^{\gamma}(v_1)\overline{G(v_0,v_0;\gamma)}$. As $\Psi_{\gamma,v_1}(v_0)=\Psi_{\gamma,v_0}(v_1)$, we thus get using \eqref{eq:multigreen3},
\begin{equation}            \label{eq:diam}
\overline{\zeta_{v_0}^{\gamma}(v_1)}\Psi_{\gamma,v_1}(v_0)-|\zeta_{v_0}^{\gamma}(v_1)|^2\Psi_{\gamma,v_0}(v_0)  =    \Im \zeta_{v_0}^{\gamma}(v_1)\cdot \overline{G(v_0,v_1;\gamma)} \, .
\end{equation}
Next, since $G(v_1,v_1;\gamma) = \frac{G(v_0,v_1;\gamma)}{\zeta_{v_1}^{\gamma}(v_0)}$ and $\frac{1}{\zeta_{v_1}^{\gamma}(v_0)} = \zeta_{v_0}^{\gamma}(v_1) + 2m_{v_0}^{\gamma}$, we have
\begin{equation}           \label{eq:greerel}
G(v_1,v_1;\gamma) = \zeta_{v_0}^{\gamma}(v_1)G(v_0,v_1;\gamma) + 2m_{v_0}^{\gamma}G(v_0,v_1;\gamma) = \zeta_{v_0}^{\gamma}(v_1)G(v_0,v_1;\gamma) - \zeta_{v_0}^{\gamma}(v_1) \, ,
\end{equation}
so
\[
\Psi_{\gamma,v_1}(v_1)  = \Im \zeta_{v_0}^{\gamma}(v_1) [\Re G(v_0,v_1;\gamma) - 1] + \Re \zeta_{v_0}^{\gamma}(v_1) \Psi_{\gamma,v_0}(v_1) \, ,
\]
and thus
\begin{align}          \label{eq:psi1rel}
\Psi_{\gamma,v_1}(v_1) - \zeta_{v_0}^{\gamma}(v_1)\Psi_{\gamma,v_0}(v_1) & = \Im \zeta_{v_0}^{\gamma}(v_1)[\Re G(v_0,v_1;\gamma)-1] - i\Im \zeta_{v_0}^{\gamma}(v_1)\Psi_{\gamma,v_0}(v_1) \nonumber \\
& = \Im \zeta_{v_0}^{\gamma}(v_1) \overline{G(v_0,v_1;\gamma)} - \Im \zeta_{v_0}^{\gamma}(v_1) \, .
\end{align}
This completes the proof of the first claim, by (\ref{eq:diam}). Next, we use again that $\Psi_{\gamma,v_0}(v_1) - \zeta_{v_0}^{\gamma}(v_1)\Psi_{\gamma,v_0}(v_0) = \Im \zeta_{v_0}^{\gamma}(v_1)\overline{G(v_0,v_0;\gamma)}$. In addition, by (\ref{eq:psi1rel}),
\begin{align*}
\overline{\zeta_{v_1}^{\gamma}(v_0)}[\Psi_{\gamma,v_1}(v_1) - \zeta_{v_0}^{\gamma}(v_1)\Psi_{v_1}(v_0)] & = \Im \zeta_{v_0}^{\gamma}(v_1)[\overline{\zeta_{v_1}^{\gamma}(v_0)G(v_0,v_1;\gamma)}-\overline{\zeta_{v_1}^{\gamma}(v_0)}] \\
& = \Im \zeta_{v_0}^{\gamma}(v_1)\overline{G(v_0,v_0;\gamma)} \, ,
\end{align*}
where the last equality is proved as in (\ref{eq:greerel}). This proves the second claim for $k=1$.

Now let $k\ge 2$. If we apply (\ref{eq:idpsi}) with $v_1$ instead of $v_0$ and use \eqref{eq:multigreen3},  we get
\[
\overline{\zeta_{v_1}^{\gamma}(v_0)}\Psi_{\gamma,v_1}(v_k) - \overline{\zeta_{v_1}^{\gamma}(v_0)}\zeta_{v_{k-1}}^{\gamma}(v_k)\Psi_{\gamma,v_1}(v_{k-1}) = \Im \zeta_{v_{k-1}}^{\gamma}(v_k)\cdot \overline{G(v_0,v_{k-1};\gamma)} \, .
\]
The second claim for $k\ge 2$ now follows by (\ref{eq:idpsi}).
\end{proof}

We conclude by recalling the fact that for Lebesgue a.e. $\lambda\in \R$, the Green function has a finite limit on the real axis almost surely. Remember that $\mathscr{T}_{\ast}^{D,A}$ us the set of coloured rooted trees, and that $\IP$ is the probability measure on $\mathscr{T}_{\ast}^{D,A}$ appearing in (BSCT).

\begin{prp}             \label{lem:limits}
There exists a Lebesgue-null set $\mathfrak{A}\subset \IR$ such that, to each $\lambda \in \mathfrak{S} := \R \setminus \mathfrak{A}$, there is $\Omega_\lambda\subseteq \mathscr{T}_{\ast}^{D,A}$ with $\prob(\Omega_\lambda)=1$, such that if $[\mathcal{T},o,\mathcal{W}]\in \Omega_\lambda$, then the limit $G(v,w;\lambda+i0):=\lim_{\eta \downarrow 0} G(v,w;\lambda+i\eta)$ exists for any $v,w\in \mathcal{T}$.
\end{prp}
\begin{proof}
Fix $[\mathcal{T},o,\mathcal{W}]$. By \cite[Lemma 3.3]{AS}, there is a Lebesgue-null set $\mathfrak{A}_{[\mathcal{T},o,\mathcal{W}]} \subset \R$ such that for any $\lambda \in \mathfrak{S}_{[\mathcal{T},o,\mathcal{W}]} := \R \setminus \mathfrak{A}_{[\mathcal{T},o,\mathcal{W}]}$, $ G(v,w;\lambda+i0)$ exists for all $v,w\in \mathcal{T}$. Let $\mathfrak{D} = \{([\mathcal{T},o,\mathcal{W}],\lambda): \text{ the limit does not exist}\}$. Then
\[
(\prob\mathop\otimes Leb)(\mathfrak{D}) = \int_{\mathscr{T}_{\ast}^{D,A}} Leb(\mathfrak{D}_{[\mathcal{T},o,\mathcal{W}]})\,\dd \prob([\mathcal{T},o,\mathcal{W}]) \, ,
\]
where $\mathfrak{D}_{[\mathcal{T},o,\mathcal{W}]} = \{\lambda \in \R: ([\mathcal{T},o,\mathcal{W}],\lambda) \in \mathfrak{D}\}$. Since $\mathfrak{D}_{[\mathcal{T},o,\mathcal{W}]} \subseteq \mathfrak{A}_{[\mathcal{T},o,\mathcal{W}]}$, we have $Leb(\mathfrak{D}_{[\mathcal{T},o,\mathcal{W}]}) = 0$ for all $[\mathcal{T},o,\mathcal{W}]$. Hence,
\[
0 = (\prob\mathop\otimes Leb)(\mathfrak{D}) = \int_{\R} \prob(\mathfrak{D}_\lambda)\,\dd \lambda  \, ,
\]
where $\mathfrak{D}_\lambda = \{[\mathcal{T},o,\mathcal{W}]\in \mathscr{T}_{\ast}^{D,A} : ([\mathcal{T},o,\mathcal{W}],\lambda) \in \mathfrak{D}\}$. It follows that $\prob(\mathfrak{D}_\lambda)=0$ on a Lebesgue-full set $\mathfrak{A}$. Taking $\Omega_\lambda = \mathfrak{D}_\lambda^c$ completes the proof.
\end{proof}

\section{The non-backtracking quantum variance\label{s:nb}}

Our strategy follows the one discovered in \cite{A}. We find a transformation turning the eigenfunctions of $\cA+ W$ on $G=\Gamma\backslash \tilG$ into eigenfunctions of a ``non-backtracking'' random walk. The new operator is not self-adjoint, but this difficulty is superseded by the fact that the trajectories of non-backtracking random walks (on a tree) are much simpler than those of usual random walks.

The notation is the same as in the introduction except that we drop the subscript $N$. Suppose $(\psi_j)$ is an orthonormal basis of eigenfunctions for $H  = \mathcal{A}+W$, say $H \psi_j = \lambda_j \psi_j$.

Fix $\eta_0\in(0,1)$, let $\gamma_j=\lambda_j+i\eta_0$ and let
\[
f_j(x_0,x_1) =\zeta_{x_0}^{\gamma_j}({x}_1)^{-1} \psi_j(x_1) -\psi_j(x_0) \, ,
\]
where $\zeta_x^{\gamma}(y) = - \tilde{g}_N^{(y|x)}(y,y;\gamma)$ (see notation in \S \ref{s:ident}). If $\cB$ is the non-backtracking operator \eqref{eq:nonback}, we have
\begin{align*}
(\cB \zeta^{\gamma_j}f_j)(x_0,x_1) & = \sum_{x_2 \in \mathcal{N}_{x_1} \setminus \{x_0\}} [\psi_j(x_2)-\zeta^{\gamma_j}_{x_1}(x_2)\psi_j(x_1)]\\
& = [\lambda_j \psi_j(x_1) - W(x_1) \psi_j(x_1) - \psi_j(x_0)] - \psi_j(x_1)\left[\gamma_j - W(x_1) - \frac{1}{\zeta_{x_0}^{\gamma_j}(x_1)}\right] \\
& =   f_j(x_0,x_1) -i\eta_0 \,\psi_j(x_1) \, ,
\end{align*}
where we used (\ref{eq:green3'}). Hence we get
\begin{equation}\label{e:newef}
\cB (\zeta^{\gamma_j} f_j) = f_j - i\eta_0 \,\tau_+\psi_j\, 
\end{equation}
where $\tau_{\pm} :\ell^2(V) \to \ell^2(B)$ are defined by
\[
(\tau_-\psi)(x_0,x_1) = \psi(x_0) \quad \text{and} \quad (\tau_+\psi)(x_0,x_1) = \psi(x_1) \, .
\]
In \cite{A} it was possible to set $\eta_0=0$, and \eqref{e:newef} said exactly that $f_j$ was an eigenfunction of the weighted non-backtracking operator $\cB \zeta^{\gamma_j}$
for the eigenvalue $1$. At our level of generality, we do not know if $\zeta^{\lambda_j+i0}$ is well-defined on $\widetilde{G}_N$. We have to work with $\eta_0>0$ and let $\eta_0$ tend to $0$ only at the end of the proof, after $N$ has gone to $\infty$. Hence, $f_j$ is not exactly an eigenfunction, and our formulas will contain error terms of size $\eta_0$ that we will need to estimate precisely, to show that they disappear as $N\to +\infty$, followed by $\eta_0\downarrow 0$.

Similarly, if we put
\[
f_j^{\ast}(x_0,x_1) = \zeta_{x_1}^{\gamma_j}(x_0)^{-1}\psi_j(x_0) - \psi_j(x_1) \, ,
\] we note that $f_j^{\ast}=\iota f_j$ where $\iota$ is the edge reversal involution, and we get
\begin{equation}\label{e:newef2}
  \cB^{\ast} (\iota \zeta^{\gamma_j}f_j^{\ast}) = f_j^{\ast}  -i\eta_0\,\tau_-\psi_j\, .
\end{equation}

Let $I$ be an open interval such that $\overline{I}\subset I_1$. We define for $K \in \mathscr{H}_k$,
 
\begin{equation}\label{e:varnb}
 {\mathrm{Var_{nb, \eta_0}^I}}(K) = \frac{1}{N} \sum_{\lambda_j\in I} \left|\left\langle f_j^{\ast},K_Bf_j\right\rangle\right| \, .
\end{equation}
The dependence of this quantity on $\eta_0$ is hidden in the definition of $f_j, f_j^*$. The scalar product $\la\cdot, \cdot\ra$ is on $\ell^2(B)$ endowed with the uniform measure; cf. (\ref{eq:kb}).

\begin{rem} We call (\ref{e:varnb}) ``quantum variance'', in analogy to the quantity bearing this name in quantum chaos. However, there are some significant differences~:
\begin{itemize}
\item we use the functions $f_j$ and $f_j^*$ instead of the original $\psi_j$. They are (quasi)-eigenfunctions, respectively of the non-selfadjoint operators $\cB \zeta^{\gamma_j}$ and  $\cB^{\ast} \iota \zeta^{\gamma_j}$.
\item if $K$ is the identity operator $Id$, we do not have the normalization ${\mathrm{Var_{nb, \eta_0}^I}}(Id)=1$.
%In fact, in the models treated in \cite{A}, we have ${\mathrm{Var_{nb, \eta_0=0}^I}}(Id)=0$, which means that $f_j$ and $f_j^*$ are orthogonal.
\item we did not take the square of $\left|\left\langle f_j^{\ast},K_Bf_j\right\rangle\right|$ in the definition. This is purely for technical convenience, the square will appear later when we apply the Cauchy-Schwarz inequality.
\end{itemize}
\end{rem}
We will need to extend \eqref{e:varnb} to operators $K$ that depend on the eigenvalue $\lambda_j$ in a holomorphic fashion, as spelled out in the following definition. Note that $K$ also depends on $N$, also this tends to be implicit in our notation. We let $\C^+=\{\gamma\in \IC, \Im \gamma>0\}$.

\begin{defa}\label{d:hol}Assumptions {\bf (Hol)}.

We assume that $\gamma\mapsto K^\gamma=K_N^\gamma$ is a map from $\gamma\in\C^+$ to $\mathscr{H}_k$ such that~:
\begin{itemize}
\item For $\eta_0>0$, for each $N$ and $(x_0;x_k)$, the function $\lambda \mapsto K^{\lambda+i\eta_0}(x_0;x_k)$ from $\R\to \C$ has an analytic extension $K_{\eta_0}$ to the strip $\{z: |\Im z| < \eta_0/2\}$. 
\item Given $\eta_0>0$, we have $\sup_N \sup_{\Re z\in I_1, |\Im z|< \eta_0/2}\sup_{(x_0;x_k)}|K_{N, \eta_0}^z(x_0;x_k)|<+\infty$ and $\sup_N \sup_{\Re z\in I_1,|\Im z|< \eta_0/2}\sup_{(x_0;x_k)}|\partial_z K_{N, \eta_0 }^z(x_0;x_k)|<+\infty$. We write $\vertiii{K}_{\eta_0}$ for the maximum of these two quantities.
\item For all $s>0$,
\begin{equation}\label{e:relevantbound}
\sup_{\eta_1\in(0,1)}\limsup_{N\to+\infty}\sup_{\lambda\in I_1}\frac1N\sum_{(x_0;x_k)\in B_k} |K_{N}^{\lambda+i\eta_1}(x_0; x_k)|^s <+\infty \, .
\end{equation}
%\item We assume that, for a given $\eta_0$, for any $\eps>0$, there exists $d\in \IN$ and $K_N^k(x, y)$ ($ k=0,\ldots, d$) such that
%\[\sup_N \sup_{\eta_0/2<\Im z<3\eta_0/2}\sup_{x, y}|K_{N}^z(x, y)-\sum_{k=0}^d z^k K_N^k(x, y)|<\eps\]
%(approximation by a polynomial in $N$)
%\[\sup_N  \sup_{x, y}|K_N^k(x, y)|<+\infty.\]

%Note that this assumption is stronger than the previous one.
\end{itemize}
\end{defa}
If $\gamma\mapsto K^\gamma$ is holomorphic on $\C^+$, then it obviously satisfies the first point of the definition with $K_{\eta_0}(z)=K^{z+i\eta_0}$.
For instance, if $K^\gamma(x_0;x_k)$ has the form $\sum_{n\ge 0} a_{(x_0;x_k)}^{(n)}\gamma^n$, then we see that $\lambda \mapsto K^{\lambda+i\eta_0}(x_0;x_k)$ extends to $K_{\eta_0}(z)=\sum_{n\ge 0} a_{(x_0;x_k)}^{(n)}(z+i\eta_0)^n$. Note that, although $\gamma\mapsto \overline{K^\gamma}$ is not holomorphic, its restriction to an horizontal line is still a real-analytic map $\R\ni \lambda \mapsto \overline{K^{\lambda+i\eta_0}(x_0;x_k)}$, as it possesses an analytic extension given by $z\mapsto \sum_{n\ge 0} \overline{a_{(x_0;x_k)}^{(n)}}(z-i\eta_0)^n$. So $  \overline{K^\gamma}$ will satisfy  {\bf (Hol)} if $K^\gamma$ does.

Conditions \textbf{(Hol)} are stable under the sum and composition of operators.

We extend \eqref{e:varnb} to this setting, by letting
\begin{equation}\label{e:varnb2}
 {\mathrm{Var_{nb, \eta_0}^I}}(K^{\gamma}) = \frac{1}{N} \sum_{\lambda_j\in I} \left|\left\langle f_j^{\ast},K_{ B}^{\lambda_j+i\eta_0}f_j\right\rangle\right| \, .
\end{equation}

Most of the paper is devoted to showing~:
\begin{thm}\label{t:thm4}
Under assumptions \emph{\textbf{(EXP)}}, \emph{\textbf{(BSCT)}}, \emph{\textbf{(Green)}}, if $K^\gamma\in \mathscr{H}_k$ has the form $K^\gamma =\cF_\gamma K$ for the operators $\cF_\gamma$ in Corollary~\ref{cor:recurrence}, then
\[
\lim_{\eta_0\downarrow 0}\lim_{N\to +\infty} {\mathrm{Var_{nb, \eta_0}^I}}(K^{\gamma})=0 \, .
\]
\end{thm}
These $\gamma \mapsto \cF_{\gamma}K$ satisfy \textbf{(Hol)}. The fact that this implies Theorem \ref{thm:1} is proven in Section \ref{sec:retour}, that may be read independently of the proof of Theorem \ref{t:thm4}.

\section{Step 1~: Bound on the non-backtracking quantum variance}                      \label{s:proof1} 

Given $\gamma\in \IC^+$, we introduce a norm on each $\mathscr{H}_k$, $k \ge 1$, defined by 
\begin{equation}      \label{eq:normgamma}
\|K\|_{{\gamma}}^2 =  \frac{1}{N} \sum_{(x_0;x_k)\in B_k} \frac{|\Im \zeta_{x_1}^{\gamma}(x_0)|}{|\zeta_{x_1}^{\gamma}(x_0)|^2} \cdot |K(x_0;x_k)|^2 \cdot \frac{|\Im \zeta_{x_{k-1}}^{\gamma}(x_k)|}{|\zeta_{x_{k-1}}^{\gamma}(x_k)|^2} \, .
\end{equation}
We denote by $\la\cdot, \cdot\ra_{\gamma}$ the associated scalar product. The reason for introducing the weight $\frac{|\Im \zeta_x^{\gamma}(y)|}{|\zeta_x^{\gamma}(y)|^2}$ will be apparent in Section~\ref{s:Mark}. The aim of this section is to prove Theorem~\ref{thm:upvar}. Here, we assume that $I=(a,b)$, with $[a,b]\subset I_1$. This implies that there is $\eta_{a,b}$ such that $(a-2\eta,b+2\eta)\subset I_1$ for all $\eta \le \eta_{a,b}$. We then assume that $\eta \le \min(\eta_0/2,\eta_{a,b})$.

\begin{thm}       \label{thm:upvar}
Under assumptions \emph{\textbf{(BSCT)}}, \emph{\textbf{(Green)}}, if $K^\gamma \in \mathscr{H}_k$ satisfies the set of assumptions \emph{\textbf{(Hol)}}, then for any interval $I=(a,b)$ as above,
\[
\lim_{\eta_0\downarrow 0}\limsup_{N\to +\infty}\varnbi(K^{\gamma})^2\leq D\,|I|\,\lim_{\eta_0 \downarrow 0}\lim_{\eta\downarrow 0}\limsup_{N\to \infty} \int_{a-2\eta}^{b+2\eta} \|K^{\lambda+i(\eta^4+\eta_0)}\|_{\lambda+i(\eta^4+\eta_0)}^2\, \dd \lambda \, .
\]
\end{thm}

In the scheme of \S \ref{s:outline}, this corresponds to Step 1. This is more complicated than usual, due to the fact that we have replaced the orthonormal family $(\psi_j)$ by non-orthogonal functions $(f_j), (f_j^*)$, and also because $K$ ``depends on $\lambda_j$'' in \eqref{e:varnb2}.

Recall that $D$ above is the maximal degree and we assumed $|W_N(x)|\le A$. In particular, any eigenvalue $\lambda_j \in I_0 :=[-(A+D),A+D]$. For $\lambda\in \IR$ and $\eta_0\in(0,1)$, let
\[
\alpha_{\lambda+i\eta_0}(x_0,x_1)=\frac{|\Im \zeta^{\lambda+i\eta_0}_{x_1}(x_0)|^{1/2}}{\zeta_{x_1}^{\lambda+i\eta_0}(x_0)} \, .
\]
Denoting $\gamma_j={\lambda_j+i\eta_0}$, we have (by a double application of the Cauchy-Schwarz inequality)
\begin{align}    \label{eq:firsteqonvar}
  {\mathrm{Var_{nb, \eta_0}^I}}(K^{\gamma})  
  &  \le \frac{1}{N} \sum_{\lambda_j\in I}  \left\| \overline{\alpha_{\gamma_j}}^{-1}f_j^{\ast}\right\|  \left\|  \alpha_{\gamma_j} K^{\gamma_j}_Bf_j\right\|  \nonumber \\
& \le \frac{1}{N} \Big(\sum_{\lambda_j\in I}   \left\|\overline{\alpha_{\gamma_j}}^{-1} f_j^{\ast}\right\|^2\Big)^{1/2} \Big(\sum_{\lambda_j\in I} \left\|  \alpha_{\gamma_j} K^{\gamma_j}_Bf_j\right\|^2\Big)^{1/2}.
\end{align}
We check at the end of the section that
\begin{equation}\label{e:flim'}
\lim_{\eta_0\downarrow 0} \limsup_{N\to +\infty}\frac{1}{N} \sum_{\lambda_j\in I}   \left\|\overline{\alpha_{\gamma_j}}^{-1} f_j^{\ast}\right\|^2 \le D \cdot |I| \, .
\end{equation}

We now introduce an approximation $\chi$ of $\bbbone_I$ by an entire function, by the standard convolution procedure~:

Fix $0 < \eta \le \eta_0/2$. Let $\phi(x) = \frac{1}{\pi^{1/2}}e^{-x^2}$ and denote $\phi_{\epsilon}(x) = \epsilon^{-1}\phi(x/\epsilon)$. Let $\chi$ be the convolution $\chi = \phi_{\eta^{3/2}} \ast \bbbone_I$ on $\R$. Then $\chi$ extends to an entire function on $\C$ given by 
\begin{equation}
\chi(z) = \frac{1}{\eta^{3/2}\pi^{1/2}} \int_I e^{-(z-y)^2/\eta^3}\,\dd y.\label{e:convol}
\end{equation}
Note that $0 \le \chi(x) \le 1$ for $x\in \R$, and $|\chi(z)| \le e^{\eta^5}$ for $|\Im z| \le \eta^4$. We assume $\eta$ is small enough so that $\chi  \ge \frac{1}{3} \bbbone_I$ and $|\chi(z)| \le e^{-1/\eta}$ on $\{z\in\IC: |\Im z|\leq \eta^4,\ d(\Re z, I)\geq 2\eta\}$. We finally note that $|\frac{\partial \chi}{\partial t_2}(t_1+it_2)| \le C\eta^{-3}e^{\eta^5}$ for any $z=t_1+it_2$ with $t_1\in I_0$ and $|t_2| \le \eta^4$.

%Introduce $0<\eta<\eta_0$ and let $\chi (z)=\frac{1}{\eta^{3/2} \pi^{1/2}}\int_I  e^{-(z-y)^2/\eta^3 }\,\dd y$. Then $\chi$ is an entire function on $\C$. Note that $0 \le \chi(x) \le 1$ for $x\in \R$, and $|\chi(z)| \le e^{\eta}$ for $|\Im z| \le \eta^2$. We assume $\eta$ is small enough so that $\chi \ge \frac{1}{3}  \bbbone_I$ and $|\chi(z)| \le e^{-1/\eta}$ on $\{z\in\IC: |\Im z|\leq \eta^2,\ d(\Re z, I)\geq 2\eta\}$. 

By (\ref{eq:firsteqonvar}) and (\ref{e:flim'}) we have
\begin{equation}    \label{e:very1st}
\limsup_{N\to \infty}{\mathrm{Var_{nb, \eta_0}^I}}(K^{\gamma})^2 \le \limsup_{N\to \infty}\frac{3D\,|I|}{N} \sum_{j=1}^N  \chi(\lambda_j)\,\|  \alpha_{\gamma_j}K_B^{\gamma_j}f_j\|^2 \, .
\end{equation}
Now by (\ref{eq:KB2}), we have
\begin{align*}
\|\alpha_{\gamma_j}K_B^{\gamma_j}f_j\|^2 & = \sum_{(x_0,x_1)\in B} \sum_{(x_2;x_k) ,(y_2;y_k) } |\alpha_{\gamma_j}(x_0,x_1)|^2K^{\gamma_j}(x_0;x_k)\overline{K^{\gamma_j}(x_0;y_k)} \\
& \qquad \cdot [\zeta_{x_{k-1}}^{\gamma_j}(x_k)^{-1}\psi_j(x_k)-\psi_j(x_{k-1})]\overline{[\zeta_{y_{k-1}}^{\gamma_j}(y_k)^{-1}\psi_j(y_k)-\psi_j(y_{k-1})]} \, ,
\end{align*}
where $(x_0;x_k)=(x_0,x_1,x_2,\dots,x_k)$, $(x_0;y_k)=(x_0,x_1,y_2,\dots,y_k)$ and with the convention that $K^{\gamma_j}(x_0;x_k)=0$ if the path $(x_0,x_1,x_2,\dots,x_k)$ backtracks. The function $\lambda \mapsto |\alpha_{\lambda+i\eta_0}(x_0,x_1)|^2=\frac{-\Im \zeta_{x_1}^{\lambda+i\eta_0}(x_0)}{|\zeta_{x_1}^{\lambda+i\eta_0}(x_0)|^2}$ extends analytically to the rectangle $\mathscr{R} = \{z\in \C: \Re z \in [-(A+D+\eta),(A+D+\eta)], \Im z \in [-\eta^4,\eta^4]\}$ through the formula $\frac{\zeta_{x_1}^{z-i\eta_0}(x_0) - \zeta_{x_1}^{z+i\eta_0}(x_0)}{2i\,\zeta_{x_1}^{z+i\eta_0}(x_0)\zeta_{x_1}^{z-i\eta_0}(x_0)}$. We denote this by ${\alpha}_{\eta_0}^z(x_0,x_1)$ (which is not the same as $|\alpha_{z+i\eta_0}(x_0,x_1)|^2$). The same is true for the other $\zeta$ terms. We denote the extension of $\lambda \mapsto K^{\lambda+i\eta_0}(x_0;x_k) \overline{K^{\lambda+i\eta_0}(x_0;y_k)}$ by ${K}^{z}_{\eta_0}(x_0;x_k,y_k)$. Again, if $(x_0;y_k) = (x_0;x_k)$, this is not the same as $|K^{z+i\eta_0}(x_0;x_k)|^2$. However, see Lemma~\ref{l:errors} to compare both.

Given $x, y\in V$ and $z\in \C \setminus \R$, let
\[
g^z(x, y)=\la \delta_x, (H-z)^{-1}\delta_y\ra_{\ell^2(V)}=\sum_{j=1}^N  \frac{ \psi_j(x)\overline{\psi_j(y)}}{\lambda_j-z}
\]
be the Green function of $H$ on the finite graph $G$. Then by Cauchy's integral formula,

\begin{align}      \label{eq:intr}
\frac{1}{N}\sum_{j=1}^N\chi(\lambda_j)\,\|\alpha_{\gamma_j}K_B^{\gamma_j}f_j\|^2 & = \frac{-1}{2i\pi N} \int_{z\in \partial \mathscr{R}} \sum_{(x_0,x_1)\in B} \sum_{(x_2;x_k) ,(y_2;y_k) } \chi(z){\alpha}_{\eta_0}^z(x_0,x_1) \nonumber \\
& \qquad {K}^{z}_{\eta_0}(x_0;x_k,y_k)\cdot \Big[\frac{g^z(x_k,y_k)}{\zeta_{x_{k-1}}^{z+i\eta_0}(x_k)\zeta_{y_{k-1}}^{z-i\eta_0}(y_k)} - \frac{g^z(x_k,y_{k-1})}{\zeta_{x_{k-1}}^{z+i\eta_0}(x_k)} \\
& \qquad \qquad \qquad \qquad \qquad   - \frac{g^z(x_{k-1},y_k)}{\zeta_{y_{k-1}}^{z-i\eta_0}(y_k)} + g^z(x_{k-1},y_{k-1})\Big] \,\dd z\, . \nonumber
\end{align}

We now observe that the integral over the vertical segments of the contour do not contribute as $\eta,\eta_0 \downarrow 0$. More precisely,

\begin{lem}     \label{l:intseg}
The integral $\frac{-1}{2i\pi N}\int_{z\in \partial \mathscr{R}} F(z)\,\dd z$ in \eqref{eq:intr} may be replaced by $\frac{1}{2i\pi N} (\int_{a-2\eta}^{b+2\eta} F(\lambda+i\eta^4)\,\dd \lambda - \int_{a-2\eta}^{b+2\eta} F(\lambda-i\eta^4)\,\dd \lambda$, up to an error term at most $C_{k,D,A}\eta_0^{-3}\eta^{-4}\vertiii{K}_{\eta_0}^2e^{-1/\eta}$.
\end{lem}
\begin{proof}
The error is the integral of $F(z)$ on the two vertical paths $\{ \Re z = -A-D-\eta, \Im z\in [-\eta^4, \eta^4]\}$,
 $\{ \Re z = A+D+\eta, \Im z\in [-\eta^4, \eta^4]\}$, and the four connected components of the set $\{\Im z= \pm \eta^4, \Re z \in [-A-D-\eta, A+D+\eta]\setminus (a-2\eta,b+2\eta)\}$. On these pieces, we know that $|\chi(z)|\le e^{-1/\eta}$. Moreover, $| {K}^{z}_{\eta_0}(x_0;x_k,y_k)| \le \vertiii{K}_{\eta_0}^2$. Next, $|{\alpha}_{\eta_0}^z| = \frac{1}{2}\big|\frac{1}{\zeta_{x_1}^{z+i\eta_0}(x_0)} - \frac{1}{\zeta_{x_1}^{z-i\eta_0}(x_0)}\big| \le c_{D,A}\big(\frac{1}{\eta_0+\eta^4}+\frac{1}{\eta_0-\eta^4}\big)$ by (\ref{e:green3cons}). Since $\eta\le \eta_0/2$ by assumtpion, this yields $|{\alpha}_{\eta_0}^z| \le C_{D,A}\eta_0^{-1}$. The Green functions and $\zeta$ terms may be bounded similarly by $4c_{D,A}\eta_0^{-2}\eta^{-4}$. A factor $C_{k,D}$ comes from the number of paths, divided by $N$.
\end{proof}

Our next aim is to lift this expression to the universal cover $\widetilde{G}$. In other words, we wish to replace $g^z$ by $\tilde g^z$ everywhere, to be able to use the identities of \S \ref{s:ident}.

\begin{lem}          \label{lem:BSTvar}
Denote $z=\lambda+i\eta^4$. Given $R\in \N^{\ast}$, there is $d_{R,k,\eta}>0$ such that the integral $\frac{1}{2i\pi N} \int_{a-2\eta}^{b+2\eta} F(z)\,\dd \lambda$ in Lemma~\ref{l:intseg} may be replaced by
\begin{align*}
& \frac{1}{2i\pi N} \int_{a-2\eta}^{b+2\eta} \sum_{\rho_G(x_0) \ge d_{R,k,\eta}} \sum_{x_1\sim x_0} \sum_{(x_2;x_k) ,(y_2;y_k) } \chi(z){\alpha}_{\eta_0}^z(x_0,x_1) \\
& \qquad {K}^{z}_{\eta_0}(x_0;x_k,y_k)\cdot \Big[\frac{\tilg^z(\tilde{x}_k,\tilde{y}_k)}{\zeta_{e_k}^{z+i\eta_0}\zeta_{e_k'}^{z-i\eta_0}} - \frac{\tilg^z(\tilde{x}_k,\tilde{y}_{k-1})}{\zeta_{e_k}^{z+i\eta_0}} - \frac{\tilg^z(\tilde{x}_{k-1},\tilde{y}_k)}{\zeta_{e_k'}^{z-i\eta_0}} + \tilg^z(\tilde{x}_{k-1},\tilde{y}_{k-1})\Big] \,\dd \lambda \,,
\end{align*}
where $\zeta_{e_k}^{\gamma} = \zeta_{x_{k-1}}^{\gamma}(x_k)$ and $\zeta_{e_k'}^{\gamma} = \zeta_{y_{k-1}}^{\gamma}(y_k)$,
up to an error term $(\frac{\# \{\rho_G(x_0) < d_{R,k,\eta}\}}{N}\eta^{-4} + \frac{1}{R})C_{k,D,A}\eta_0^{-3} \vertiii{K}_{\eta_0}^2e^{\eta^5}$.

Similarly, $\frac{1}{2i\pi N} \int_{a-2\eta}^{b+2\eta}F(\bar{z})\,\dd \lambda$ in Lemma~\ref{l:intseg} may be replaced by
\begin{align*}
& \frac{1}{2i\pi N} \int_{a-2\eta}^{b+2\eta} \sum_{\rho_G(x_0) \ge d_{R,k,\eta}} \sum_{x_1\sim x_0} \sum_{(x_2;x_k) ,(y_2;y_k) }   \chi(\bar{z})  {\alpha}^{\bar{z}}_{\eta_0}(x_0,x_1) {K}^{\bar{z}}_{\eta_0}(x_0;x_k,y_k) \\
& \qquad \qquad \qquad\cdot \Big[\frac{\tilg^{\bar{z}}(\tilde{x}_k,\tilde{y}_k)}{\zeta_{e_k}^{\bar{z}+i\eta_0}\zeta_{e_k'}^{\bar{z}-i\eta_0}} - \frac{\tilg^{\bar{z}}(\tilde{x}_k,\tilde{y}_{k-1})}{\zeta_{e_k}^{\bar{z}+i\eta_0}}  - \frac{\tilg^{\bar{z}}(\tilde{x}_{k-1},\tilde{y}_k)}{\zeta_{e_k'}^{\bar{z}-i\eta_0}} + \tilg^{\bar{z}}(\tilde{x}_{k-1},\tilde{y}_{k-1})\Big]   \,\dd \lambda
\end{align*}
up to an error term $(\frac{\# \{\rho_G(x_0) < d_{R,k,\eta}\}}{N}\eta^{-4} + \frac{1}{R})C_{k,D,A}\eta_0^{-3} \vertiii{K}_{\eta_0}^2e^{\eta^5}$.
\end{lem}
\begin{proof}
We first approximate $\lambda\mapsto g^{\lambda + i\eta^4}(x,y)$ by a polynomial on the compact interval $I_0$. Let $h_{\eta}(t)=-(t- i\eta^4)^{-1}$ and choose a polynomial $q_{ \eta}$ with $\|h_{ \eta}-q_{ \eta}\|_{\infty}<\frac{1}{R}$. Then $\|h_{\eta}(H-\lambda)-q_{\eta}(H-\lambda)\|<\frac{1}{R}$, so $|g^{\lambda + i\eta}(x,y)-q_{ \eta}(H-\lambda)(x,y)|<\frac{1}{R}$ for any $x,y$ and $\lambda$. So replacing each $g^{\lambda + i\eta^4}(x,y)$ by $q_{ \eta}(H-\lambda)(x,y)$ in the sums gives an error term $\frac{C_{k,D,A}\eta_0^{-3} \vertiii{K}_{\eta_0}^2e^{\eta^5}}{R}$ as in Lemma~\ref{l:intseg}.

Denote $C_{k,D,A,\eta_0} = C_{k,D,A}\eta_0^{-3}\|K\|_{\eta_0}^2$.

Let $d_{R,\eta}$ be the degree of $q_{\eta}$. Suppose $\rho_G(x_0) \ge d_{R,\eta}+k =: d_{R,k,\eta}$. Then it is easy to see that $q_{\eta}(H-\lambda)(x_k,y_k) = q_{\eta}(\widetilde{H}-\lambda)(\tilde{x}_k,\tilde{y}_k)$, c.f. Lemma~\ref{lem:rootspec}. The same holds for the other edges $(x_k,y_{k-1})$ and so on. The terms with $\rho_G(x_0) < d_{R,k,\eta}$ bring an error term $\frac{\# \{\rho_G(x_0) < d_{R,k,\eta}\}}{N}\eta^{-4}C_{k,D,A,\eta_0}$. Finally, we replace the $q_{\eta}(\widetilde{H}-\lambda)(\tilde{x},\tilde{y})$ by $\tilg^{\lambda + i\eta^4}(\tilde{x},\tilde{y})$ which yields again an error of the form $\frac{C_{k,D,A,\eta_0}}{R}$.

This proves the first statement, and the second one is proven similarly.
\end{proof}

We continue to simplify the expression and record the following.

\begin{lem}    \label{l:errors}
If we replace ${\alpha}_{\eta_0}^z(x_0,x_1){K}^{z}_{\eta_0}(x_0;x_k,y_k)$ and $  {\alpha}^{\bar{z}}_{\eta_0}(x_0,x_1)K^{\bar{z}}_{\eta_0}(x_0;x_k,y_k)$ in Lemma~\ref{lem:BSTvar} by $|\alpha_{z+ i\eta_0}(x_0,x_1)|^2 K^{z+ i\eta_0}(x_0;x_k) \overline{K^{z+ i\eta_0}(x_0;y_k)}$, then as $N\to \infty$, the error we get is at most $C_{k,D,A}\eta_0^{-6}\vertiii{K}_{\eta_0}^2e^{\eta^5} \eta^4$. We may also replace $\chi(\lambda\pm i\eta^4)$ by $\chi(\lambda)$, modulo the asymptotic error $C_{k,D,A} \eta_0^{-3}\vertiii{K}_{\eta_0}^2e^{\eta^5}\eta$. Finally, we may replace each $\zeta_{e_k}^{\bar{z} +i\eta_0}$ by $\zeta_{e_k}^{z + i \eta_0}$ and $\zeta_{e_k'}^{z-i\eta_0}$ by $\zeta_{e_k'}^{\bar{z}-i\eta_0}$, modulo an asymptotic error $C_{k,D,A}\eta_0^{-6}\vertiii{K}_{\eta_0}^2 e^{\eta^5}\eta^4$.
\end{lem}

\begin{proof}
We start with ${\alpha}_{\eta_0}^z(x_0,x_1){K}^{z}_{\eta_0}(x_0;x_k,y_k)$. Denote $e=(x_0, x_1)$ and $\zeta_{e}^{\gamma} = \zeta_{x_1}^{\gamma}(x_0)$. We note that
\begin{multline*}
 \left|{\alpha}_{\eta_0}^z(x_0,x_1)-|\alpha^{z+i\eta_0}(x_0,x_1)|^2\right|  = \left| \frac{\zeta_{e}^{z-i\eta_0}-\zeta_{e}^{z+i\eta_0}}{2i \zeta_{e}^{z+i\eta_0}\zeta_{e}^{z-i\eta_0}} - \frac{\zeta_{e}^{\bar{z}-i\eta_0} - \zeta_{e}^{z+i\eta_0}}{2i\zeta_{e}^{z+i\eta_0}\zeta_{e}^{\bar{z}-i\eta_0}}\right| \\
  = \frac{1}{2} \left| \frac{1}{\zeta_{e}^{\bar{z}-i\eta_0}} - \frac{1}{\zeta_{e}^{z-i\eta_0}}\right| \le C_{D,A}\eta_0^{-2} \left| \zeta_{e}^{z-i\eta_0} - \zeta_{e}^{\bar{z}-i\eta_0}\right|\\
 \le C_{D,A}\eta_0^{-4} |z - \bar{z}| = 2C_{D,A}\eta_0^{-4} \eta^4 \, ,
\end{multline*}
where we used (\ref{e:green3cons}) in the first inequality and the resolvent identity in the second one. Similarly, $K^{z+ i\eta_0}(x_0;x_k) \overline{K^{z+ i\eta_0}(x_0;y_k)}$ is the same as ${K}^{z}_{\eta_0}(x_0;x_k,y_k)$, but with each $z-i\eta_0$ replaced by $\bar{z}-i\eta_0$. It follows that $|{K}^{z}_{\eta_0}(x_0;x_k,y_k) - K^{z+ i\eta_0}(x_0;x_k) \overline{K^{z+ i\eta_0}(x_0;y_k)}| \le 2\sup |\partial_zK(v_0;v_k)| \sup |K(v_0;v_k)| \cdot |z-\bar{z}| \le 4\vertiii{K}_{\eta_0}^2 \eta^4$. Hence, ${\alpha}_{\eta_0}^z(x_0,x_1){K}^{z}_{\eta_0}(x_0;x_k,y_k)$ is the same as $|\alpha_{z+ i\eta_0}(x_0,x_1)|^2 K^{z+ i\eta_0}(x_0;x_k) \overline{K^{z+ i\eta_0}(x_0;y_k)}$, modulo $C_{D,A}\eta_0^{-4}\vertiii{K}_{\eta_0}^2 \eta^4$. This error is further multiplied by the function $\chi$. Bounding the $\zeta$ terms by some $c_{D,A}\eta_0^{-2}$ and $|\chi(z)|$ by $e^{\eta^5}$, we end up with an error term at most
\[
\int_{a-2\eta}^{b+2\eta}\frac{C_{D,A}\eta_0^{-6}\|K\|_{\eta_0}^2e^{\eta^5}\eta^4}{N}\sum_{(x_0,x_1)}\sum_{(x_2;x_k) ,(y_2;y_k) } |\tilg^{\lambda\pm i \eta^4}(\tilde{x}_k,\tilde{y}_k)| \,\dd \lambda
\]
and a similar upper bound for each term involving $\tilg^{\lambda\pm i\eta^4} $. Since $I_{\eta}=(a-2\eta,b+2\eta)\subset I_1$, we may use  Remark~\ref{rem:IDS3} to deduce that the integrand is uniformly bounded over $\lambda\in I_{\eta}$ by $C_{k,D,A}\eta_0^{-6}\vertiii{K}_{\eta_0}^2e^{\eta^5}\eta^4$ as $N\to \infty$. Note that $|I_{\eta}| \le |I_0| =2(D+A)$.

This proves the first claim. The second claim is similar, for example $|  {\alpha}^{\bar{z}}_{\eta_0}(x_0,x_1) - |\alpha^{z+i\eta_0}(x_0,x_1)|^2| \le C_{D,A}\eta_0^{-2} |\zeta_{e}^{z+i\eta_0}-\zeta_{e}^{\bar{z}+i\eta_0}| \le 2C_{D,A}\eta_0^{-4}\eta^4$. Moreover, $K^{\bar{z}}_{\eta_0}(x_0;x_k,y_k)$ is the same as $K^{z+ i\eta_0}(x_0;x_k) \overline{K^{z+ i\eta_0}(x_0;y_k)}$ with each $z+i\eta_0$ replaced by $\bar{z}+i\eta_0$, so the proof carries on. For the third claim, note that $|\chi(\lambda\pm i\eta^4) - \chi(\lambda)| \le \sup_{z\in \mathscr{R}} |\frac{\partial \chi}{\partial x_2}(z)| \cdot \eta^4 \le Ce^{\eta^5} \eta$. For the last claim, $|(\zeta_e^{z\pm i\eta_0})^{-1} - (\zeta_e^{\bar{z}\pm i\eta_0})^{-1}| \le 2C_{D,A}\eta_0^{-4}\eta^4$ as we previously saw when analyzing ${\alpha}_{\eta_0}^z$, so we get a similar error.
\end{proof}

By virtue of Lemma~\ref{lem:BSTvar} and ~\ref{l:errors}, denoting $z=\lambda+i\eta^4$,
we know at this stage that modulo some error terms, the expression \eqref{eq:intr} may be replaced by
\begin{multline}\label{eq:greenterms1}
  \frac{1}{\pi N}\int_{a-2\eta}^{b+2\eta} \sum_{\rho_G(x_0)\ge d_{R,k,\eta}}\sum_{x_1\sim x_0} \sum_{(x_2;x_k) ,(y_2;y_k) } \chi(\lambda)|\alpha_{z+ i\eta_0}(x_0,x_1)|^2\\
  \qquad  \qquad   \qquad   \qquad    \qquad   \qquad   \qquad  \qquad   \qquad   \cdot K^{z+ i\eta_0}(x_0;x_k) \overline{K^{z+ i\eta_0}(x_0;y_k)} \\
\cdot \left(\frac{\Im \tilg^z(\tilde{x}_k,\tilde{y}_k)}{\zeta_{e_k}^{z+i\eta_0}\zeta_{e_k'}^{\bar{z}-i\eta_0}} - \frac{\Im \tilg^z(\tilde{x}_k,\tilde{y}_{k-1})}{\zeta_{e_k}^{z+i\eta_0}}  - \frac{\Im \tilg^z(\tilde{x}_{k-1},\tilde{y}_k)}{\zeta_{e_k'}^{\bar{z}-i\eta_0}} + \Im \tilg^z(\tilde{x}_{k-1},\tilde{y}_{k-1})\right)   \,\dd \lambda\, . 
\end{multline}

 We now make the expression more homogeneous as follows:

\begin{lem}   \label{l:greenco}
Assume we have made all the replacements in Lemma~\ref{l:errors}. If we finally replace each of the four $\Im \tilg^z(\tilde{x},\tilde{y})$ by $\Im \tilg^{z+i\eta_0}(\tilde{x},\tilde{y})$ in \eqref{eq:greenterms1}, then the error term vanishes as $N\to \infty$, followed by $\eta\downarrow 0$, followed by $\eta_0\downarrow 0$.
\end{lem}
\begin{proof}
We only analyze the first error term, the other three are similar.

Choose $p, q, r$ such that $\frac1p+\frac1q+\frac1r=1$, and use the H\"older's inequality,
\begin{multline*} 
\Big|\frac{1}{\pi N}\int_{a-2\eta}^{b+2\eta} \sum_{\rho_G(x_0) \ge d_{R,k,\eta}} \sum_{x_1\sim x_0} \sum_{(x_2;x_k) ,(y_2;y_k) } \chi(\lambda) K^{z+ i\eta_0}(x_0;x_k)\overline{K^{z+ i\eta_0}(x_0;y_k)} \\
\frac{|\alpha_{z+ i\eta_0}(x_0,x_1)|^2}{\zeta_{e_k}^{z+i\eta_0}\zeta_{e_k'}^{\bar{z}-i\eta_0}} \big(\Im \tilg^{z}(\tilde{x}_k, \tilde{y}_k)-\Im \tilg^{z+i\eta_0}(\tilde{x}_k,\tilde{y}_k)\big) \,\dd \lambda\, \Big|\\
\leq \frac{e^{\eta^5}}{\pi N}\left(\int \sum_{(x_0,x_1)\in B} \sum_{(x_2;x_k) ,(y_2;y_k) } \left|K^{z+ i\eta_0}(x_0;x_k)K^{z+ i\eta_0}(x_0;y_k)\right|^p \,\dd \lambda\right)^{1/p}\\
\times \left(\int \sum_{(x_0,x_1)\in B} \sum_{(x_2;x_k) ,(y_2;y_k) } \Big|\frac{|\alpha_{z+ i\eta_0}(x_0,x_1)|^2}{\zeta_{e_k}^{z+i\eta_0}\zeta_{e_k'}^{\bar{z}-i\eta_0}} \Big|^q \,\dd \lambda\right)^{1/q}\\
\times
\left(\int \sum_{(x_0,x_1)\in B} \sum_{(x_2;x_k) ,(y_2;y_k) } \left|\Im \tilg^{z}(\tilde{x}_k,\tilde{y}_k)-\Im \tilg^{z+i\eta_0}(\tilde{x}_k,\tilde{y}_k)\right|^r \,\dd \lambda\right)^{1/r}.
\end{multline*}
Here $\int = \int_{a-2\eta}^{b+2\eta}$. The first sum is bounded by $D^{k-1}\sum_{(x_0;x_k)\in B_k} |K^{z+i\eta_0}(x_0;x_k)|^{2p}$. Assumption \textbf{(Hol)} on $K$ implies that
\[
\sup_{\eta_0, \eta}\limsup_{N\to \infty} \frac1N \int \sum_{(x_0;x_k)\in B_k}  |K^{\lambda+i\eta^4+ i\eta_0}(x_0;x_k))|^{2p} \,\dd \lambda <+\infty \, .
\]
Next, by Remark~\ref{rem:IDS1},
\begin{multline*}
\lim_{N\to \infty} \frac1N \int \sum_{(x_0,x_1)\in B} \sum_{(x_2;x_k) ,(y_2;y_k) } \Big|\frac{|\alpha_{z+ i\eta_0}(x_0,x_1)|^2}{\zeta_{e_k}^{z+i\eta_0}\zeta_{e_k'}^{\bar{z}-i\eta_0}} \Big|^q \,\dd \lambda \\
 = \int \IE\left( \sum_{(x_0;x_k),(y_0;y_k),x_0=y_0=o } \Big|\frac{|\hat{\alpha}_{z+ i\eta_0}(x_0,x_1)|^2}{\hat{\zeta}_{e_k}^{z+i\eta_0}\hat{\zeta}_{e_k'}^{\bar{z}-i\eta_0}} \Big|^q \right)\dd \lambda
\end{multline*}
and the RHS is uniformly bounded in $\eta,\eta_0\in(0,1)$ by Remark~\ref{rem:IDS2}. Remember the convention that objects wearing a hat $\hat{\cdot}$ are defined on the limit $(\cT, \cW)$,
by similar formulas to those on $G_N$. We also refer to \S \ref{s:ident} for notation related to Green functions.

 Finally, again by Remark~\ref{rem:IDS1} we have
\begin{multline*}
\lim_{N\to \infty} \frac1N \int\sum_{(x_0,x_1)\in B}  \sum_{(x_2;x_k) ,(y_2;y_k) } \left|\Im \tilg^{z}(\tilde{x}_k,\tilde{y}_k)-\Im \tilg^{z+i\eta_0}(\tilde{x}_k,\tilde{y}_k)\right|^r  \,\dd \lambda  \\
=\int\IE\left(\sum_{(v_0;v_k),(w_0;w_k),v_0=w_0=o } \left|\Im \cG^z(v_k,w_k)-
\Im \cG^{z+i\eta_0}(v_k,w_k) \right|^r \right) \dd \lambda \, .
\end{multline*}

We check that the RHS vanishes as $\eta,\eta_0\downarrow 0$. Let $X_{\eta}^{\eta_0} = \Im \mathcal{G}^{\lambda+i(\eta^4+\eta_0)}(v_k,w_k) - \Im \mathcal{G}^{\lambda+i\eta^4}(v_k,w_k)$, $X^{\eta_0} = \Im \mathcal{G}^{\lambda+i\eta_0}(v_k,w_k) - \Im \mathcal{G}^{\lambda+i0}(v_k,w_k)$ and $Y_{\eta}^{\eta_0} = X_{\eta}^{\eta_0}-X^{\eta_0}$. Denote $\sum_{v_k,w_k}=\sum_{(v_0;v_k),(w_0;w_k),v_0=w_0=o }$. For any $M>0$, we have $\int \expect \sum_{v_k,w_k} |Y_{\eta}^{\eta_0}|^r = \int \expect \sum_{v_k,w_k} |Y_{\eta}^{\eta_0}|^r 1_{|Y_{\eta}^{\eta_0}| \le M} + \int \expect \sum_{v_k,w_k} |Y_{\eta}^{\eta_0}|^r1_{|Y_{\eta}^{\eta_0}|> M}$.

By Proposition~\ref{lem:limits}, $\sum_{v_k,w_k} |Y_{\eta}^{\eta_0}|^r \to 0$ for Lebesgue-a.e. $\lambda\in \IR$ and $\IP$-a.e. $[\mathcal{T},o,\cW]\in \mathscr{T}_{\ast}^{D,A}$ as $\eta \downarrow 0$. So the first term tends to $0$ by dominated convergence. For the second, for any $s>r$, $\int \expect \sum_{v_k,w_k} |Y_{\eta}^{\eta_0}|^r1_{|Y_{\eta}^{\eta_0}|>M} \le \frac{1}{M^{s-r}} \int \expect \sum_{v_k,w_k} |Y_{\eta}^{\eta_0}|^s \le \frac{C_s}{M^{s-r}}$ by \textbf{(Green)}. This vanishes as $M\to\infty$. Thus, $\int \expect \sum_{v_k,w_k}|Y_{\eta}^{\eta_0}|^r \to 0$ as $\eta \downarrow 0$. Similarly, $\int \expect \sum_{v_k,w_k}|X^{\eta_0}|^r \to 0$ as $\eta_0 \downarrow 0$. Since $|X_{\eta}^{\eta_0}|^r \le 2^{r-1}(|Y_{\eta}^{\eta_0}|^r + |X^{\eta_0}|^r)$, it follows that $\int \expect \sum_{v_k,w_k} |X_{\eta}^{\eta_0}|^r \to 0$ as $\eta \downarrow 0$ followed by $\eta_0 \downarrow 0$.
\end{proof}

\bigskip

By virtue of Lemma~\ref{l:greenco}, denoting $\Psi_{\gamma,v}(w) = \Im \tilg^{\gamma}(v,w)$, the term in parentheses (\ref{eq:greenterms1}) may be replaced by
\begin{equation}      \label{eq:greenterms2}
\left(\frac{\Psi_{z+i\eta_0,\tilde{x}_k}(\tilde{y}_k)}{\zeta_{e_k}^{z+i\eta_0}\zeta_{e_k'}^{\bar{z}-i\eta_0}} - \frac{\Psi_{z+i\eta_0,\tilde{x}_k}(\tilde{y}_{k-1})}{\zeta_{e_k}^{z+i\eta_0}} - \frac{\Psi_{z+i\eta_0,\tilde{x}_{k-1}}(\tilde{y}_k)}{\zeta_{e_k'}^{\bar{z}-i\eta_0}} + \Psi_{z+i\eta_0,\tilde{x}_{k-1}}(\tilde{y}_{k-1}) \right) \, .
\end{equation}
Recall that $e_k=(x_{k-1},x_k)$, $e_k'=(y_{k-1},y_k)$ and that there are non-backtracking paths $(x_0,x_1,\dots,x_{k-1},x_k)$ and $(x_0,x_1,\dots,y_{k-1},y_k)$. Moreover, $\rho_G(x_0) \ge d_{R,\eta,k} \ge k$.

Suppose $e_k' \neq e_k$. Then there is a path $(v_0,\dots,v_s)$ with $v_0=\tilde{x}_k$, $v_1 = \tilde{x}_{k-1}$, $v_{s-1} = \tilde{y}_{k-1}$ and $v_s = \tilde{y}_k$. Taking the complex conjugate in identity \eqref{eq:psiiden2}, noting that $\Psi_{z+i\eta_0,v}(w)$ is real, we see that \eqref{eq:greenterms2} is zero. If $e_k=e'_k$, \eqref{eq:psiiden1} tells us \eqref{eq:greenterms2} equals $\frac{|\Im \zeta_{x_{k-1}}^{z+i\eta_0}(x_k)|}{|\zeta_{x_{k-1}}^{z+i\eta_0}(x_k)|^2}$.

Since $\rho_G(x_0) \ge k$ in Lemma~\ref{lem:BSTvar}, the paths $(x_0, x_1, x_2,\cdots, x_k) $ and $(x_0, x_1, y_2,  \cdots, y_k) $ are determined by $e_k$ and $e'_k$, respectively. So the terms in the sum are only nonzero if $(x_0, x_1, x_2,\cdots, x_k)=(x_0, x_1, y_2,  \cdots, y_k) $. Hence, if we make all replacements in Lemmas~\ref{l:errors} and \ref{l:greenco}, modulo the errors appearing in these lemmas, the expression \eqref{eq:intr} finally takes the form
\begin{align*}
& \frac{1}{\pi N} \int_{a-2\eta}^{b+2\eta} \sum_{\rho_G(x_0) \ge d_{R,k,\eta}} \sum_{x_1\sim x_0} \sum_{(x_2;x_k) } \chi(\lambda) |\alpha_{z+ i\eta_0}(x_0,x_1)|^2|K^{z+ i\eta_0}(x_0;x_k)|^2 \\
& \qquad \cdot \frac{|\Im \zeta_{x_{k-1}}^{z+i\eta_0}(x_k)|}{|\zeta_{x_{k-1}}^{z+i\eta_0}(x_k)|^2} \,\dd \lambda \le \frac{1}{\pi} \int_{a-2\eta}^{b+2\eta} \|K^{z+i\eta_0}\|_{z+i\eta_0}^2\,\dd \lambda \, ,
\end{align*}
where we used that $\chi(\lambda) \le 1$ on $\R$. Collecting all estimates on the error terms, taking $N\to \infty$, then $\eta \downarrow 0$, then $\eta_0 \downarrow 0$, then $R\to \infty$, we finally get $\frac{1}{N}\sum_{j=1}^N \chi(\lambda_j)\|\alpha_{\gamma_j}K_B^{\gamma_j}f_j\|^2 \lesssim \frac{1}{\pi} \int_{a-2\eta}^{b+\eta} \|K^{z+i\eta_0}\|_{z+i\eta_0}^2\,\dd \lambda$. Recalling \eqref{e:very1st}, if we prove \eqref{e:flim'}, then this will complete the proof of Theorem~\ref{thm:upvar}.

We have $\|\overline{\alpha_{\gamma_j}}^{-1}f_j^{\ast}\|^2=\sum_{(x_0,x_1)\in B} \frac{1}{|\Im \zeta_{x_1}^{\gamma_j}(x_0)|}|\psi_j(x_0)-\zeta_{x_1}^{\gamma_j}(x_0)\psi_j(x_1)|^2$. Repeating the same arguments, we see that modulo asymptotically vanishing error terms, we have
\begin{multline*}
\frac{1}{N}\sum_{\lambda_j\in I} \|\overline{\alpha_{\gamma_j}}^{-1}f_j^{\ast}\|^2 \lesssim \frac{3}{\pi N}\int_{a-2\eta}^{b+2\eta} \sum_{\rho_G(x_0)\ge d_{R,\eta}} \sum_{x_1\sim x_0} \frac{\chi(\lambda)}{|\Im \zeta_{x_1}^{z+i\eta_0}(x_0)|} \\
\cdot \big[ \Psi_{z+i\eta_0,\tilde{x}_0}(\tilde{x}_0) - \zeta_{x_1}^{z+i\eta_0}(x_0)\Psi_{z+i\eta_0,\tilde{x}_1}(\tilde{x}_0) -  \overline{\zeta_{x_1}^{z+i\eta_0}(x_0)}\Psi_{z+i\eta_0,\tilde{x}_0}(\tilde{x}_1) \\
+ |\zeta_{x_1}^{z+i\eta_0}(x_0)|^2\Psi_{z+i\eta_0,\tilde{x}_1}(\tilde{x}_1)\big]\,\dd\lambda \,.
\end{multline*}
The term in square brackets is just $|\Im \zeta_{x_1}^{z+i\eta_0}(x_0)|$ by \eqref{eq:psiiden1}. Hence, using $\chi(\lambda)\le 1$ we get $\frac{1}{N}\sum_{\lambda_j\in I} \|\overline{\alpha_{\gamma_j}}^{-1}f_j^{\ast}\|^2 \lesssim \frac{3(|I|+4\eta)D}{\pi}$ for any small $\eta>0$, and \eqref{e:flim'} follows.

\section{Step 2~: Invariance property of the quantum variance\label{s:inv}}

In the scheme of \S\ref{s:outline}, we are now in Step 2~: using the functional equations \eqref{e:newef} and \eqref{e:newef2} satisfied by $f_j, f_j^*$, we show that there are certain transformations $\mathcal{R}_{n,r}^{\gamma}:\mathscr{H}_k=\IC^{B_k}\to \mathscr{H}_{n+k}=\IC^{B_{n+k}}$ that leave the quantum variance \eqref{e:varnb} unchanged.

Recall from Section \ref{s:nb} that $\cB (\zeta^{\gamma_j} f_j) = f_j - i\eta_0\,\tau_+\psi_j$ and $\cB^{\ast} (\iota \zeta^{\gamma_j}f_j^{\ast}) = f_j^{\ast}-i\eta_0\,\tau_-\psi_j$ if $\gamma_j=\lambda_j+i\eta_0$. So
\[
(\cB \zeta^{\gamma_j})^2f_j = \cB \zeta^{\gamma_j}f_j - i\eta_0\cB \zeta^{\gamma_j}\tau_+\psi_j = f_j - i\eta_0(I+\cB \zeta^{\gamma_j})\tau_+\psi_j \, .
\]
Iterating $r$ times,
\[
(\cB \zeta^{\gamma_j})^rf_j = f_j-i\eta_0 \sum_{t=0}^{r-1}(\cB \zeta^{\gamma_j})^t\tau_+\psi_j \, .
\]
Similarly
\[
(\cB^{\ast} \iota \zeta^{\gamma_j})^{n-r}f_j^{\ast} = f_j^{\ast} - i\eta_0 \sum_{t'=0}^{n-r-1} (\cB^{\ast}  \iota \zeta^{\gamma_j})^{t'}\tau_-\psi_j \, .
\]

If we define for $r\le n$ and $\gamma\in \IC\setminus \IR$ the operator $\mathcal{R}_{n,r}^{\gamma}:\mathscr{H}_k\to \mathscr{H}_{n+k}$ by
\begin{multline*}
(\mathcal{R}_{n,r}^{\gamma}K)(x_0;x_{n+k}) = \overline{\zeta_{x_1}^{\gamma}(x_0)\zeta_{x_2}^{\gamma}(x_1)\cdots\zeta_{x_{n-r}}^{\gamma}(x_{n-r-1})} K(x_{n-r};x_{n-r+k}) \\
 \cdot\zeta_{x_{n-r+k}}^{\gamma}(x_{n-r+k+1})\zeta_{x_{n-r+k+1}}^{\gamma}(x_{n-r+k+2})\cdots\zeta_{x_{n+k-1}}^{\gamma}(x_{n+k}) \, ,
\end{multline*}
we thus get
\begin{align*}
\langle f_j^{\ast},(\mathcal{R}_{n,r}^{\gamma_j}K)_B f_j\rangle & = \sum_{(x_{n-r};x_{n-r+k})} \overline{\left[(\cB^{\ast}\iota \zeta^{\gamma_j})^{n-r}f_j^{\ast}\right](x_{n-r},x_{n-r+1})} K(x_{n-r};x_{n-r+k}) \\
& \qquad \qquad \qquad \qquad \cdot \left[(\cB \zeta^{\gamma_j})^rf_j\right](x_{n-r+k-1},x_{n-r+k}) \\
& = \left\langle (\cB^{\ast} \iota\zeta^{\gamma_j})^{n-r}f_j^{\ast}, K_B (\cB \zeta^{\gamma_j})^rf_j \right\rangle = \langle f_j^{\ast}, K_B f_j \rangle - \mathcal{E}_{n,r, j}(\eta_0,K) \, ,
\end{align*}
where the $\mathcal{E}$ stands for an ``error term'' that should vanish as $\eta_0 \downarrow 0$~:
\begin{align*}
\mathcal{E}_{n,r, j}(\eta_0,K) & = i\eta_0 \sum_{t=0}^{r-1} \langle f_j^{\ast},K_B(\cB \zeta^{\gamma_j})^t\tau_+\psi_j\rangle + i\eta_0 \sum_{t'=0}^{n-r-1}\langle (\cB^{\ast} \iota \zeta^{\gamma_j})^{t'}\tau_-\psi_j, K_B f_j\rangle \\
& \quad + \eta_0^2\sum_{t=0}^{r-1}\sum_{t'=0}^{n-r-1}\langle (\cB^{\ast} \iota \zeta^{\gamma_j})^{t'}\tau_-\psi_j, K_B(\cB \zeta^{\gamma_j})^t\tau_+\psi_j\rangle  \, .
\end{align*}

Since this holds for each $1 \le r \le n$ and $K=K^{\gamma}$, we get by the triangle inequality
\begin{equation}\label{e:R}
\varnbi(K^{\gamma}) \le  \varnbi\Big(\frac{1}{n}\sum_{r=1}^n\mathcal{R}_{n,r}^{\gamma}K^{\gamma}\Big) + \frac{1}{N} \sum_{\lambda_j\in I} \Big| \frac{1}{n}\sum_{r=1}^n \mathcal{E}_{n,r, j}(\eta_0,K^{\gamma})\Big| \, .
\end{equation}
  
We first show that the latter term may be neglected.

\begin{lem}   \label{l:morerem}
Suppose $K^\gamma\in \mathscr{H}_k$ satisfies assumptions \emph{\textbf{(Hol)}} and let $\bar{I}\subseteq I_1$. Then for all $n\in\IN$,
\[
\lim_{\eta_0\downarrow 0} \limsup_{N\to \infty} \bigg( \frac{1}{N}\sum_{\lambda_j\in I}  \Big| \frac{1}{n}\sum_{r=1}^n \mathcal{E}_{n,r, j}(\eta_0,K^{\gamma})\Big|\bigg)^2 =0 \, .
\]
\end{lem}
\begin{proof}
We have $\big(\frac{1}{N} \sum_{\lambda_j\in I} |\frac{1}{n} \sum_{r=1}^n \mathcal{E}_{n,r,j} |\big)^2 \le \frac{1}{n} \sum_{r=1}^n \big(\frac{1}{N}\sum_{\lambda_j\in I} |\mathcal{E}_{n,r,j}|\big)^2$. Now, letting as above $\gamma_j=\lambda_j+i\eta_0$,
\begin{multline*}
 \bigg(\sum_{\lambda_j\in I}|\mathcal{E}_{n,r,j}|\bigg)^2 \le \eta^2_0 c_{n,r}\bigg\{\sum_{t=0}^{r-1} \Big(\sum_{\lambda_j\in I} \left|\left\langle f_j^{\ast},K_B^{\gamma_j}(\mathcal{B} \zeta^{\gamma_j})^t\tau_+\psi_j\right\rangle\right|\Big)^2\\
 \qquad\qquad\qquad + \sum_{t'=0}^{n-r-1} \Big(\sum_{\lambda_j\in I} \left|\left\langle (\mathcal{B}^{\ast}\iota \zeta^{\gamma_j}  )^{t'}\tau_-\psi_j, K_B^{\gamma_j} f_j\right\rangle\right|\Big)^2 \\
+ \eta^2_0\sum_{t,t'} \Big(\sum_{\lambda_j\in I} \left|\left\langle(\mathcal{B}^{\ast}\iota \zeta^{\gamma_j}  )^{t'}\tau_-\psi_j, K_B^{\gamma_j}(\mathcal{B} \zeta^{\gamma_j})^t\tau_+\psi_j\right\rangle\right|\Big)^2\bigg\} \, ,
\end{multline*}
where $c_{n,r}=n+r(n-r)$. So it suffices to show that $\limsup_N \big(\frac{1}{N}\sum_{\lambda_j\in I} |\langle \cdot, \cdot \rangle| \big)^2$ is uniformly bounded in $\eta_0$ for each $t,t'$. For the first term, we have
\begin{align*}
& \Big(\frac{1}{N} \sum_{\lambda_j \in I} |\langle f_j^{\ast},K_B^{\gamma_j}(\mathcal{B} \zeta^{\gamma_j})^t\tau_+\psi_j\rangle|\Big)^2  \le \frac{1}{N} \sum_{\lambda_j\in I} \|\overline{\alpha_{\gamma_j}}^{-1}f_j^{\ast}\|^2\cdot \frac{1}{N}\sum_{\lambda_j\in I}\|\alpha_{\gamma_j}K_B^{\gamma_j}(\mathcal{B} \zeta^{\gamma_j})^t\tau_+\psi_j\|^2\, .
\end{align*}
The first sum is uniformly bounded as $\eta_0\downarrow 0$, by (\ref{e:flim'}). Next, by (\ref{eq:KB2}), we have
\begin{multline*}
\|\alpha_{\gamma_j}K_B^{\gamma_j}(\mathcal{B} \zeta^{\gamma_j})^t\tau_+\psi_j\|^2   = \sum_{(x_0,x_1)\in B}\sum_{(x_2;x_k) ,(y_2;y_k) } |\alpha_{\gamma_j}(x_0,x_1)|^2K^{\gamma_j}(x_0;x_k) \\
\cdot \overline{K^{\gamma_j}(x_0;y_k)}\cdot[(\mathcal{B} \zeta^{\gamma_j})^t\tau_+\psi_j](x_{k-1},x_k)\overline{[(\mathcal{B} \zeta^{\gamma_j})^t\tau_+\psi_j](y_{k-1},y_k)} \, .
\end{multline*}
Arguing as in Section~\ref{s:proof1}, applying Lemmas~\ref{l:intseg} to ~\ref{l:errors}, we get for $z=\lambda+i\eta^4$,
\begin{multline*}
\frac{1}{N}\sum_{\lambda_j\in I} \|\alpha_{\gamma_j}K_B^{\gamma_j}(\cB \zeta^{\gamma_j})^t\tau_+\psi_j\|^2 \lesssim \frac{3}{\pi N} 
\int_{a-2\eta}^{b+2\eta} \sum_{\rho_G(x_0) \ge d_{R,k,t,\eta}} \sum_{x_1\sim x_0} \sum_{(x_2;x_{k+t}) ,(y_2;y_{k+t}) }  \\
\chi(\lambda) |\alpha_{z+i\eta_0}(x_0,x_1)|^2K^{z+i\eta_0}(x_0;x_k)\overline{K^{z+i\eta_0}(x_0;y_k)} \\
\zeta_{x_k}^{z+i\eta_0}(x_{k+1})\cdots \zeta_{x_{k+t-1}}^{z+i\eta_0}(x_{k+t}) \overline{\zeta_{y_k}^{z+i\eta_0}(y_{k+1})\cdots \zeta_{y_{k+t-1}}^{z+i\eta_0}(y_{k+t})} \Psi_{z,\tilde{x}_{k+t}}(\tilde{y}_{k+t})\,\dd \lambda\,.
\end{multline*}
Using H\"older's inequality as in Lemma~\ref{l:greenco}, we see that as $N \to \infty$, this quantity is uniformly bounded in $\eta,\eta_0$ by \textbf{(Hol)} and \textbf{(Green)}.
One bounds $\frac{1}{N}\sum_{\lambda_j}\|K_B^{\gamma_j}f_j\|^2$ similarly. Finally,
\begin{multline*}
\frac{1}{N}\sum_{\lambda_j \in I} \|(\cB^{\ast}\iota \zeta^{\gamma_j})^{t'}\tau_-\psi_j\|^2 = \frac{1}{N}\sum_{\lambda_j\in I}\sum_{(x_0;x_{t'+1})} |\psi_j(x_0)|^2 |\zeta_{x_1}^{\gamma_j}(x_0)\dots\zeta_{x_{t'-1}}^{\gamma_j}(x_{t'})|^2 \\
 \lesssim \frac{3}{\pi N} \int_{a-2\eta}^{b+2\eta} \sum_{(x_0;x_{t'+1}),\rho_G(x_0)\ge d_{R,\eta,t'}} \chi(\lambda) \Psi_{z,\tilde{x}_0}(\tilde{x}_0)|\zeta_{x_1}^{z+i\eta_0}(x_0)\dots\zeta_{x_{t'-1}}^{z+i\eta_0}(x_{t'})|^2\,\dd\lambda \,,
\end{multline*}
which is asymptotically bounded using H\"older's inequality again as in Lemma~\ref{l:greenco}.
\end{proof}

Using the invariance law \eqref{e:R}, Theorem~\ref{thm:upvar} with $\tilde{K}^{\gamma} = \frac{1}{n}\sum_{r=1}^n \mathcal{R}_{n,r}^{\gamma}K^{\gamma}$, and Lemma~\ref{l:morerem}, we deduce the following statement~:
 
\begin{prp}\label{p:mainbound2}
Under the assumptions of Theorem~\ref{thm:upvar},
\begin{multline*}
\lim_{\eta_0 \downarrow 0}\limsup_{N\to +\infty}\varnbi(K^{\gamma})^2
\\ \leq D\,|I|\,\lim_{\eta_0\downarrow 0}\lim_{\eta\downarrow 0}\limsup_{N\to \infty}  \int_{a-2\eta}^{b+2\eta} \Big\| \frac{1}{n}\sum_{r=1}^n\mathcal{R}_{n,r}^{\lambda+i(\eta^4+\eta_0)}K^{\lambda+i(\eta^4+\eta_0)} \Big\|_{\lambda+i(\eta^4+\eta_0)}^2\,\dd \lambda \, .
\end{multline*}
 
\end{prp}

\section{Step 3~: A stationary Markov chain appears\label{s:Mark}}

Denoting $\gamma=\lambda+i(\eta^4+\eta_0)$ in Proposition \ref{p:mainbound2}, we are now concerned with estimating
\begin{equation}\label{e:lastline}
\left\Vert \frac{1}{n}\sum_{r=1}^n\mathcal{R}_{n,r}^{\gamma}K^{\gamma} \right\Vert_{\gamma}^2 =
\frac{1}{n^2}\sum_{r, r'=1}^n \left\langle \mathcal{R}_{n,r}^{\gamma}K^{\gamma},\mathcal{R}_{n,r'}^{\gamma}K^{\gamma}\right\rangle_{\gamma} \, .
 \end{equation}

Suppose $r \ge r'$, so that $n-r\le n-r'$. Then
\begin{align*}
& \langle \mathcal{R}_{n,r}^{\gamma}K,\mathcal{R}_{n,r'}^{\gamma}K\rangle_{{\gamma}} =  \frac{1}{N}\sum_{(x_0;x_{n+k})\in B_{n+k}}  \frac{|\Im \zeta_{x_1}^{\gamma}(x_0)|}{|\zeta_{x_1}^{\gamma}(x_0)|^2} \cdot  |\zeta_{x_1}^{\gamma}(x_0)\cdots \zeta_{x_{n-r}}^{\gamma}(x_{n-r-1})|^2 \\
& \qquad \cdot|\zeta_{x_{n-r'+k}}^{\gamma}(x_{n-r'+k+1})\cdots\zeta_{x_{n+k-1}}^{\gamma}(x_{n+k})|^2  \\
& \qquad \cdot \overline{K(x_{n-r};x_{n-r+k})\,\zeta_{x_{n-r+k}}^{\gamma}(x_{n-r+k+1})\cdots \zeta_{x_{n-r'+k-1}}^{\gamma}(x_{n-r'+k})} \\
& \qquad \cdot \overline{\zeta_{x_{n-r+1}}^{\gamma}(x_{n-r})\cdots \zeta_{x_{n-r'}}^{\gamma}(x_{n-r'-1})}K(x_{n-r'};x_{n-r'+k}) \cdot \frac{|\Im \zeta_{x_{n+k-1}}^{\gamma}(x_{n+k})|}{|\zeta_{x_{n+k-1}}^{\gamma}(x_{n+k})|^2} \, .
\end{align*}
Letting ${\eta}_1=\Im \gamma$, (\ref{eq:sumzeta}) tells us that $\sum_{x_0 \in \mathcal{N}_{x_1}\setminus \{x_2\}} |\Im \zeta_{x_1}^{\gamma}(x_0)| = \frac{|\Im \zeta_{x_2}^{\gamma}(x_1)|}{|\zeta_{x_2}^{\gamma}(x_1)|^2} - \eta_1$. Similarly, we have $\sum_{x_{n+k}\in \mathcal{N}_{x_{n+k-1}} \setminus \{x_{n+k-2}\}} |\Im \zeta_{x_{n+k-1}}^{\gamma}(x_{n+k})| = \frac{|\Im \zeta_{x_{n+k-2}}^{\gamma}(x_{n+k-1})|}{|\zeta_{x_{n+k-2}}^{\gamma}(x_{n+k-1})|^2} - \eta_1$.
By iteration, this induces some simplifications~:
\begin{multline}\label{e:simplifi}
\langle \mathcal{R}_{n,r}^{\gamma}K,\mathcal{R}_{n,r'}^{\gamma}K\rangle_{{\gamma}} =  \frac{1}{N} \sum_{(x_{n-r};x_{n-r'+k})\in B_{k+r-r'}} \frac{|\Im \zeta_{x_{n-r+1}}^{\gamma}(x_{n-r})|}{|\zeta_{x_{n-r+1}}^{\gamma}(x_{n-r})|^2}\overline{K(x_{n-r};x_{n-r+k})} \\
 \qquad \cdot K(x_{n-r'};x_{n-r'+k}) \cdot \overline{\zeta_{x_{n-r+k}}^{\gamma}(x_{n-r+k+1})\cdots \zeta_{x_{n-r'+k-1}}^{\gamma}(x_{n-r'+k})} \\
 \qquad \cdot  \overline{\zeta_{x_{n-r+1}}^{\gamma}(x_{n-r})\cdots \zeta_{x_{n-r'}}^{\gamma}(x_{n-r'-1})}  \cdot \frac{|\Im \zeta_{x_{n-r'+k-1}}^{\gamma}(x_{n-r'+k})|}{|\zeta_{x_{n-r'+k-1}}^{\gamma}(x_{n-r'+k})|^2} - \mathbf{E}_{n,r,r'}(\eta_1,K) \, ,
\end{multline}
with the error term
\begin{align*}
\mathbf{E}_{n,r,r'}(\eta_1,K) & = \frac{\eta_1}{N} \sum_{s=1}^{n-r} \sum_{(x_s;x_{n+k})}|\zeta_{x_{s+1}}^{\gamma}(x_s)\cdots\zeta_{x_{n-r}}^{\gamma}(x_{n-r-1})|^2 \\
& \qquad \cdot|\zeta_{x_{n-r'+k}}^{\gamma}(x_{n-r'+k+1})\cdots\zeta_{x_{n+k-2}}^{\gamma}(x_{n+k-1})|^2 \cdot |\Im \zeta_{x_{n+k-1}}^{\gamma}(x_{n+k})|  \\
& \qquad \cdot \overline{K(x_{n-r};x_{n-r+k})\,\zeta_{x_{n-r+k}}^{\gamma}(x_{n-r+k+1})\cdots \zeta_{x_{n-r'+k-1}}^{\gamma}(x_{n-r'+k})}  \\
& \qquad \cdot \overline{\zeta_{x_{n-r+1}}^{\gamma}(x_{n-r})\cdots \zeta_{x_{n-r'}}^{\gamma}(x_{n-r'-1})}K(x_{n-r'};x_{n-r'+k})  \\
& \quad + \frac{\eta_1}{N} \sum_{s'=n-r'+k}^{n+k-1} \sum_{(x_{n-r'};x_{s'})} \frac{|\Im \zeta_{x_{n-r+1}}^{\gamma}(x_{n-r})|}{|\zeta_{x_{n-r+1}}^{\gamma}(x_{n-r})|^2} \\
& \quad \qquad \cdot |\zeta_{x_{n-r'+k}}^{\gamma}(x_{n-r'+k+1})\cdots \zeta_{x_{s'-1}}^{\gamma}(x_{s'})|^2 \\
& \quad \qquad \cdot \overline{K(x_{n-r};x_{n-r+k})\,\zeta_{x_{n-r+k}}^{\gamma}(x_{n-r+k+1})\cdots \zeta_{x_{n-r'+k-1}}^{\gamma}(x_{n-r'+k})} \\
& \quad \qquad \cdot \overline{\zeta_{x_{n-r+1}}^{\gamma}(x_{n-r})\cdots \zeta_{x_{n-r'}}^{\gamma}(x_{n-r'-1})}K(x_{n-r'};x_{n-r'+k}) \, .
\end{align*}
%For $s=n-r$, $s'=n-r+k$, we put respectively $\zeta_{x_{n-r+1}}^{\gamma}(x_{n-r})\zeta_{x_{n-r}}^{\gamma}(x_{n-r-1}):=1$ and $\zeta_{x_{n-r+k}}^{\gamma}(x_{n-r+k+1})\zeta_{x_{n-r+k-1}}^{\gamma}(x_{n-r+k}) :=1$.

The expression is slightly nicer if we replace $K$ by $Z_{\gamma} K$ defined by
\begin{equation}\label{e:ZK}
(Z_{\gamma}K)(x_0;x_k) = \zeta_{x_0}^{\gamma}(x_1)\cdots \zeta_{x_{k-1}}^{\gamma}(x_k) K(x_0;x_k) \, .
\end{equation}
If $\gamma\mapsto K^\gamma$ satisfies \textbf{(Hol)} then so does $\gamma\mapsto Z_\gamma K^\gamma$. Using (\ref{eq:mv}), we get in that case
\begin{multline}
\langle \mathcal{R}_{n,r}^{\gamma}Z_{\gamma}K^\gamma,\mathcal{R}_{n,r'}^{\gamma}Z_{\gamma}K^\gamma\rangle_{{\gamma}} =\frac{1}{N} \sum_{(x_{n-r};x_{n-r'+k})\in B_{k+r-r'}} \frac{|\Im \zeta_{x_{n-r+1}}^{\gamma}(x_{n-r})|}{|m_{x_{n-r+1}}^{\gamma}|^2 |\zeta_{x_{n-r}}^{\gamma}(x_{n-r+1})|^2} \\
 \qquad \cdot  |\zeta_{x_{n-r}}(x_{n-r+1})\cdots \zeta_{x_{n-r'+k-1}}(x_{n-r'+k})|^2\overline{m_{x_{n-r}}^{\gamma}K^{\gamma}(x_{n-r};x_{n-r+k})}  \\
 \qquad \cdot m^\gamma_{x_{n-r'}} K^{\gamma}(x_{n-r'};x_{n-r'+k}) \cdot   {u_{x_{n-r+1}}^{\gamma}(x_{n-r})\cdots u_{x_{n-r'}}^{\gamma}(x_{n-r'-1})}   \\
 \qquad  \cdot \frac{|\Im \zeta_{x_{n-r'+k-1}}^{\gamma}(x_{n-r'+k})|}{|\zeta_{x_{n-r'+k-1}}^{\gamma}(x_{n-r'+k})|^2} - \mathbf{E}_{n,r,r'}(\eta_1,Z_{\gamma}K^\gamma) \, ,
 \, \label{e:phew}
\end{multline}
where $u_x^{\gamma}(y)$ is the complex number of modulus $1$ given by
\begin{equation}\label{e:ugamma}
u_x^{\gamma}(y)=\overline{\zeta_{x}^{\gamma}(y)}\zeta_{x}^{\gamma}(y)^{-1}  \, .
\end{equation}
Let us define a positive measure $\mu^\gamma_k$ on the set $B_k$ of non-backtracking paths of length $k$, by putting
\begin{equation}\label{e:muk}
\mu^\gamma_k\left[(x_0; x_k)\right] = \frac{|\Im \zeta_{x_{1}}^{\gamma}(x_{0})|}{|m^\gamma_{x_{1}} \zeta_{x_{0}}^{\gamma}(x_{1})|^2}  \cdot  |\zeta_{x_{0}}(x_{1})\cdots \zeta_{x_{k-1}}(x_{k})|^2 \cdot 
\frac{|\Im \zeta_{x_{k-1}}^{\gamma}(x_{k})|}{|\zeta_{x_{k-1}}^{\gamma}(x_{k})|^2} \, .
\end{equation}
Let us also introduce the operator
\begin{equation}\label{e:Su}
(\mathcal{S}_{u^\gamma} K)(x_0;x_k) = \frac{|\zeta^\gamma_{x_1}(x_0)|^2}{|\Im \zeta^\gamma_{x_1}(x_0)|}\sum_{x_{-1}\in \mathcal{N}_{x_0} \setminus \{x_1\}} |\Im \zeta^\gamma_{x_0}(x_{-1})| \,\overline{u^\gamma_{x_0}(x_{-1})} K(x_{-1};x_{k-1}) \, . 
\end{equation}
Then, using (\ref{eq:mv}) again, we see that \eqref{e:phew}  takes the nicer form
\begin{equation}\label{e:nicer}
\langle \mathcal{R}_{n,r}^{\gamma}Z_{\gamma}K^\gamma,\mathcal{R}_{n,r'}^{\gamma}Z_{\gamma}K^\gamma\rangle_{{\gamma}}  = \frac{1}{N} \la \cS_{u^\gamma}^{{r-r'}}m^\gamma K^{\gamma}, m^\gamma K^{\gamma}\ra_{\ell^2(\mu^\gamma_k)} - \mathbf{E}_{n,r,r'}(\eta_1,Z_{\gamma}K^\gamma) \, ,
\end{equation} where we let
$(m^\gamma K)(x;y)=m^\gamma_x K(x;y)$. Let us also define
 \begin{equation}\label{e:Sgamma}
(\mathcal{S}_\gamma K)(x_0;x_k) = \frac{|\zeta^\gamma_{x_1}(x_0)|^2}{|\Im \zeta^\gamma_{x_1}(x_0)|}\sum_{x_{-1}\in \mathcal{N}_{x_0} \setminus \{x_1\}} |\Im \zeta^\gamma_{x_0}(x_{-1})|  \,K(x_{-1};x_{k-1}) \, .
\end{equation}
Such operators would be called ``transfer operators'' in ergodic theory, or ``transition matrices'' in the theory of Markov chains. Note that $\mathcal{S}_\gamma$ has non-negative coefficients and that $\mathcal{S}_{u^\gamma}$ just differs from $\mathcal{S}_\gamma$
by the ``phases'' $\overline{u^\gamma_{x_0}(x_{-1})}$. The effect of adding a phase to a stochastic operator is a much studied subject in the theory of Markov chains, or more generally in ergodic theory (see Wielandt's theorem \cite[Chapter 8]{Mey01}, or in the context of hyperbolic dynamical systems \cite[Chapter 4]{PP}).

The matrix elements of $\cS_\gamma$ are given by 
\begin{equation}\label{e:sgamma}\cS_\gamma (\omega,\omega') =\frac{|\zeta_{x_1}^{\gamma}(x_0)|^2}{|\Im \zeta_{x_1}^{\gamma}(x_0)|}  |\Im \zeta_{x_0}^{\gamma}(x_{-1})|
\end{equation}
if $\omega=(x_0;x_k)$, $\omega'=(x_{-1};x_{k-1})$ and $\omega'\rightsquigarrow \omega$, and $\cS_\gamma (\omega,\omega') =0$ otherwise. Recall from \S \ref{sec:basic} that if $\omega=(x_0;x_k)$, we write $\omega' \rightsquigarrow \omega$ if $\omega'=(x_{-1},x_0,\dots,x_{k-1})$ for some $x_{-1}\in \mathcal{N}_{x_0}\setminus \{x_1\}$. 

Note that $\cS_\gamma$ is substochastic~: $\sum_{\omega'\in B_k}\cS_\gamma (\omega,\omega') \le 1$ for any $\omega \in B_k$, by \eqref{eq:sumzeta}. More precisely, if $\omega=(x_0;x_k)$ and $\eta_1=\Im \gamma >0$, then
\begin{equation}\label{e:sumzeta2}
\sum_{\omega'\in B_k}\cS_\gamma (\omega,\omega') = 1-\eta_1\frac{|\zeta_{x_1}^{\gamma}(x_0)|^2}{|\Im \zeta_{x_1}^{\gamma}(x_0)|} \, .
\end{equation}
Taking the adjoint in $\ell^2(\mu_k^{\gamma})$, a direct calculation gives
\[
(\mathcal{S}_{\gamma}^{\ast}K)(x_0;x_k) = \frac{|\zeta_{x_{k-1}}^{\gamma}(x_k)|^2}{|\Im \zeta_{x_{k-1}}^{\gamma}(x_k)|} \sum_{x_{k+1} \in \mathcal{N}_{x_k} \setminus \{x_{k-1}\}} |\Im \zeta_{x_k}^{\gamma}(x_{k+1})| \,K(x_1;x_{k+1}) \, .
\]
 The adjoint $\cS_{\gamma}^{\ast}$ is also substochastic, with
\begin{equation}
\label{e:sumzeta3}
\sum_{\omega'\in B_k}\cS_{\gamma}^{\ast} (\omega,\omega') = 1-\eta_1 \frac{|\zeta_{x_{k-1}}^{\gamma}(x_k)|^2}{|\Im \zeta_{x_{k-1}}^{\gamma}(x_k)|}.
\end{equation}

\begin{rem}\label{r:compat}
By \eqref{eq:sumzeta}, for any $(x_0;x_{k-1})\in B_{k-1}$, we have
\begin{equation}
\sum_{x_k\in \mathcal{N}_{x_{k-1}}\setminus \{x_{k-2}\}} \mu^\gamma_k\left[(x_0;x_k)\right]\leq  \mu^\gamma_{k-1}\left[(x_0;x_{k-1})\right]\label{e:compat}
\end{equation}
and
for any $(x_1;x_k)\in  B_{k-1}$,
\begin{equation}
\sum_{x_0\in \mathcal{N}_{x_1}\setminus \{x_2\}} \mu^\gamma_k\left[(x_0;x_k)\right]\leq   \mu^\gamma_{k-1}\left[(x_1; x_{k}) \right]\label{e:inv}
\end{equation}

In \eqref{e:lastline} we take $\gamma=\lambda+i(\eta^4+\eta_0)$ (c.f. Proposition~\ref{p:mainbound2}), and thus $\eta_1=\Im \gamma = \eta^4+\eta_0$. In the limiting case $\eta_1=0$, \eqref{e:compat} and \eqref{e:inv} turn into equalities. Equation \eqref{e:compat}
is then the Kolmogorov compatibility condition~: it tells us that the family of measures $(\mu_k^\gamma)$ may be extended to a positive measure (actually, a Markov measure) on the set $B_\infty$ of infinite non-backtracking paths. Equality in condition \eqref{e:inv} means that this Markov chain is stationary. This stationarity is the property that makes the measures $\mu_k^\gamma$ nice, and this is the reason for introducing (somewhat artificially) the weight $\frac{\Im \zeta_x^{\gamma}(y)}{|\zeta_x^{\gamma}(y)|^2}$ in \eqref{eq:normgamma}.

This family of stationary Markov chains (indexed by $\gamma$) is in some sense the ``classical dynamical system'' that we were seeking in \S \ref{s:outline}.

Since $\eta_1=\eta^4+\eta_0$ is non-zero (but small), we do not actually have exact equality in \eqref{e:compat} and \eqref{e:inv}. This causes some error terms that we need to control as $\eta, \eta_0\To 0$.
\end{rem}

\section{Spectral gap and mixing\label{s:mixing}}

In this section, we convert the expanding assumption \textbf{(EXP)} into an estimate on the rate of mixing of the ``Markov chains'' $(\mu_k^\gamma)$ defined in \eqref{e:muk}. Every transitive Markov chain is mixing, but here we need
 estimates that are uniform both as $N\To +\infty$ {\em{and}} as $\gamma$ approaches the real axis.
 
 A technical difficulty is that the measures $(\mu_k^\gamma)$ are not {\em{a priori}} bounded from above, and the transition probabilities are not bounded from below as $\gamma$ approaches the real axis. Peaks of $(\mu_k^\gamma)$, as well as small transition probabilities, tend to ``disconnect'' the graph and are bad for mixing. So we will need to show that there are few peaks and few small transitions (Proposition \ref{lem:bad1}).
 
Let 
\begin{equation}\label{e:nuk}\nu_k^{\gamma} = \frac{1}{\mu_k^{\gamma}(B_k)}\mu_k^{\gamma}
\end{equation}
be the normalized measure. We denote by $\ell^2(\nu_k^{\gamma})$ the set $\ell^2(B_k)$ endowed with the scalar product $\langle f,g\rangle_{\nu_k^{\gamma}} = \sum_{\omega\in B_k} \nu_k^{\gamma}(\omega)\overline{f(\omega)}g(\omega)$.

We anticipate the calculations of Section \ref{sec:retour}, where we will need to consider the non-backtracking quantum variance of operators $K_\gamma$ of the form $K_\gamma= \cF_\gamma K$ where $K$ is independent of $\gamma$, and $\cF_\gamma:\mathscr{H}_m\to\mathscr{H}_k$ is a $\gamma$-dependent operator for some $1\le k\le m+1$, having the form $\cF_\gamma=\mathcal{L}^{\gamma}d^{-1}\mathcal{S}_{T,\gamma}$, $\mathcal{T}^\gamma$, $\mathcal{O}_1^\gamma$, $\mathcal{U}^{\gamma}_j$, $\mathcal{O}_j^\gamma$, $\mathcal{P}_j^\gamma$, $j\ge 2$, or a polynomial combination thereof. See (\ref{e:ST}, \ref{e:cL}, \ref{e:T}, \ref{e:Um}, \ref{e:Om}, \ref{e:Pm}) for the definitions. In the case $\cF_\gamma=\mathcal{L}^{\gamma}d^{-1}\mathcal{S}_{T,\gamma}$, the operator depends on an additional parameter $T\in\N^{\ast}$, that has to be taken arbitrarily large in Corollary~\ref{cor:recurrence}.

Comparing with \eqref{e:nicer}, this means that we will need to deal with $\la \cS_{u^\gamma}^{{r-r'}}K^\gamma , K^\gamma\ra_{\mu^\gamma_k} $ where now $K^\gamma=B_{\gamma} K$, $K$ is $\gamma$-independent, and $B_\gamma:\mathscr{H}_m\to \mathscr{H}_k$ is defined by
\[
B_\gamma=m^\gamma Z_{\gamma}^{-1} \cF_\gamma \,.
\]
 
For simplicity, the calculations below are written for $k=1$. This suffices for our purposes, as we shall see in Section~\ref{s:step4}. Like in the statement of Theorem \ref{thm:1}, we will always assume that the $\gamma$-independent operator $K$ satisfies $\norm{K}_\infty:=\sup_{x,y\in V}|K(x,y)|\leq 1$.

The main results of this section are the two following propositions, that estimate the norm of the transfer operator $\cS_\gamma $ \eqref{e:sgamma} on proper subspaces. We call $F$ the space of functions $f$ on $B$ such that $f(e)$ ``depends only on the terminus'', that is, $f(e)= f(e')$ if $t_e=t_{e'}$. The first proposition estimates the norm of $\cS_\gamma $ on the orthogonal of $F$, and the second one estimates the norm of $\cS_{\gamma}^2 $ on the orthogonal of constant functions.

We denote by $\ell^2(B_1,U)$ the set $\ell^2(B_1)$ endowed with the scalar product $\langle f,g\rangle_U = \frac{1}{N}\sum_{e\in B_1} \overline{f(e)}g(e)$. Let $P_{F,U}$ be the orthogonal projector on $F$ in $\ell^2(B_1, U)$~:
\begin{equation}    \label{eq:pfu}
P_{F,U} K(e)= \frac{1}{d(t_e)} \sum_{e': \,t_{e'}=t_e} K(e') \, .
\end{equation}

We use as a ``reference operator'' the transfer operator $\mathcal{S}$ defined by
\begin{eqnarray*}
\cS: \ell^2(B, U) &\To&  \ell^2(B, U)  \\
\cS f(e)&=&\frac{1}{q(o_e)}\sum_{e' \rightsquigarrow e} f(e')
\end{eqnarray*}
where $q(x)=d(x)-1$. Both $\cS$ and $\cS^*$ are stochastic, if the adjoint of $\cS$ is taken in $\ell^2(B_1,U)$. The influence of the spectral gap assumption {\bf{(EXP)}} on the spectrum of $\cS$ is studied in \cite{A17} and we will use these results below.

We denote $\mathcal{Q}=\mathcal{S}^{\ast}\mathcal{S}$ and $\mathcal{Q}_2=\mathcal{S}^{2\,\ast}\mathcal{S}^2$. 
Note that $\mathcal{Q} (e,e')=0$ unless there exists $e''$ such that $e \rightsquigarrow e''$ and $e' \rightsquigarrow e''$. In this case, we say that $[e,e']$ is a \emph{pair}; $[e,e']$ form a pair iff they share the same terminus. The set of pairs is denoted by $P(B_1)$.

\begin{prp}          \label{lem:sp}
Let $B_{\gamma}K\in \mathscr{H}_1$. Let $w=P_{F^{\bot},\nu}B_{\gamma}K$ be the orthogonal projection of $B_\gamma K$ on $F^\perp$ in $\ell^2(\nu_1^{\gamma})$.
Then for any $M>0$ we have
\[
\|\mathcal{S}_{\gamma} w\|_{\nu_1^{\gamma}}^2 \le (1-\tfrac{3}{4}M^{-2})\cdot \|w\|_{\nu_1^{\gamma}}^2  + C_{N,M}(B_\gamma) \cdot \|K\|_{\infty}^2\, ,
\]
where
\begin{multline}     \label{eq:cnmb}
C_{N,M}(B_\gamma)  = \sup_{\|K\|_{\infty}=1}  \frac{\cs^{-1}}{2N} \sum_{[e,e'] \in Badp(M)} \mathcal{Q}(e,e') |B_\gamma K(e)- B_\gamma K(e')|^2
\\ + \cs^{-2} \sum_{e\,\in Bad(M)} \nu_1^{\gamma}(e)|B_\gamma K(e)-P_{F,U} B_\gamma K(e)|^2.
\end{multline}

\end{prp}
The sets $Bad(M)$ of bad edges and $Badp(M)$ of bad pairs of edges will be defined in the course of the proof. They correspond to the aforementioned peaks of $\mu_1^{\gamma}$ and problems of small transition probabilities. If there were no bad edges and bad pairs, Proposition \ref{lem:sp} would be a genuine spectral gap estimate.

\begin{prp}         \label{lem:7.5b}
Let $B_{\gamma}K\in \mathscr{H}_1$. Let $f=P_{\mathbf{1}^{\bot},\nu}B_{\gamma}K$ be the orthogonal projection of $B_\gamma K$ on $\mathbf{1}^{\bot}$ in $\ell^2(\nu_1^{\gamma})$.
Then for any $M>0$ we have
\[
\|\cS_{\gamma}^2  f\|_{\nu_1^{\gamma}}^2 \le (1-M^{-2}c(D, \beta)) \cdot \| f\|_{\nu_1^{\gamma}}^2 + C_{N,M,2}(B_\gamma) \cdot \|K\|_{\infty}^2 \, ,
\]
where $c(D, \beta)>0$ is explicit and depends only on $D$ (upper bound on the degree) and the spectral gap $\beta$ of \emph{\textbf{(EXP)}}, and
\begin{multline*}
C_{N,M,2}(B_\gamma)  =  \sup_{ \|K\|_{\infty}=1} \frac{\cs^{-1}}{2N} \sum_{[e,e'] \in Badp(2, M)} \mathcal{Q}_2(e,e') |B_\gamma K(e)- B_\gamma K(e')|^2
\\ + \cs^{-2} \sum_{e\,\in Bad(M)} \nu_1^{\gamma}(e) |B_\gamma K(e)-P_{\mathbf{1},U} B_\gamma K(e)|^2 \,,
\end{multline*}
where $P_{\mathbf{1},U}$ is the orthogonal projector on $\mathbf{1}$ in $\ell^2(B_1,U)$.
\end{prp}

Here, $Badp(2,M)$ is another set of bad edge-couples defined in the proof.

The quantities $C_{N,M}(B_\gamma), C_{N,M,2}(B_\gamma)$ are estimated in Proposition~\ref{p:lesscrude}.

\begin{proof}[Proof of Proposition \ref{lem:sp}]
Let $\mathcal{Q}^\gamma =\mathcal{S}_{\gamma}^{\ast}\cS_\gamma $ (where now the adjoint is considered in $\ell^2(\nu_1^\gamma)$). 
The operator $\mathcal{Q}^\gamma $ being self-adjoint on $\ell^2(\nu_1^{\gamma})$ is equivalent to the relation 
\begin{equation}            \label{eq:qself}
\nu_1^{\gamma}(e)\mathcal{Q}^\gamma (e,e')=\nu_1^{\gamma}(e')\mathcal{Q}^\gamma (e',e)
\end{equation}
for all $e,e'\in B_1$. Note that
 $\mathcal{Q}^\gamma (e,e')=0$ unless $[e,e']$ is a pair.

Define $D^\gamma(e) = \sum_{e'} \mathcal{Q}^\gamma (e,e')\le 1$ and $\mathcal{M}^\gamma (e,e') = D^\gamma(e)\delta_{e=e'}-\mathcal{Q}^\gamma (e,e')$.

Then using \eqref{eq:qself}, we have the \emph{Dirichlet identity}
\begin{equation}          \label{eq:Diri}
\frac{1}{2}\sum_{e,e'}\nu_1^{\gamma}(e)\mathcal{Q}^\gamma (e,e')|K(e)-K(e')|^2 =  \langle K,\mathcal{M}^\gamma K\rangle_{\nu_1^{\gamma}} \, .
\end{equation}

We observe that for any $K\in \ell^2(\nu_1^{\gamma})$,
\begin{equation}       \label{eq:SPnorm}
\|\cS_\gamma  K \|_{\nu_1^{\gamma}} \le \|K\|_{\nu_1^{\gamma}} \, .
\end{equation}
Indeed, denoting $\la\cdot, \cdot\ra_{\nu} :=\la\cdot, \cdot\ra_{\nu_1^{\gamma}}$, we have $\|\cS_\gamma  K \|_{\nu}^2 = \langle K, \mathcal{Q}^\gamma K\rangle_{\nu}$ and $\langle K,\mathcal{M}^\gamma  K \rangle_{\nu} \ge 0$ by Dirichlet, so $\|K\|_{\nu}^2 \ge \langle K,D^\gamma K\rangle_{\nu} \ge \langle K,\mathcal{Q}^\gamma  K \rangle_{\nu}$ as claimed. 

\begin{rem}            \label{rem:F}
The Dirichlet identity shows that
\[
F = \{K \in \IC^B : \mathcal{M}^\gamma  K =0\}= \{K \in \IC^B : (I-\cQ)K=0\} \, .
\]
\end{rem}

\begin{rem}     \label{rem:F2}
If $J\perp F$ in $\ell^2(B_1,U)$, then $\langle J,(I-\cQ)J\rangle_U \ge \frac34\,\|J\|_U^2$.

Indeed, $\langle \tau_+\delta_y, J\rangle_U=0$ for all $y\in V$, so $\sum_{x\sim y} J(x,y)=0$ for all $y\in V$ and thus $(\cQ J)(x_0,x_1)=(\cS^{\ast}\cS J)(x_0,x_1) = \frac{J(x_0,x_1)}{q(x_1)^2}$ (recall that $q(x)=d(x)-1$ where $d(x)$ is the degree of $x$). As $\min q(x)\ge 2$, we get $\|\cQ J\|_U \le \frac{1}4\,\|J\|_U$ and the claim follows.
\end{rem}

Fix a large $\cs>0$. We call $e\in B_1$ \emph{bad} if $\nu_1^{\gamma}(e) > \frac{\cs}{N}$. We call a pair $[e,e']\in P(B_1)$ \emph{bad} if $\nu_1^{\gamma}(e)\mathcal{Q}^\gamma (e,e') <\frac{\cs^{-1}}{N}$. We call $Bad(M)$ and $Badp(M)$ the sets of bad $e$ and $[e,e']$, respectively.

To prove Proposition~\ref{lem:sp}, we first note that by (\ref{eq:Diri}), and letting $K_\gamma=B_\gamma K$,
\begin{multline}\label{e:badpair}
\|w\|_{\nu}^2 - \|\cS_\gamma w\|_{\nu}^2  \ge \langle w, \mathcal{M}^\gamma  w\rangle_{\nu} = \langle K_\gamma, \mathcal{M}^\gamma  K_\gamma\rangle_{\nu} \\
=\frac{1}{2}\sum_{[e,e']\in P(B_1)} \nu_1^{\gamma}(e)\mathcal{Q}^\gamma (e,e')|K_{\gamma}(e)-K_{\gamma}(e')|^2 \\
 \ge \frac{\cs^{-1}}{2N} \sum_{[e,e'] \not\in Badp(M)} \mathcal{Q}(e,e') |K_{\gamma}(e)- K_{\gamma}(e')|^2\\
= \cs^{-1} \langle  K_{\gamma},(I-\mathcal{Q}) K_{\gamma}\rangle_U -  \frac{\cs^{-1}}{2N} \sum_{[e,e'] \in Badp(M)} \mathcal{Q}(e,e') |K_{\gamma}(e)- K_{\gamma}(e')|^2 \, ,
\end{multline}
where we used $\cQ(e,e')\le 1$. By Remark~\ref{rem:F2},
\begin{align*}
\left\langle K_{\gamma},(I-\mathcal{Q}) K_{\gamma}\right\rangle_U & =\left\langle  K_{\gamma}-P_{F,U} K_{\gamma}, (I-\mathcal{Q})( K_{\gamma}-P_{F,U} K_{\gamma})\right\rangle_U \\
& \ge \frac34 \cdot \| K_{\gamma}-P_{F,U} K_{\gamma}\|_U^2 \, .
\end{align*}
Now
\begin{multline}\label{e:badedge}
\| K_{\gamma}-P_{F,U} K_{\gamma}\|_U^2  \ge \cs^{-1} \sum_{e\, \not\in Bad(M)} \nu_1^{\gamma}(e)| K_{\gamma}(e)-P_{F,U} K_{\gamma}(e)|^2 \\
 = \cs^{-1}\|K_{\gamma}-P_{F,U}K_{\gamma}\|_{\nu}^2 - \cs^{-1} \sum_{e\,\in Bad(M)} \nu_1^{\gamma}(e)| K_{\gamma}(e)-P_{F,U} K_{\gamma}(e)|^2 \,
 \\ \geq \cs^{-1} \|w\|_{\nu}^2 - \cs^{-1} \sum_{e\,\in Bad(M)} \nu_1^{\gamma}(e)| K_{\gamma}(e)-P_{F,U} K_{\gamma}(e)|^2 .
\end{multline}
We used that $\|K_{\gamma}-P_{F,U}K_{\gamma}\|_{\nu}^2 \geq \|w\|_{\nu}^2$ since $w = P_{F^{\bot},\nu}(K_{\gamma}-P_{F,U}K_{\gamma})$.
The result is obtained by putting together \eqref{e:badpair} and \eqref{e:badedge}.
\end{proof}

\begin{proof}[Proof of Proposition \ref{lem:7.5b}]
We now let $\mathcal{Q}^\gamma_2 = \mathcal{S}_{\gamma}^{2\,\ast}\mathcal{S}_{\gamma}^2$ (where the adjoint is taken in $\ell^2(\nu_1^\gamma)$). Then
 $\mathcal{Q}^\gamma_2(e,e') \neq 0$ iff there exists $e'', e_1, e'_1$ such that $e\rightsquigarrow e_1 \rightsquigarrow e''$ and $e'\rightsquigarrow e_1'\rightsquigarrow e''$. We denote the set of such couples $[e,e']$ by $P_2(B_1)$ and let $\mathcal{M}_2^{\gamma}(e,e')=D_2 \delta_{e=e'}-\mathcal{Q}_2(e,e')$, where $D_2(e)=\sum_{e'}\mathcal{Q}_2^{\gamma}(e,e')\le 1$.
 
 Fix $M>0$. We say that $[e,e']\in P_2(B_1)$ is \emph{bad} if $\nu_1^{\gamma}(e)\mathcal{Q}_2(e,e') < \frac{M^{-1}}{N}$. We call $Badp(2, M)$ the set of bad couples in $P_2(B_1)$.
 
The proof is then similar to Proposition~\ref{lem:sp}, replacing the space $F$ by the space of constant functions and using \cite[Theorem 1.1]{A17} instead of Remark~\ref{rem:F2}. In particular, the quantity $c(\beta, D)$ is the one appearing in \cite[Theorem 1.1]{A17}.
\end{proof}

Later on, we will need to iterate the result of Proposition \ref{lem:7.5b}, considering $\mathcal{S}_{\gamma}^{2r}$ instead of $\mathcal{S}_{\gamma}^{2}$. Since $\mathcal{S}_{\gamma}^*$ is not exactly stochastic, $\mathcal{S}_{\gamma}$ does not preserve the orthogonal of constants. Still, we can iterate \eqref{e:sumzeta3} to get $\mathcal{S}_{\gamma}^{\ast\,l} \mathbf{1} = 1 - \eta_1 \sum_{s=0}^{l-1} \mathcal{S}_{\gamma}^{\ast\,s}\xi^{\gamma}$, where $\xi^{\gamma}(x_0,x_1) = \frac{|\zeta_{x_0}^{\gamma}(x_1)|^2}{|\Im \zeta_{x_0}^{\gamma}(x_1)|}$. Hence, for any $K$ we have $\langle \mathbf{1},\mathcal{S}_{\gamma}^lK\rangle_{\nu} = \langle \mathbf{1},K\rangle_{\nu} - \eta_1\langle \sum_{s=0}^{l-1}\mathcal{S}_{\gamma}^{\ast\,s}\xi^{\gamma},K\rangle_{\nu}$. Denoting
\[
\mathcal{Z}_l K :=  \xi^{\gamma} \sum_{s=0}^{2l-1} \mathcal{S}_{\gamma}^sK \, , \qquad \mathcal{Z}_0 K := 0 \, ,
\]
we see that if $K\perp \mathbf{1}$, then $(\mathcal{S}_{\gamma}^{2l}K + \eta_1 \mathcal{Z}_l K) \perp \mathbf{1}$.

\begin{prp}             \label{p:iter}
Let $K\in \mathscr{H}_{m}$. Let $ f=P_{\mathbf{1}^\perp, \nu}B_\gamma K$ be the orthogonal projection of $B_\gamma K$ on $\mathbf{1}^{\bot}$ in $\ell^2(\nu_1^{\gamma})$.
Then for any $M>0$ we have

\begin{align*}
\| \mathcal{S}_{\gamma}^{2r }  f \|_{\nu} & \le \left(1-M^{-2}c(D,\beta)\right)^{r/2}  \| f\|_{\nu} +   \sum_{l=0}^{r-1}C_{N,M,l,2}(B_\gamma)^{1/2} \|K\|_{\infty}  + 2\eta_1  \sum_{l=1}^{r-1} \| \mathcal{Z}_{l}  f\|_{\nu}\, .
\end{align*}
where $C_{N,M,l,2}(B_\gamma)= C_{N,M,2}((\mathcal{S}_{\gamma}^{2l} + \eta_1 \mathcal{Z}_l )P_{\mathbf{1}^\perp, \nu}B_\gamma)$.
\end{prp}
\begin{proof}
The proof is by induction on $r$. This holds for $r=1$ by Proposition~\ref{lem:7.5b}. Assume the result holds for $r$. If $f\perp \mathbf{1}$, we have just seen that $(\mathcal{S}_{\gamma}^{2r} + \eta_1 \mathcal{Z}_r)f \perp \mathbf{1}$ in $\ell^2(\nu_1^{\gamma})$. So using Proposition~\ref{lem:7.5b} and \eqref{eq:SPnorm},
\begin{align*}
& \|\mathcal{S}_{\gamma}^{2(r+1)}f\|_{\nu} \le \|\mathcal{S}_{\gamma}^{2}(\mathcal{S}_{\gamma}^{2r } + \eta_1 \mathcal{Z}_{r})f\|_{\nu} + \eta_1 \|\mathcal{Z}_{r} w\|_{\nu} \\
& \quad \le \left(1-M^{-2}c(D,\beta)\right)^{1/2} \|(\mathcal{S}_{\gamma}^{2r}  + \eta_1 \mathcal{Z}_{r})f\|_{\nu} + \,C_{N,M,r,2}(B_\gamma)^{1/2} \, \|K\|_{\infty} + \eta_1 \|\mathcal{Z}_{r} f\|_{\nu}\, .
\end{align*}
Since $\|(\mathcal{S}_{\gamma}^{2r} + \eta_1 \mathcal{Z}_{r})f\| \le \|\mathcal{S}_{\gamma}^{2r} f\| + \eta_1\| \mathcal{Z}_{r}f\|$, the claim follows.
\end{proof}

The rest of this section is devoted to estimating the ``bad'' quantities.

\begin{prp}          \label{lem:bad1}
Under assumptions \emph{\textbf{(BSCT)}} and \emph{\textbf{(Green)}}, for any $s \ge 1$, there exists $C_s$ such that for all $M>1$ we have
\[
\sup_{ \eta_1\in(0,1)} \limsup_{N\to \infty} \sup_{\Re \gamma \in I_1, \Im\gamma=\eta_1} \nu_1^{\gamma}(Bad(M)) \le C_{s}\cs^{-s} \quad \text{and} \ \limsup_{N\to \infty} \frac{\# Badp(M)}{N} \le C_{s}\cs^{-s} \, .
\]
\end{prp}
\begin{proof}
We have $\nu_1^{\gamma}(Bad) = \nu_1^{\gamma}\{e : \nu_1^{\gamma}(e)>\frac{\cs}{N}\} $, so
\[
\nu_1^{\gamma}(Bad)  \le \cs^{-s}N^s \sum_{e\in B_1} \nu_1^{\gamma}(e) \nu_1^{\gamma}(e)^s= \cs^{-s} \Big(\frac{N}{\mu_1^{\gamma}(B_1)}\Big)^{s+1}\frac{1}{N}\sum_{e\in B_1} \mu_1^{\gamma}(e)^{s+1} \, .
\]
Recalling the definition of $\mu_1^{\gamma}$ \eqref{e:muk}, and using Remark~\ref{rem:IDS1}, we get
\[
\Big(\frac{N}{\mu_1^{\gamma}(B_k)}\Big)^{s+1}\frac{1}{N}\sum_{e\in B_1} \mu_1^{\gamma}(e)^{s+1} \Lim_{N\To +\infty } \frac{1}{ \expect[\sum_{o'\sim o}\hat{\mu}_1^{\gamma}(o,o')]^{s+1} }\expect\bigg[\sum_{o'\sim o}\hat{\mu}_1^{\gamma}(o,o')^{s+1}\bigg]
\]
uniformly in $\Re \gamma\in I_1$, for any fixed $\Im \gamma=\eta_1$. By Remark~\ref{rem:IDS2}, this is bounded by some constant $C_s$. The second assertion is proved similarly.
 \end{proof}

\begin{prp}\label{p:lesscrude}
For all $t\in\N$,
\begin{multline*}
C_{N,M}(\cS_{u^\gamma}^t B_\gamma)\leq  \frac{2\cs^{-1}}{N}   \#  Badp(M)^{1/3}  \left(\sum_e \frac{1}{\nu_1^{\gamma} (e)}\right)^{1/3}   \left( \sum_e \nu_1^{\gamma} (e)  \Big(\sum_w |B_\gamma (e, w)|\Big)^6\right)^{1/3}\\
+2 \cs^{-2} \nu_1^{\gamma}(Bad(M))^{1/2}\left(\sum_{e } \nu_1^{\gamma}(e) \Big( \sum_w |B_\gamma (e, w)|\Big)^4\right)^{1/2} \\
+2  \cs^{-2}  \nu_1^{\gamma}(Bad(M))^{1/2}\left(\sum_{e } \frac{ [(P_{F,U} \nu_1^{\gamma})(e)]^2}{\nu_1^{\gamma}(e)} \right)^{1/4}\left(\sum_{e } \nu_1^{\gamma}(e)  \Big(\sum_w |B_\gamma (e, w)|\Big)^8\right)^{1/4} \, ,
\end{multline*}
where $(P_{F,U}\nu_1^{\gamma})(e) = \frac{1}{d(t_e)} \sum_{t_{e'}=t_e} \nu_1^{\gamma}(e')$, and
\begin{multline}
C_{N,M, 2}(\cS_{u^\gamma}^t B_\gamma) \\
\leq  \frac{2\cs^{-1}}{N}   \#  Badp(2, M)^{1/3}  \left(\sum_e \frac{1}{\nu_1^{\gamma} (e)}\right)^{1/3}   \left( \sum_e \nu_1^{\gamma} (e)  \Big(\sum_w |B_\gamma (e, w)|\Big)^6\right)^{1/3}\\
+2 \cs^{-2} \nu_1^{\gamma}(Bad(M))^{1/2}\left(\sum_{e } \nu_1^{\gamma}(e) \Big( \sum_w |B_\gamma (e, w)|\Big)^4\right)^{1/2} \\
+2  \cs^{-2}  \nu_1^{\gamma}(Bad(M))^{1/2}\left(\frac{1}{N^2}\sum_{e } \frac{ 1}{\nu_1^{\gamma}(e)} \right)^{1/4}\left(\sum_{e } \nu_1^{\gamma}(e)  \Big( \sum_w |B_\gamma (e, w)|\Big)^8\right)^{1/4}\,.\label{e:2}
\end{multline}
Similar estimates hold if $B_\gamma$ is replaced by $P_{{\mathbf{1}}^\perp, \nu}B_\gamma$, where $P_{{\mathbf{1}}^\perp, \nu}$ is the projection on the orthogonal of constants in $\ell^2(\nu_1^\gamma)$.
\end{prp}

We first deduce the following corollary. Recall that the operators $\cF_{\gamma}$ from Corollary~\ref{cor:recurrence} depend on a parameter $T\in\N^{\ast}$, and $B_{\gamma} = m^{\gamma}Z_{\gamma}^{-1}\cF_{\gamma}$. In this section, $T$ is fixed, but will be taken to $+\infty$ in Section \ref{sec:retour}.

\begin{cor} \label{c:Minfty}
For any $s>0$, there exists $C_{s,T}$ such that, for all $M$, 
\[
\sup_{ \eta_1\in(0,1)} \limsup_{N\to \infty} \sup_{\Re \gamma \in I_1, \Im\gamma=\eta_1}\sup_{t\in\N} C_{N,M}(\cS_{u^\gamma}^t B_\gamma) \leq C_{s,T} M^{-s}
\]
and
\[
\sup_{ \eta_1\in(0,1)} \limsup_{N\to \infty} \sup_{\Re \gamma \in I_1, \Im\gamma=\eta_1} \sup_{t\in\N} C_{N,M, 2}(\cS_{u^\gamma}^t B_\gamma) \leq C_{s,T} M^{-s} \, .
\]
Similar estimates hold if $B_\gamma$ is replaced by $P_{{\mathbf{1}}^\perp, \nu}B_\gamma$.
\end{cor}
\begin{proof}[Proof of Corollary~\ref{c:Minfty}]
This will follow from Propositions \ref{lem:bad1} and \ref{p:lesscrude} if we show that
\begin{equation}\label{e:nu-1}
\sup_{ \eta_1\in(0,1)} \limsup_{N\to \infty} \sup_{\Re \gamma \in I_1, \Im\gamma=\eta_1} N^{-2}\sum_e \frac{1}{\nu_1^{\gamma} (e)} <+\infty
\end{equation}
\begin{equation}\label{e:nu-2}
\sup_{ \eta_1\in(0,1)} \limsup_{N\to \infty} \sup_{\Re \gamma \in I_1, \Im\gamma=\eta_1} \sum_e \nu_1^{\gamma} (e)  \Big(\sum_w |B_\gamma (e, w)|\Big)^\alpha <+\infty
\end{equation}
($\alpha=4, 6, 8$) and
\begin{equation}\label{e:nu-3}
\sup_{ \eta_1\in(0,1)} \limsup_{N\to \infty} \sup_{\Re \gamma \in I_1, \Im\gamma=\eta_1} \sum_e  \frac{1}{\nu_1^{\gamma}(e)} \frac{1}{d(t_e)^2} \Big(\sum_{t_{e'}=t_e} \nu_1^{\gamma}(e')\Big)^2 <+\infty\,.
\end{equation}
For \eqref{e:nu-1}, we have by Remark~\ref{rem:IDS1} that
\[
N^{-2}\sum_e\frac{1}{\nu_1^{\gamma}(e)} = \frac{\sum_e \mu_1^{\gamma}(e)}{N} \cdot \frac{1}{N}\sum_e \frac{1}{\mu_1^{\gamma}(e)} \Lim_{N\to \infty } \IE\left(\sum_{o'\sim o}\hat{\mu}_1^{\gamma} (o, o')\right) \cdot \IE\left(\sum_{o'\sim o}\frac{1}{\hat{\mu}_1^{\gamma} (o, o')}\right)
\]
uniformly in $\Re\gamma \in I_1$, for any fixed $\Im \gamma=\eta_1$. So the claim follows Remark~\ref{rem:IDS2}.

For \eqref{e:nu-3}, we have $\frac{1}{d(t_e)^2}(\sum_{t_{e'}=t_e}\nu_1^{\gamma}(e'))^2 \le \sum_{t_{e'}=t_e} \nu_1^{\gamma}(e')^2$, so we deduce the upper bound $D (\sum_e \frac{1}{\nu_1^{\gamma}(e)^2})^{1/2}(\sum_e\nu_1^{\gamma}(e)^4)^{1/2}$, which is uniformly bounded by
\[
\frac{D}{\expect\left(\sum_{o'\sim o} \hat{\mu}_1^{\gamma}(o,o')\right)} \expect\left(\sum_{o'\sim o} \frac{1}{\hat{\mu}_1^{\gamma}(o,o')^2}\right)^{1/2}\expect\left(\sum_{o'\sim o}\hat{\mu}_1^{\gamma}(o,o')^4\right)^{1/2}.
\]

Finally, for \eqref{e:nu-2}, we write
\begin{multline*}
\sum_e \nu_1^{\gamma}(e)\Big(\sum_w |B_{\gamma}(e,w)|\Big)^{\alpha} \\
 \le \frac{N}{\sum_e\mu_1^{\gamma}(e)} \left(\frac{1}{N}\sum_e\frac{\mu_1^{\gamma}(e)^2(m_{o_e}^{\gamma})^{2\alpha}}{\zeta^{\gamma}_{o_e}(t_e)^{2\alpha}}\right)^{1/2}\left(\frac{1}{N}\sum_e\Big(\sum_w\cF_{\gamma}(e,w)|\Big)^{2\alpha}\right)^{1/2} .
\end{multline*}
The first two terms are bounded by $\frac{1}{\expect(\sum_{o'\sim o}\hat{\mu}_1^{\gamma}(o,o'))}(\expect\sum_{o'\sim o}\frac{\hat{\mu}_1^{\gamma}(o,o')^2(\hat{m}_o^{\gamma})^{2\alpha}}{\hat{\zeta}^{\gamma}_o(o')^{2\alpha}})^{1/2}$ and the last term is shown to be uniformly bounded in Remark~\ref{rem:holverif}. This completes the proof.
\end{proof}

\begin{proof}[Proof of Proposition \ref{p:lesscrude}]
An important point here is to obtain a bound that does not depend on $t$. Recalling \eqref{eq:cnmb}, we first estimate
\begin{multline}    \label{eq:cnmbpreuve}
 \sum_{[e,e'] \in Badp(M)} \mathcal{Q}(e,e') |\cS_{u^\gamma}^t B_\gamma K(e)- \cS_{u^\gamma}^t B_\gamma K(e')|^2 \\ 
 \leq  4\sum_{[e,e'] \in Badp(M)} \mathcal{Q}(e,e') |\cS_{u^\gamma}^t B_\gamma K(e)|^2  = 4 \sum_e n(e) |\cS_{u^\gamma}^t B_\gamma K(e)|^2 \,,
\end{multline}
 where $n(e)=\sum_{e' : [e,e'] \in Badp(M)} \mathcal{Q}(e,e')$. Using H\"older, this is less than
\[
 4 \left(\sum_e n^3(e)\right)^{1/3}  \left(\sum_e \frac{1}{\nu_1^{\gamma} (e)}\right)^{1/3}   \left( \sum_e \nu_1^{\gamma} (e)|\cS_{u^\gamma}^t B_\gamma K(e)|^6\right)^{1/3} \, .
\]
But again by H\"older and the fact that $\mathcal{Q}$ is stochastic, we have
\[
\sum_e n^3(e) \le \sum_e \Big(\sum_{e'} \bbbone_{[e,e']\in Badp(M)}\Big)\Big(\sum_{e'}\mathcal{Q}(e,e')^{3/2}\Big)^2 \le \#  Badp(M)  \, .
\]
Next, recalling \eqref{e:Su}, \eqref{e:Sgamma}, we have $|\mathcal{S}_{u^{\gamma}}^tB_{\gamma}K(e)| \le (\mathcal{S}_{\gamma}^t|B_{\gamma}K|)(e)$. As $\cS_{\gamma}^t$ and $\cS_{\gamma }^{\ast\,t}$ are substochastic, and $\nu_1^{\gamma}(e)\mathcal{S}_{\gamma}^t(e,e')=\nu_1^{\gamma}(e')\mathcal{S}_{\gamma}^{\ast\,t}(e',e)$, we have
\begin{multline*}
\sum_e \nu_1^{\gamma}(e) [\mathcal{S}_{\gamma}^t|B_{\gamma}K|(e)]^6  \le \sum_e \nu_1^{\gamma}(e)\Big(\sum_{e'}\mathcal{S}^t_{\gamma}(e,e')\Big)^5\Big(\sum_{e'}\mathcal{S}^t_{\gamma}(e,e')[|B_{\gamma}K|(e')]^6\Big) \\
\le \sum_{e,e'}\nu_1^{\gamma}(e')\mathcal{S}_{\gamma}^{\ast\,t}(e,e')[|B_{\gamma}K|(e')]^6 \le \sum_{e'}\nu_1^{\gamma}(e') [|B_{\gamma}K|(e')]^6 \, .
\end{multline*}
Collecting the estimates, we showed that \eqref{eq:cnmbpreuve} is bounded by
\[
4  \left(\#  Badp(M) \right)^{1/3}  \left(\sum_e \frac{1}{\nu_1^{\gamma} (e)}\right)^{1/3}   \left( \sum_e \nu_1^{\gamma} (e)  \left[|B_\gamma K|(e)\right]^6\right)^{1/3}.
\]
For the second term in \eqref{eq:cnmb}, we have
\begin{multline}
 \sum_{e\,\in Bad(M)} \nu_1^{\gamma}(e)|\cS_{u^\gamma}^t B_\gamma K(e)- P_{F,U} \cS_{u^\gamma}^t B_\gamma K(e)|^2
 \\ \leq  2\sum_{e\,\in Bad(M)} \nu_1^{\gamma}(e)\left(\left[\cS_{\gamma}^t |B_\gamma K|(e)\right]^2+\left[P_{F,U}\cS_{\gamma}^t | B_\gamma K|(e)\right]^2\right)
\end{multline}
and again, as $\mathcal{S}_{\gamma}^t$ and $\mathcal{S}_{\gamma}^{\ast\,t}$ are substochastic,
\[
\sum_{e\,\in Bad(M)} \nu_1^{\gamma}(e)\left[\cS_{\gamma}^t \,|B_\gamma K|(e)\right]^2 \le \nu_1^{\gamma}(Bad(M))^{1/2}\left(\sum_{e }  \nu_1^{\gamma}(e) \left[|B_\gamma K|(e)\right]^4\right)^{1/2} \, .
\]
Also,
\begin{multline*}
\sum_{e\,\in Bad(M)} \nu_1^{\gamma}(e) \left[P_{F,U} \cS_{\gamma}^t\, | B_\gamma K|(e)\right]^2 \\
\leq  \nu_1^{\gamma}(Bad(M))^{1/2}\left(\sum_{e }  \nu_1^{\gamma}(e)\left[P_{F,U} \cS_{\gamma}^t \,| B_\gamma K|(e)\right]^4\right)^{1/2} \, .
\end{multline*}
Using that $P_{F,U}$ is stochastic and $\mathcal{S}_{\gamma}^t$ is substochastic, we have
\begin{multline*}
\sum_{e }  \nu_1^{\gamma}(e)\left[P_{F,U} \cS_{\gamma}^t | B_\gamma K|(e)\right]^4 \le \sum_{e,e'} \nu_1^{\gamma}(e)P_{F,U}(e,e') \left[\mathcal{S}_{\gamma}^t\,|B_{\gamma}K|(e')\right]^4 \\ 
\le \left(\sum_{e'} \frac{[(P_{F,U}\nu_1^{\gamma})(e')]^2}{\nu_1^{\gamma}(e')}\right)^{1/2}\left(\sum_{e'}\nu_1^{\gamma}(e')\left[\mathcal{S}_{\gamma}^t\,|B_{\gamma}K|(e')\right]^8\right)^{1/2} \\
  \leq  \left(\sum_{e } \frac{ [(P_{F,U} \nu_1^{\gamma})(e)]^2}{\nu_1^{\gamma}(e)} \right)^{1/2}\left(\sum_{e } \nu_1^{\gamma}(e)  \left[| B_\gamma K|(e)\right]^8\right)^{1/2} .
 \end{multline*}
 This yields the first inequality. The second one is proven similarly.
\end{proof}

\begin{rem}\label{r:bigonotation}
Note that if $\|K\|_{\infty}\le 1$, then
\begin{equation}\label{eq:bgammak}
\|B_{\gamma}K\|_{\nu_1^{\gamma}}^2 = \sum_{e\in B}\nu_1^{\gamma}(e)|B_{\gamma}K(e)|^2 \le \sum_e\nu_1^{\gamma}(e)\Big(\sum_w |B_{\gamma}(e,w)|\Big)^2 \, ,
\end{equation}
so $\sup_{ \eta_1 >0} \limsup_{N\to \infty} \sup_{\Re \gamma \in I_1, \Im\gamma=\eta_1}\|B_{\gamma}K\|_{\nu_1^{\gamma}}^2  \le C_T$ by the proof in Corollary~\ref{c:Minfty}, see also Remark~\ref{rem:holverif}.

For a quantity $A(N, \gamma, \kappa)$ depending on $N, \gamma$ (and possibly on an additional parameter $\kappa$), we will write $A(N, \gamma, \kappa)=O_\kappa(1)_{N\To +\infty, \gamma}$
to mean that, for any given $\kappa$,
\[
\sup_{ \eta_1\in (0,1)} \limsup_{N\to \infty} \sup_{\Re \gamma \in I_1, \Im\gamma=\eta_1}A(N, \gamma, \kappa) < +\infty \, .
\]
For instance, if $\|K\|_{\infty}\le 1$, then $\|B_{\gamma}K\|_{\nu_1^{\gamma}}^2= O_T(1)_{N\To +\infty, \gamma}$. This is true more generally for $\|B_{\gamma}K\|_{\nu_k^{\gamma}}^2$, with $B_{\gamma}=\frac{m^{\gamma}}{Z_{\gamma}}\cF_{\gamma}:\mathscr{H}_m\to\mathscr{H}_k$, and $\cF_{\gamma}$ as in Corollary~\ref{cor:recurrence}.

Similarly, for the operator $\mathcal{Z}_l$ appearing in Proposition~\ref{p:iter}, the arguments in Corollary~\ref{c:Minfty} and Remark~\ref{rem:holverif} show that $\|\mathcal{Z}_l f\|_{\nu_1^{\gamma}} = O_{l,T}(1)_{N\To +\infty, \gamma}$.

Finally, by Corollary \ref{c:Minfty}, $\sup_t C_{N,M,2}(\mathcal{S}_{u^{\gamma}}^tB_\gamma)$ is uniformly bounded by $C_{s,T} M^{-s}$ for any $M$ and $s$, as $N\to +\infty$. We use the notation $O_T(M^{-\infty})_{N\To +\infty, \gamma}$ to express this.
\end{rem}

\section{Transition matrices with phases}        \label{sec:su}
We now consider the operator $\mathcal{S}_{u^{\gamma}}$ given in \eqref{e:Su}. If we denote by $M_{u^{\gamma}}$ the multiplication operator $(M_{u^{\gamma}}K)(x_0;x_k) = \overline{u_{x_1}^{\gamma}(x_0)}K(x_0;x_k)$, where ${u_{x_1}^{\gamma}(x_0)}$ is the function of modulus $1$ defined in \eqref{e:ugamma}, then $\mathcal{S}_{u^{\gamma}} =\cS_\gamma  M_{u^{\gamma}}$.

It is well known that putting phases into a matrix with positive entries will strictly diminish its spectral radius, unless the phases satisfy very special relations~: this is the contents of Wielandt's theorem \cite[Chapter 8]{Mey01}. This is reflected in Proposition~\ref{p:Sualtmod} below. Without the error term, part (i) says that the norm of $\cS_{u^{\gamma}}^4$ is strictly smaller than one, in contrast to $\cS_{\gamma}^4$ (the latter only contracts the norm on proper subspaces, see Section~\ref{s:mixing}). The contraction property of $\cS_{u^{\gamma}}^4$ holds true except in special cases, described in part (ii) of Proposition~\ref{p:Sualtmod}.

Note that we are not using Wielandt's theorem directly, as we want some information on the norm of the operator $\cS_{u^{\gamma}}^4$ instead of its spectral radius.
In addition, as in Section~\ref{s:mixing}, we need estimates that are uniform both as $N\to \infty$ and as $\gamma$ approaches the real axis.

Recall from Section \ref{s:mixing} that $B_\gamma$ is an operator $\mathscr{H}_m\to \mathscr{H}_k$ with $1\le k\le m$. As in Section \ref{s:mixing}, the case $k=1$ suffices for our purposes, but we need more general operators $A_{\gamma}:\mathscr{H}_m\to\mathscr{H}_1$ defined in terms of $B_{\gamma}$. The quantities $C_{N,M}(A_\gamma), C_{N,M,2}(A_\gamma)$ were introduced in Propositions \ref{lem:sp} and \ref{lem:7.5b}. In particular, $C_{N,M,2}(I)$ corresponds to the case where $A_\gamma$ is the identity operator. The measure $\nu_1^{\gamma}$ is defined in \eqref{e:muk} and \eqref{e:nuk}.

\begin{prp}\label{p:Sualtmod}
Fix $\gamma\in\C^+$, $A_{\gamma}K\in \mathscr{H}_1$, $\varepsilon \in (0,1)$, $\cs>0$ and a graph $G=G_N$. Denote $\eta_1=\Im \gamma$. Then
 \begin{enumerate}[\rm (i)]
\item Either we have
\begin{equation}\label{e:decay}
\| \mathcal{S}_{u^{\gamma}}^{4}A_\gamma K\|_{\nu_1^{\gamma}}^2 \le (1-\varepsilon)^2\|A_\gamma K\|_{\nu_1^{\gamma}}^2 + \tilde{C}_{N,\cs, 2}(A_\gamma)\cdot \|K\|_{\infty}^2
\end{equation}
with
\[
\tilde{C}_{N,\cs,2}(A_\gamma) = \max \{C_{N,\cs}(A_\gamma),C_{N,\cs,2}(A_\gamma), C_{N,\cs}(\mathcal{S}_{u^{\gamma}}A_\gamma),C_{N,\cs,2}(\mathcal{S}^2_{u^{\gamma}}A_\gamma) \} \, ,
\]
\item or there exist $\theta : V \to \R$ and constants $s_j$ with $|s_j| \le 1$, $j=1,2$, such that
\[
\Big\|u^{\gamma}_{x_1}(x_0) - s_2 \frac{e^{-i[\theta(x_0) + \theta(x_1)]}}{n_{x_0}^{\gamma}}\Big\|^2_{\nu_1^{\gamma}} \le c_{\cs,\beta} \left[\varepsilon^{1/2} + \eta_1 \, \|\xi^{\gamma}\|_{\nu_1^{\gamma}}+ \eta^2_1\, \|\xi^{\gamma}\|^2_{\nu_1^{\gamma}}\right] +C'_{N,\cs }\, ,
\]
and
\[
\|u^{\gamma}_{x_1}(x_0) - s_1 n_{x_1}^{\gamma}e^{i[\theta(x_0) + \theta(x_1)]}\|^2_{\nu_1^{\gamma}} \le  c_{\cs,\beta} \left[\varepsilon^{1/2} + \eta_1 \, \|\xi^{\gamma}\|_{\nu_1^{\gamma}}+ \eta^2_1 \, \|\xi^{\gamma}\|^2_{\nu_1^{\gamma}}\right] +C'_{N,\cs }\,,
\]
where $\xi^{\gamma}(x_0,x_1) = \frac{|\zeta_{x_0}^{\gamma}(x_1)|^2}{|\Im \zeta_{x_0}^{\gamma}(x_1)|}$, $n_x^{\gamma} = (\overline{m_x^{\gamma}})(m_x^{\gamma})^{-1}$ and $C'_{N,\cs } = \frac{8\cs^2C_{N,\cs, 2 }(I)}{c(D, \beta)}$. 

Moreover, there is an explicit $f(\beta, D)$, depending only on the spectral gap $\beta$ and on the degree, such that $c_{\cs, \beta}\leq f(\beta, D)M^{3}$ as $M\to +\infty$.
\end{enumerate}
In particular, in case \emph{(ii)},
\begin{equation}        \label{eq:uu}
\|u_{x_0}^{\gamma}(x_1)u_{x_1}^{\gamma}(x_0) - s_1s_2\|_{\nu_1^{\gamma}}^2 \le 4c_{\cs,\beta } \left[\varepsilon^{1/2} + \eta_1 \, \|\xi^{\gamma}\|_{\nu_1^{\gamma}}+ \eta^2_1 \, \|\xi^{\gamma}\|^2_{\nu_1^{\gamma}}\right] + 4C'_{N,\cs}
\, .
\end{equation}
\end{prp}
 
\begin{proof}
(a) We start with some preliminary inequalities. Denote $\la\cdot, \cdot\ra_{\nu}=\la\cdot, \cdot\ra_{\nu_1^{\gamma}}$.

Recall that we denote by $F$ the space of functions on $B$ which depend only on the terminus.

Let $\delta_1=\frac34\cs^{-2}$, $K_{\gamma}=A_{\gamma}K$ and let $w=P_{F^{\bot}} K_{\gamma}$ be the orthogonal projection of $K_{\gamma}$ on $F^{\bot}$ in $\ell^2(\nu_1^{\gamma})$. By the proof of Proposition~\ref{lem:sp},
\[
\langle w, \mathcal{M}^\gamma  w \rangle_{\nu} \ge \delta_1\,\|w\|_{\nu}^2 - C_{N,M}(A_\gamma) \|K\|_{\infty}^2.
\]
By Remark~\ref{rem:F} and the fact that $\mathcal{M}^{\gamma \ast}=\mathcal{M}^\gamma $, we have
\[
\langle w,\mathcal{M}^\gamma  w\rangle_{\nu} = \langle K_{\gamma}, \mathcal{M}^\gamma  K_{\gamma} \rangle _{\nu} \le \|K_{\gamma}\|_{\nu}^2 - \|\cS_\gamma K_{\gamma}\|_{\nu}^2 \, .
\]
So if $f = P_F K_{\gamma}=K_{\gamma}-w \in F$ is the projection of $K_{\gamma}$ on $F$, we have
\begin{equation}       \label{eq:Su1}
\|K_{\gamma} - f\|_{\nu}^2 \le  \delta_1^{-1} \left( \|K_{\gamma}\|_{\nu}^2 - \|\mathcal{S}_\gamma K_{\gamma}\|_{\nu}^2 + C_{N,M}(A_\gamma)\|K\|_{\infty}^2 \right)  \, .
\end{equation}
Similarly, if $\delta_2 = \cs^{-2}c(D, \beta)$ and $C\,\mathbf{1} = P_{\mathbf{1}} |K_{\gamma}|$ is the projection of $|K_{\gamma}|$ on $\mathbf{1}$, then using Proposition~\ref{lem:7.5b}, we get
\begin{equation}       \label{eq:Su2}
\left\| \,|K_{\gamma}|-C\,\mathbf{1}\right\|_{\nu}^2 \le \delta_2^{-1}\left(\|K_{\gamma}\|_{\nu}^2 - \|\cS_{\gamma}^2\,|K_{\gamma}|\,\|_{\nu}^2 + C_{N,\cs,2}(A_\gamma)\|K\|_{\infty}^2\right) .
\end{equation}
Now
\[
\left\|K_{\gamma} - \|K_{\gamma}\|_{\nu}\frac{f}{|f|}\right\|_{\nu} \le \left\|K_{\gamma}-f\right\|_{\nu} + \left\|f-\|K_{\gamma}\|_{\nu}\frac{f}{|f|}\right\|_{\nu}
\]
and
\[
\left\|f-\|K_{\gamma}\|_{\nu}\frac{f}{|f|}\right\|_{\nu} = \left\|\,|f|-\|K_{\gamma}\|_{\nu}\,\mathbf{1}\right\|_{\nu}
\]
(this is true even if $f$ vanishes, if we give an arbitrary value of modulus $1$ to $\frac{f}{|f|}$ in this case).
Also,
\[
\left\|\,|f|-\|K_{\gamma}\|_{\nu}\,\mathbf{1}\right\|_{\nu}  \le \left\|\,|K_{\gamma}| - |f| \right\|_{\nu} + \left\|\,|K_{\gamma}| - \|K_{\gamma} \|_{\nu}\,\mathbf{1}\right\|_{\nu}
\]
and
\[
\left\|\,|K_{\gamma}| - |f| \right\|_{\nu} \le \|K_{\gamma} - f \|_{\nu} \, .
\]
Finally, $\left\|\,|K_{\gamma}| - \|K_{\gamma}\|_{\nu}\,\mathbf{1}\right\|_{\nu} \le \left\|\,|K_{\gamma}| - C\,\mathbf{1}\right\|_{\nu} + \left|\,\|K_{\gamma}\|_{\nu}-C\,\right| \le  2\, \left\|\,|K_{\gamma}|-C\,\mathbf{1}\right\|_{\nu}$. Putting all these inequalities together, we obtain
\begin{equation}       \label{eq:Su3}
\left\| K_{\gamma} - \|K_{\gamma}\|_{\nu}\frac{f}{|f|}\right\|_{\nu} \le 2 \left\|K_{\gamma} - f\right\|_{\nu} + 2 \left\|\,|K_{\gamma}|-C\,\mathbf{1}\right\|_{\nu} \, .
\end{equation}
Comparing with \eqref{eq:Su1} and \eqref{eq:Su2}, this says the following~: if $ \|\cS_{\gamma}^2\,|K_{\gamma}|\,\|_{\nu}$ is close to $\|K_{\gamma}\|_{\nu}$ and if $\|\mathcal{S}_{\gamma} K_{\gamma}\|_{\nu}$ is close to $\|K_{\gamma}\|_{\nu}$,
then $K_{\gamma}$ must be close to $\|K_{\gamma}\|_{\nu}\frac{f}{|f|}$, where $f$ is a function that depends only on the terminus.

Repeating the arguments of \eqref{eq:Su1} with $M_{u^{\gamma}}\mathcal{S}_{u^{\gamma}}K_{\gamma}$ instead of $K_{\gamma}$, then taking $\tilde{f} = P_F M_{u^{\gamma}} \mathcal{S}_{u^{\gamma}} K_{\gamma} \in F$, we get
\begin{equation}       \label{eq:Su4}
\|M_{u^{\gamma}}\mathcal{S}_{u^{\gamma}}K_{\gamma} - \tilde{f}\|_{\nu}^2 \le  \delta_1^{-1} \left( \|\mathcal{S}_{u^{\gamma}}K_{\gamma}\|_{\nu}^2 - \|\mathcal{S}_{u^{\gamma}}^2K_{\gamma}\|_{\nu}^2 + C_{N,M}( \mathcal{S}_{u^{\gamma}}A_\gamma)\|K\|_{\infty}^2 \right) \, .
\end{equation}
Similarly to \eqref{eq:Su2}, if $\tilde{C}\,\mathbf{1} = P_{\mathbf{1}} |\mathcal{S}_{u^{\gamma}}K_{\gamma}|$, we get
\begin{equation}       \label{eq:Su5}
\left\| \,|\mathcal{S}_{u^{\gamma}}K_{\gamma}|-\tilde{C}\,\mathbf{1}\right\|_{\nu}^2 \le \delta_2^{-1}\left(\left\|\mathcal{S}_{u^{\gamma}}K_{\gamma}\right\|_{\nu}^2 - \left\|\cS_{\gamma}^2\,|\mathcal{S}_{u^{\gamma}}K_{\gamma}|\,\right\|_{\nu}^2 + C_{N,\cs,2}( \mathcal{S}_{u^{\gamma}}A_\gamma)\|K\|_{\infty}^2\right) \, .
\end{equation}
Finally, arguing as in (\ref{eq:Su3}), we have
\begin{multline}       \label{eq:Su6}
\left\| M_{u^\gamma}\mathcal{S}_{u^{\gamma}}K_{\gamma} - \|K_{\gamma}\|_{\nu}\frac{\tilde{f}}{|\tilde{f}|}\right\|_{\nu} \\
\le 2 \left\| M_{u^\gamma}\mathcal{S}_{u^{\gamma}}K_{\gamma} - \tilde{f}\right\|_{\nu} + 2 \left\|\,|\mathcal{S}_{u^{\gamma}}K_{\gamma}|-\tilde{C}\,\mathbf{1}\right\|_{\nu} +  \|K_{\gamma}\|_{\nu}-\|\mathcal{S}_{u^{\gamma}}K_{\gamma}\|_{\nu}  \, .
\end{multline}

\medskip

(b) We can now start the proof itself. Suppose (i) is not true~:
\[
 \| \mathcal{S}_{u^{\gamma}}^{4}K_{\gamma}\|_{\nu}^2 > (1-\varepsilon)^2\|K_{\gamma}\|_{\nu}^2 + \tilde{C}_{N,\cs, 2}(A_\gamma)\cdot \|K\|_{\infty}^2.
\]
Using $\| \mathcal{S}_{u^{\gamma}}^{4}K_{\gamma}\|_{\nu}\le   \|\mathcal{S}_{u^{\gamma}}K_{\gamma}\|_{\nu} = \|\mathcal{S}_\gamma M_{u^\gamma}K_{\gamma}\|_{\nu}$, $\| \mathcal{S}_{u^{\gamma}}^{4}K_{\gamma}\|_{\nu} \le \|\cS_{\gamma}^2\,|K_{\gamma}|\|_{\nu}=\|\cS_{\gamma}^2\,| M_{u^\gamma} K_{\gamma}|\|_{\nu}$, $\| \mathcal{S}_{u^{\gamma}}^{4}K_{\gamma}\|_{\nu}  \le \|\cS_{\gamma}^2\,|\mathcal{S}_{u^{\gamma}}K_{\gamma}|\|_{\nu}$ and $\|K_{\gamma}\|_{\nu}\ge\|\mathcal{S}_{u^{\gamma}}K_{\gamma}\|_{\nu}$, we see that we must also have
\[
\|K_{\gamma}\|_{\nu}^2 - \|\mathcal{S}_\gamma M_{u^\gamma}K_{\gamma}\|^2_{\nu}< 2\varepsilon\,\|K_{\gamma}\|_{\nu}^2 - \tilde{C}_{N,\cs, 2}(A_\gamma)\cdot \|K\|_{\infty}^2 \, ,
\]
\[
\|K_{\gamma}\|_{\nu}^2-\|\cS_{\gamma}^2\,| M_{u^\gamma} K_{\gamma}|\|^2_{\nu} < 2\varepsilon\,\|K_{\gamma}\|^2_{\nu} - \tilde{C}_{N,\cs, 2}(A_\gamma)\cdot \|K\|^2_{\infty}
\]
\[
\|\mathcal{S}_{u^{\gamma}}K_{\gamma}\|_{\nu}^2-\|\cS_{\gamma}^2\,|\mathcal{S}_{u^{\gamma}}K_{\gamma}|\|^2_{\nu} < 2\varepsilon\,\|\mathcal{S}_{u^{\gamma}}K_{\gamma}\|^2_{\nu} - \tilde{C}_{N,\cs, 2}(A_\gamma)\cdot \|  K\|^2_{\infty}
\]
as well as
\[
\|\mathcal{S}_{u^{\gamma}}K_{\gamma}\|_{\nu}^2-\| \mathcal{S}_{u^{\gamma}}^{2}K_{\gamma}\|^2_{\nu} < 2\varepsilon\,\|\mathcal{S}_{u^{\gamma}} K_{\gamma}\|^2_{\nu} - \tilde{C}_{N,\cs, 2}(A_\gamma)\cdot \|  K\|^2_{\infty} \, .
\]
Applying (\ref{eq:Su1}), (\ref{eq:Su2}) and (\ref{eq:Su3}) to $M_{u^\gamma}K_{\gamma}$ instead of $K_{\gamma}$, and $f=P_F M_{u^\gamma} K_{\gamma}$, it follows that
\begin{equation}      \label{eq:Su7}
\left\| M_{u^\gamma} K_{\gamma} - \|K_{\gamma}\|_{\nu}\frac{f}{|f|}\right\|_{\nu}^2 \le 16(\delta_1^{-1}+ \delta_2^{-1})\, \varepsilon\cdot\norm{K_{\gamma}}^2_{\nu}  \, .
\end{equation}
Applying \eqref{eq:Su4}, \eqref{eq:Su5} and \eqref{eq:Su6} yields
\begin{equation}       \label{eq:Su8}
\left\|  M_{u^\gamma}\mathcal{S}_{u^{\gamma}}K_{\gamma} - \|K_{\gamma}\|_{\nu}\frac{\tilde{f}}{|\tilde{f}|}\right\|_{\nu}^2 \le  24(\delta_1^{-1}+ \delta_2^{-1})\, \varepsilon\cdot\norm{K_{\gamma}}^2_{\nu}+ 3\varepsilon \cdot \|K_{\gamma}\|_{\nu}^2\, .
\end{equation}
As $f,\tilde{f}\in F$, we have $\frac{f}{|f|}(x_0,x_1) = e^{i\theta(x_1)}$ and $\frac{\tilde{f}}{|\tilde{f}|}(x_0,x_1) = e^{i\theta'(x_1)}$ for some $\theta,\theta':V\to \R$. Note that in this case, $(\mathcal{S}_\gamma \frac{f}{|f|})(x_0,x_1) =  e^{i\theta(x_0)} - \eta_1 \xi^{\gamma}(x_1,x_0) e^{i\theta(x_0)} $, where $\xi^{\gamma}(x_0,x_1) = \frac{|\zeta_{x_0}^{\gamma}(x_1)|^2}{|\Im \zeta_{x_0}^{\gamma}(x_1)|}$, using \eqref{e:sumzeta2}. Applying $\cS_\gamma$ to \eqref{eq:Su7}, we thus get
\begin{align*}
\left\|\mathcal{S}_{u^{\gamma}}K_{\gamma}- \|K_{\gamma}\|_{\nu} e^{i\theta(x_0)}\right\|_{\nu}^2 & \le 2 \left\|\mathcal{S}_\gamma M_{u^\gamma}K_{\gamma}-\|K_{\gamma}\|_{\nu}\mathcal{S}_\gamma \frac{f}{|f|}\,\right\|_{\nu}^2 + 2 \eta_1^2\, \|\xi^{\gamma}\|_{\nu}^2\cdot\|K_{\gamma}\|_{\nu}^2 \\
& \leq  32(\delta_1^{-1}+ \delta_2^{-1})\, \varepsilon\cdot\norm{K_{\gamma}}^2_{\nu} + 2 \eta_1^2\, \|\xi^{\gamma}\|_{\nu}^2\cdot\|K_{\gamma}\|_{\nu}^2,
\end{align*}
Applying $M_{u^\gamma} $ and comparing with \eqref{eq:Su8}, it follows that
\begin{equation}\label{e:firstu}
\left\|\overline{u_{x_1}^{\gamma}(x_0)}e^{i\theta(x_0)} - e^{i\theta'(x_1)} \right\|_{\nu}^2 \le  (2\times32+2\times 24)(\delta_1^{-1}+ \delta_2^{-1})\cdot \varepsilon  + 4 \eta_1^2\, \|\xi^{\gamma}\|_{\nu}^2   + 6\varepsilon    \, .
\end{equation}
Repeating the procedure with $K_{\gamma}$ replaced by $\mathcal{S}_{u^{\gamma}} K_{\gamma}$, and $f$ replaced by $\tilde f$, the same arguments show that there exists $\theta'':V\to \R$ such that
\begin{equation}\label{e:secondu}
\left\|\overline{u_{x_1}^{\gamma}(x_0)}e^{i\theta'(x_0)} - e^{i\theta''(x_1)} \right\|_{\nu}^2 \le (112\delta_1^{-1}+ 112\delta_2^{-1}+6)\cdot \varepsilon  + 4 \eta_1^2 \|\xi^{\gamma}\|_{\nu}^2    \, .
\end{equation}
Hence we have proved that $u_{x_1}^{\gamma}(x_0)$ is close to both $e^{i(\theta(x_0)-\theta'(x_1))}$ and $e^{i(\theta'(x_0)-\theta''(x_1))}$.

(c) Because of relation \eqref{eq:mv}, the function $u$ satisfies $u_{x_1}^{\gamma}(x_0) = u_{x_0}^{\gamma}(x_1) \frac{n_{x_1}^{\gamma}}{n_{x_0}^{\gamma}}$, where $n_x^{\gamma} =  (\overline{m_x^{\gamma}})(m_x^{\gamma})^{-1}$. 

To conclude the proof, we show~: if $e^{i(\theta(x_0)-\theta'(x_1))}$ and $e^{i(\theta'(x_0)-\theta''(x_1))}$ are close to $u^\gamma$, and if the function $u_{x_1}^{\gamma}(x_0)$ satisfies
the relation above, then this gives constraints on $\theta, \theta', \theta''$ that imply part (ii) of the proposition.

Let $g(x_0,x_1) = e^{i(\theta(x_0)-\theta'(x_1))}$ and ${\mathbf{c}}=(112\delta_1^{-1}+ 112\delta_2^{-1}+6)$. We have shown in (b) that $\|u_{x_1}^{\gamma}(x_0)-g\|_{\nu}^2 \le {\mathbf{c}}\varepsilon + 4\eta_1^2\|\xi^{\gamma}\|_{\nu}^2$. Recall that we denote by $\iota$ the involution of edge reversal.
Hence, if we define $\tilde{g}(x_0,x_1) = g(x_1,x_0) \frac{n_{x_1}^{\gamma}}{n_{x_0}^{\gamma}}$, we get 
\begin{equation}\label{e:ug}
\|\tilde{g} - u_{x_1}^{\gamma}(x_0)\|_{\nu}^2 = \|\iota g - u_{x_0}^{\gamma}(x_1)\|_{\nu}^2 \le {\mathbf{c}}\varepsilon + 4\eta_1^2\, \|\xi^{\gamma}\|_{\nu}^2 \, .
\end{equation}
Thus, $\|\tilde{g} - g\|_{\nu}^2 \le 4{\mathbf{c}}\varepsilon + 16\eta_1^2\,\|\xi^{\gamma}\|_{\nu}^2$. Hence, defining
\[
h_1(x_0,x_1) = n_{x_1}^{\gamma} e^{i[\theta(x_1)+\theta'(x_1)]} \quad \text{and} \quad h_2(x_0,x_1) = n_{x_0}^{\gamma}e^{i[\theta(x_0)+\theta'(x_0)]} \, ,
\]
we get
\[
\| h_1-h_2\|_{\nu}^2 = \|\tilde{g}-g\|_{\nu}^2 \le 4{\mathbf{c}}\varepsilon + 16\eta_1^2 \,\|\xi^{\gamma}\|_{\nu}^2 \, .
\]
Note that the functions $h_1, h_2$ have modulus $1$, and $\cS_\gamma  h_1 = h_2 - \eta_1 \iota \xi^{\gamma} h_2$, so
\[
\|\cS_{\gamma}^2h_1 - h_1\|_{\nu} \le 2 \,\|\cS_\gamma  h_1-h_1\|_{\nu} \le 2\left(\|h_2-h_1\|_{\nu} + \eta_1 \|\xi^{\gamma}\|_{\nu}\right)\leq 4{\mathbf{c}}^{1/2}\varepsilon^{1/2} + 8\eta_1 \,\|\xi^{\gamma}\|_{\nu} \, .
\]
Consider $P_{\mathbf{1}, \nu} h_1=s \,\mathbf{1}$, the projection of $h_1$ to the space of constant functions. Arguing as in \eqref{eq:Su2}, we can write $\|h_1 - s\,\mathbf{1}\|_{\nu}^2 \le \delta_2^{-1}(\|h_1\|_{\nu}^2 - \|\cS_{\gamma}^2h_1\|_{\nu}^2 + 4C_{N,\cs,2}(I))$ . But $\|h_1\|^2-\|\cS_{\gamma}^2h_1\|^2 = (\|h_1\|+\|\cS_{\gamma}^2h_1\|)(\|h_1\|-\|\cS_{\gamma}^2h_1\|) \le 2\,\|\mathcal{S}_{\gamma}^2h_1-h_1\|$. Hence,
\begin{align*}
\|h_1-s\,\mathbf{1}\|_{\nu}^2  \le 8\delta_2^{-1}{\mathbf{c}}^{1/2}\varepsilon^{1/2} + 16\eta_1\delta_2^{-1} \|\xi^{\gamma}\|_{\nu} +4 \delta_2^{-1}C_{N,\cs,2}(I)
%\\
%& \le C_{\delta_1,\delta_2,j_0} \varepsilon^{1/2} + C_{\delta_2,j_0} \cdot \eta_1 \cdot \|Z\|_{\nu} + \delta_2^{-1}C_{N,\cs,2} \, ,
\end{align*}
%where $C_{\delta_1,\delta_2,j_0} = 2^{3/2}\delta_2^{-1}j_0{\mathbf{c}}^{1/2}$ and $C_{\delta_2,j_0} = 6j_0\delta_2^{-1}$.
We observe that $\|h_1-s\,\mathbf{1}\| = \|n_{x_1}^{\gamma}e^{i(\theta(x_1)+\theta'(x_1))}-s\,\mathbf{1}\| = \|\tilde{g}n_{x_0}^{\gamma}e^{i(\theta'(x_0)+\theta'(x_1))}-s\,\mathbf{1}\| = \|\tilde{g}-\frac{e^{-i(\theta'(x_0)+\theta'(x_1))}}{n_{x_0}^{\gamma}}s\|$. Thus, comparing with \eqref{e:ug},
\begin{align*}
\Big\|u_{x_1}^{\gamma}(x_0) - s\,\frac{e^{-i(\theta'(x_0)+\theta'(x_1))}}{n_{x_0}^{\gamma}}\Big\|_{\nu}^2 & \le
16\delta_2^{-1}{\mathbf{c}}^{1/2}\varepsilon^{1/2} + 32\eta_1\delta_2^{-1} \|\xi^{\gamma}\|_{\nu} \\
& \quad +8 \delta_2^{-1}C_{N,\cs,2}(I)
+2 {\mathbf{c}}\varepsilon + 8\eta_1^2 \|\xi^{\gamma}\|_{\nu}^2
\end{align*}
This is the first half of (ii) with 
\begin{equation}\label{e:csb}
c_{\cs, \beta}=\max\{16\delta_2^{-1}{\mathbf{c}}^{1/2},2 {\mathbf{c}} , 32\delta_2^{-1}, 8\} \, .
\end{equation}
 Remembering that $\delta_1=\frac34 M^{-2}$, $\delta_2 = \cs^{-2}c(D, \beta)$ and ${\mathbf{c}}=(112\delta_1^{-1}+ 112\delta_2^{-1}+6)$, we see that there is an explicit $f(\beta, D)$ such that $c_{\cs, \beta}\leq f(\beta, D)M^{3}$ as $M\to +\infty$. Note that $|s| \le 1$ since $\|h_1\|_{\nu} =1$.

The second half of (ii) is proven similarly, using \eqref{e:secondu} instead of \eqref{e:firstu}. Here we take $g'(x_0,x_1)=e^{i(\theta'(x_0)-\theta''(x_1))}$, $h_1'(x_0,x_1)=\frac{1}{n_{x_1}^{\gamma}}e^{-i[\theta'(x_1)+\theta''(x_1)]}$, $s'\mathbf{1}=P_{\mathbf{1}} h_1'$ and $h_2'(x_0,x_1)=\frac{1}{n_{x_0}^{\gamma}}e^{-i[\theta'(x_0)+\theta''(x_0)]}$.

To prove (\ref{eq:uu}), we write $\big\|u_{x_1}^{\gamma}(x_0)^2 - ss' \frac{n_{x_1}^{\gamma}}{n_{x_0}^{\gamma}} \big\|^2 \le 2\, \big\| u_{x_1}^{\gamma}(x_0)[u_{x_1}^{\gamma}(x_0) - s\frac{e^{-i\widetilde{\theta}(x_0,x_1)}}{n_{x_0}^{\gamma}}]\big\|^2  + 2 \,\big\| s\frac{e^{-i\widetilde{\theta}(x_0,x_1)}}{n_{x_0}^{\gamma}} [u_{x_1}^{\gamma}(x_0) - s'e^{i\widetilde{\theta}(x_0,x_1)}n_{x_1}^{\gamma}]\big\|^2$, where we put $\widetilde{\theta}(x_0,x_1) = \theta'(x_0)+\theta'(x_1)$. Since $u_{x_1}^{\gamma}(x_0)^2 \frac{n_{x_0}^{\gamma}}{n_{x_1}^{\gamma}} = u_{x_1}^{\gamma}(x_0)u_{x_0}^{\gamma}(x_1)$, the proof is complete.
\end{proof}

\section{Step 4~: End of the proof of Theorem \ref{t:thm4}\label{s:step4}}
Our aim is to show that $\lim_{\eta_0\downarrow 0}\lim_{N\to +\infty} {\mathrm{Var_{nb, \eta_0}^I}}(\cF_{\gamma}K)=0$, for the operators $\cF_{\gamma}$ that appear in Corollary~\ref{cor:recurrence}. A main step was carried out in Proposition~\ref{p:mainbound2}, and the upper bound was put in a convenient form in \eqref{e:nicer}. We now use the estimates of Sections~\ref{s:mixing} and \ref{sec:su} to complete the proof. We denote $B_{\gamma}=\frac{m^{\gamma}}{Z_{\gamma}}\cF_{\gamma}:\mathscr{H}_m\to\mathscr{H}_k$ as in Section~\ref{s:mixing}, where $Z_\gamma$ is defined in \eqref{e:ZK}. It should be kept in mind that $\cF_{\gamma}$ may depend on a parameter $T$ that is fixed in this section, but will be taken arbitrarily large in the next one.

Recall that we take $\gamma= \lambda+i(\eta^4+\eta_0)$, where $\lambda, \eta, \eta_0$ come from Proposition~\ref{p:mainbound2}. In other words,  $\gamma=\lambda+i\eta_1\in \C^+$ with $\lambda\in I_1$ and $\eta_1=\eta^4+\eta_0$. Let $K\in \mathscr{H}_m$ so that $B_\gamma K\in \mathscr{H}_k$. Applying \eqref{e:nicer}, recalling that $\nu_k^{\gamma} = \frac{1}{\mu_k^{\gamma}(B_k)} \mu_k^{\gamma}$, we obtain
\begin{multline}\label{e:sum}
\frac{1}{n^2}\sum_{r, r'=1}^n \langle \mathcal{R}_{n,r}^{\gamma}\cF_{\gamma}K,\mathcal{R}_{n,r'}^{\gamma}\cF_{\gamma}K \rangle_{\gamma} = \frac{\mu_k^{\gamma}(B_k)}{Nn^2}\sum_{r'\leq r\leq n}\la \cS_{u^\gamma}^{r-r'}B_\gamma K,B_\gamma K\ra_{\nu^\gamma_k}\\
+\frac{\mu_k^{\gamma}(B_k)}{Nn^2}\sum_{r< r'\leq n}\la B_\gamma K,\cS_{u^\gamma}^{r'-r}B_\gamma K\ra_{\nu^\gamma_k} + \frac{1}{n^2}\sum_{r, r'=1}^n \mathbf{E}_{n,r,r'}(\eta_1,\cF_\gamma K)  \, .
\end{multline}

%Extracting subsequences, we may always assume that the limit exists in $[0, +\infty]$, and we want to show that it vanishes. Thus we can safely extract further subsequences at any stage of the proof.

Fix $M$ very large and take $n= M^9$.  We apply Proposition \ref{p:Sualtmod} with $\varepsilon=M^{-8}$ to the family of operators $\{\cS^{4j}_{u^\gamma}  B_\gamma K\}_{j=1}^{M^9}$. Call ${\tilde{\tilde C}}_{N, M}(B_\gamma)=\max_{j=1}^{M^9}\tilde C_{N, M, 2}(\cS^{4j+k-1}_{u^\gamma}  B_\gamma )^{1/2}\cdot \sqrt{\frac{\mu_1^{\gamma}(B)}{\mu_k^{\gamma}(B_k)}}$. We use the notation in Remark~\ref{r:bigonotation} throughout the section. In particular, ${\tilde{\tilde C}}_{N, M}(B_\gamma) = O_T(M^{-\infty})_{N\To +\infty, \gamma}$ thanks to Corollary \ref{c:Minfty}.

\begin{rem}        \label{rem:Sunorm}
It is useful to note that the norm $\|\mathcal{S}_{u^{\gamma}}^j \|_{\nu_k^{\gamma} \to \nu_k^{\gamma}}$ for $k>1$ is controlled by the same norm for $k=1$. To see this, note that for $K\in \ell^2(\nu_k^{\gamma})$, we have $(\mathcal{S}_{u^{\gamma}}^{k-1} K)(x_0;x_k) = \sum_{(x_{-k+1};x_{-1})_{x_{0,1}}} \Lambda (x_{-k+1};x_1) K(x_{-k+1};x_1)$ for some function $\Lambda (x_{-k+1};x_1)$.
Here the sum is over those $(x_{-k+1};x_{-1})$ for which the path $(x_{-k+1}, x_{-k+2},\ldots, x_{-1}, x_0, x_1)$ does not backtrack, cf. \eqref{eq:KB2}. So $(\mathcal{S}_{u^{\gamma}}^{k-1} K)(x_0;x_k)$ only depends on $(x_0,x_1)$~: we may define $\phi_K\in\ell^2(\nu_1^{\gamma})$ by $\phi_K(x_0,x_1)=(\mathcal{S}_{u^{\gamma}}^{k-1} K)(x_0;x_k)$. If $\mathscr{I}: \ell^2(\nu_1^{\gamma}) \to \ell^2(\nu_k^{\gamma})$ is the map $(\mathscr{I}\phi)(x_0;x_k) = \phi(x_0,x_1)$, we have for any $j \ge k$, $[\mathcal{S}_{u^{\gamma}}^{j-k+1} \mathscr{I}\phi_K](x_0;x_k) = (\mathcal{S}_{u^{\gamma}}^jK)(x_0;x_k)$. Moreover, $[\mathcal{S}_{u^{\gamma}}\mathscr{I}\phi](x_0;x_k) = [\mathscr{I}(\mathcal{S}_{u^{\gamma}}\phi)](x_0;x_k)$. Thus,
\[
\|\mathcal{S}_{u^{\gamma}}^jK\|_{\nu_k}^2 = \|\mathcal{S}_{u^{\gamma}}^{j-k+1}\mathscr{I}\phi_K\|_{\nu_k}^2 = \|\mathscr{I}(\mathcal{S}_{u^{\gamma}}^{j-k+1}\phi_K)\|_{\nu_k}^2 \le \frac{\mu_1^{\gamma}(B)}{\mu_k^{\gamma}(B_k)} \,\|\mathcal{S}_{u^{\gamma}}^{j-k+1}\phi_K\|_{\nu_1}^2 \, ,
\]
where we used that $\sum_{_{x_{0,1}}(x_2;x_k)} \mu_k(x_0;x_k) \le \mu_1(x_0,x_1)$ by \eqref{e:compat}. Hence,
\[
\|\mathcal{S}_{u^{\gamma}}^jK\|_{\nu_k}^2 \le \frac{\mu_1^{\gamma}(B)}{\mu_k^{\gamma}(B_k)}\,\|\mathcal{S}_{u^{\gamma}}^{j-k+1}\|_{\nu_1 \to \nu_1}^2 \cdot \|\phi_K\|_{\nu_1}^2 \, .
\]
But using \eqref{eq:sumzeta} repeatedly we have
\begin{multline*}
\sum_{(x_{-k+1};x_{-1})_{x_{0,1}}} |\Lambda(x_{-k+1};x_1)| \\
=\sum_{(x_{-k+1},x_{-1})_{x_{0,1}}} \frac{|\zeta_{x_1}^{\gamma}(x_0)\zeta_{x_0}^{\gamma}(x_{-1})\cdots\zeta_{x_{-k+3}}^{\gamma}(x_{-k+2})|^2|\Im\zeta_{x_{-k+2}}^{\gamma}(x_{-k+1})|}{|\Im \zeta_{x_1}^{\gamma}(x_0)|} \le 1\,,
\end{multline*}
and $\mu_1^{\gamma}(x_0,x_1)|\Lambda(x_{-k+1};x_1)| = \mu_k^{\gamma}(x_{-k+1};x_1)$ by \eqref{e:muk} and \eqref{eq:mv}. Hence,
\begin{align*}
\|\phi_K\|_{\mu_1}^2 & = \sum_{(x_0,x_1)}\mu_1^{\gamma}(x_0,x_1)\Big|\sum_{(x_{-k+1};x_{-1})_{x_{0,1}}} \Lambda(x_{-k+1};x_1) K(x_{-k+1};x_1)\Big|^2 \\
& \le \sum_{(x_0,x_1)} \mu_1^{\gamma}(x_0,x_1) \sum_{(x_{-k+1};x_{-1})_{x_{0,1}}}  |\Lambda(x_{-k+1};x_1)|\cdot |K(x_{-k+1};x_1)|^2 \\
& = \sum_{(x_{-k+1};x_1)} \mu_k^{\gamma}(x_{-k+1};x_1) \cdot |K(x_{-k+1};x_1)|^2 = \|K\|_{\mu_k}^2 \, .
\end{align*}
So $\|\phi_K\|_{\nu_1}^2\le \frac{\mu_k^{\gamma}(B_k)}{\mu_1^{\gamma}(B)}\|K\|_{\nu_k^{\gamma}}^2$. Summarizing, we have shown that for any $j \ge k$, we have
\[
\|\mathcal{S}_{u^{\gamma}}^j \|_{\nu_k \to \nu_k} \le \|\mathcal{S}_{u^{\gamma}}^{j-k+1}\|_{\nu_1 \to \nu_1} \, .
\]
\end{rem}
 
\bigskip

{\bf{First alternative~:}}
For $\gamma$, $\varepsilon$ as above, assume that 
%there is a subsequence $(G_{N_k})$ such that
case (i) of Proposition \ref{p:Sualtmod} is satisfied for all the operators $\{\cS^{4j}_{u^\gamma}  B_\gamma K\}_{j=1}^{M^9}$. Applying \eqref{e:decay} for $\cS^{4t}_{u^\gamma}  B_\gamma K$, $t\le j$, we obtain if $k=1$,
\begin{equation}\label{e:bigdecay}
\| \mathcal{S}_{u^{\gamma}}^{4j}B_\gamma K\|_{\nu_1^{\gamma}} \le (1-\varepsilon)^{j}\|B_\gamma K\|_{\nu_1^{\gamma}} +j \max_{1\le t\le j}\{\tilde C_{N, M, 2}(\cS^{4t}_{u^\gamma}  B_\gamma )^{1/2}\}\cdot \|K\|_{\infty} \, .
\end{equation}
For higher $k$, we apply \eqref{e:bigdecay} to $\phi_{B_{\gamma}K}(x_0, x_1)=(\mathcal{S}_{u^{\gamma}}^{k-1} B_\gamma K)(x_0;x_k)=(A_{\gamma}K)(x_0,x_1)$, where $A_{\gamma}=\cS_{u^{\gamma}}^{k-1}B_{\gamma}$, instead of $B_{\gamma}K$. We get by Remark~\ref{rem:Sunorm},
\[
\| \mathcal{S}_{u^{\gamma}}^{4j+k-1}B_\gamma K\|_{\nu_k^{\gamma}} \le (1-\varepsilon)^{j}\|B_\gamma K\|_{\nu_k^{\gamma}} +j {\tilde{\tilde C}}_{N, M}(B_\gamma) \cdot \|K\|_{\infty} \, .
\]
Using the euclidean division $r'-r-k+1=4m_{r,r'}+n_{r,r'}$ with $n_{r,r'}<4$, we see that for $r'-r\ge 4+k-1$,
\[
|\la B_\gamma K,\cS_{u^\gamma}^{r'-r}B_\gamma K\ra_{\nu^\gamma_k}|\leq c_k(1-\varepsilon)^{(r'-r)/4}\|B_\gamma K\|^2_{\nu_k^{\gamma}} +n{\tilde{\tilde C}}_{N, M}(B_\gamma)\cdot \|K\|_{\infty}\|B_\gamma K\|_{\nu_k^{\gamma}} \, ,
\]
where $c_k=\frac{1}{(1-\varepsilon)^{(k-1+n_{r,r'})/4}} \le 2^{\frac{k+2}{4}}$ if $\varepsilon \le \frac{1}{2}$. Note that $(1-\varepsilon)^{1/4} \le (1-\frac{\varepsilon}{5})$. Hence, since $4+k-1\le 4k$, we have
\begin{align*}
&\Big| \sum_{ r'\leq n}\sum_{r< r'}\la B_\gamma K,\cS_{u^\gamma}^{r'-r} B_\gamma K\ra_{\nu^\gamma_k} \Big|  \le \Big(\sum_{ r'\leq n}\sum_{r\le r'-4k}|\la B_\gamma K,\cS_{u^\gamma}^{r'-r}B_\gamma K\ra_{\nu^\gamma_k}| +4nk\|B_{\gamma}K\|_{\nu_k}^2\Big)\\
& \qquad \leq \Big[4nk+nc_k\sum_{m=1}^n (1-\varepsilon)^{m/4}\Big]
\|B_\gamma K\|^2_{\nu_k^{\gamma}} +n^3{\tilde{\tilde C}}_{N, M}(B_\gamma)\cdot \|K\|_{\infty}\|B_\gamma K\|_{\nu_k^{\gamma}} \\
& \qquad \leq \frac{n(5c_k+4k)}{\varepsilon}
\|B_\gamma K\|^2_{\nu_k^{\gamma}} +n^3{\tilde{\tilde C}}_{N, M}(B_\gamma)\cdot \|K\|_{\infty}\|B_\gamma K\|_{\nu_k^{\gamma}} \,.
\end{align*}
%Thus
%\begin{align*}
%& \frac{1}{n^2}\sum_{r, r'=1}^n \left\langle \mathcal{R}_{n,r}^{z+i\eta_0}GK^{z+i\eta_0},\mathcal{R}_{n,r'}^{z+i\eta_0}GK^{z+i\eta_0}\right\rangle_{z+i\eta_0}\\
%& \qquad \leq \frac{8}{n\varepsilon}
%\|B_\gamma K\|^2_{\ell^2(\nu_1^{\gamma})} +2n{\tilde{\tilde C}}_{N, M}(B_\gamma)\cdot \|K\|_{\infty}\|B_\gamma K\|_{\ell^2(\nu_1^{\gamma})} + \frac{1}{n^2}\sum_{r, r'=1}^n \mathrm{O}_{n,r,r'}(\gamma,GK^\gamma)
%\end{align*}
Recall that $\varepsilon=M^{-8}$ and $n= M^9$. Comparing with \eqref{e:sum}, we get
\begin{multline}     \label{e:1stalt}
\Big\|\frac{1}{n}\sum_{r=1}^n \mathcal{R}_{n,r}^{\gamma}\cF_{\gamma}K\Big\|_{\gamma}^2
\leq  \frac{\mu_k^{\gamma}(B_k)}{N}\Big(\frac{c_k'}{M}\|B_\gamma K\|^2_{\nu_k^{\gamma}}+ M^9 {\tilde{\tilde C}}_{N, M}(B_\gamma)\cdot \|K\|_{\infty}\|B_\gamma K\|_{\nu_k^{\gamma}}\Big)
\\ + \frac{1}{n^2}\sum_{r, r'=1}^n \mathbf{E}_{n,r,r'}(\eta_1,\cF_\gamma K)\, .
\end{multline}

 \bigskip

{\bf{Second alternative~:}}
Now assume case (ii) of Proposition \ref{p:Sualtmod} is satisfied; with some complex numbers $s_j=s_j(N)$ and some function $\theta$. We denote $\|\,\|_{\nu}=\|\,\|_{\ell^2(\nu_k^{\gamma})}$, $\theta_0(x_0; x_k)=\theta(x_0)$, $\theta_1(x_0; x_k)=\theta(x_1)$, $n_0^\gamma(x_0; x_k)=n^\gamma_{x_0}$ and $n_1^\gamma(x_0; x_k)=n^\gamma_{x_1}$. Then we have

\begin{prp}
Let $\|K\|_{\infty}\le 1$. For $A_\gamma K=\cS_{u^\gamma}^\ell B_\gamma K$, we have for any $t\in\N^{\ast}$,
\begin{multline*}
\left|\la B_\gamma K, \cS_{u^\gamma}^{2t} A_\gamma K\ra_{\nu} -(\overline{s_1s_2})^t\la B_\gamma K,  e^{i \theta_0}\cS_{\gamma}^{ 2t} e^{-i \theta_0}A_\gamma K\ra_{\nu} \right|
\\ \leq t\left( c_{\cs,\beta} \left[\varepsilon^{1/2} + \eta_1 O(1)_{N\To +\infty, \gamma}\right] +C'_{N,\cs }\right)^{1/4}O_T(1)_{N\To +\infty, \gamma} \, .
\end{multline*}
\end{prp}
\begin{proof}
Recall that $\mathcal{S}_{u^{\gamma}} = \mathcal{S}_{\gamma}M_{u^{\gamma}}$ with $M_{u^{\gamma}}$ the multiplication by $\overline{u_{x_1}^{\gamma}(x_0)}$. We have
\begin{multline*}
\left\Vert\cS^2_{u^\gamma} A_\gamma K-\overline{s_1s_2}  e^{i\theta_0}\cS_\gamma^2  e^{-i\theta_0}A_\gamma K
\right\Vert_\nu =  \left\Vert\cS^2_{u^\gamma} A_\gamma K-\overline{s_1s_2} \cS_\gamma n_0^{\gamma}e^{i[\theta_0+\theta_1]}\mathcal{S}_{\gamma}\frac{e^{-i[\theta_0+\theta_1]}}{n_1^{\gamma}}A_{\gamma}K
\right\Vert_\nu \\
\le  \left\Vert\cS_{\gamma}M_{u^{\gamma}}\cS_{\gamma}M_{u^{\gamma}} A_\gamma K-\overline{s_2} \cS_\gamma n_0^{\gamma}e^{i[\theta_0+\theta_1]}\mathcal{S}_{\gamma} M_{u^{\gamma}}A_{\gamma}K
\right\Vert_{\nu} \\
+  \left\Vert \overline{s_2} \cS_\gamma n_0^{\gamma}e^{i[\theta_0+\theta_1]}\mathcal{S}_{\gamma} M_{u^{\gamma}}A_{\gamma}K-\overline{s_1s_2} \cS_\gamma n_0^{\gamma}e^{i[\theta_0+\theta_1]}\mathcal{S}_{\gamma}\frac{e^{-i[\theta_0+\theta_1]}}{n_1^{\gamma}}A_{\gamma}K
\right\Vert_{\nu} \, .
\end{multline*}
Using \eqref{eq:SPnorm} and Cauchy-Schwarz, the first term is bounded by
\[
\left\| \overline{u_{x_1}^{\gamma}(x_0)} - \overline{s_2} n_0^{\gamma} e^{i[\theta_0+\theta_1]}\right\|_{\ell^4(\nu_k^{\gamma})} \left\| \mathcal{S}_{\gamma}M_{u^{\gamma}} A_{\gamma}K\right\|_{\ell^4(\nu_k^{\gamma})}\, .
\]
But $u^\gamma, s_2, n_0^{\gamma}$ all have modulus bounded by $1$, so $|\overline{u_{x_1}^{\gamma}(x_0)} - \overline{s_2} n_0^{\gamma} e^{i[\theta_0+\theta_1]}|^4 \le 4 \,|\overline{u_{x_1}^{\gamma}(x_0)} - \overline{s_2} n_0^{\gamma} e^{i[\theta_0+\theta_1]}|^2$. Hence, $\| \overline{u_{x_1}^{\gamma}(x_0)} - \overline{s_2} n_0^{\gamma} e^{i[\theta_0+\theta_1]} \|_{\ell^4(\nu_1^{\gamma})} \le ( 4c_{\cs,\beta} \left[\varepsilon^{1/2} + \eta_1 O(1)_{N\To +\infty, \gamma}\right] +4C'_{N,\cs } )^{1/4}$ by the first part of (ii). For higher $k$, using $\sum_{_{x_{0,1}}(x_2;x_k)} \mu_k(x_0;x_k) \le \mu_1(x_0,x_1)$ by \eqref{e:compat}, we get $\| \overline{u_{x_1}^{\gamma}(x_0)} - \overline{s_2} n_0^{\gamma} e^{i[\theta_0+\theta_1]} \|_{\ell^4(\nu_k^{\gamma})} \le (\frac{\mu_1^{\gamma}(B)}{\mu_k^{\gamma}(B_k)})^{1/4}\| \overline{u_{x_1}^{\gamma}(x_0)} - \overline{s_2} n_0^{\gamma} e^{i[\theta_0+\theta_1]} \|_{\ell^4(\nu_1^{\gamma})}$.

Next, $\| \mathcal{S}_{\gamma}M_{u^{\gamma}} A_{\gamma}K\|_{\ell^4(\nu_k^{\gamma})} = \| \cS_{u^{\gamma}}^{\ell+1} B_{\gamma}K\|_{\ell^4(\nu_k^{\gamma})}$. Arguing as in Proposition~\ref{p:lesscrude} and Corollary~\ref{c:Minfty}, we see this is $O_T(1)_{N\To +\infty, \gamma}$. Bounding the second term similarly, we get
\begin{multline*}
\left\Vert\cS^2_{u^\gamma} A_\gamma K-\overline{s_1s_2}  e^{i\theta_0}\cS_\gamma^2  e^{-i\theta_0}A_\gamma K
\right\Vert_{\nu} \\
\le \left( c_{\cs,\beta} \left[\varepsilon^{1/2} + \eta_1 O(1)_{N\To +\infty, \gamma}\right] +C'_{N,\cs }\right)^{1/4}O_T(1)_{N\To +\infty, \gamma}  \, .
\end{multline*}
Since $\|B_{\gamma}K\|_{\nu} = O_T(1)_{N\To +\infty, \gamma}$ (see Remark~\ref{r:bigonotation}), this proves the result for $t=1$.

For higher $t$, let $X=\overline{s_1s_2}e^{i\theta_0}\cS_{\gamma}^2e^{-i\theta_0}$ and $Y=\cS_{u^{\gamma}}^2$. Then $\|(X^t-Y^t)A_{\gamma}K\| = \|\sum_{i=1}^t X^{t-i}(X-Y)Y^{i-1}A_{\gamma}K\| \le \sum_{i=1}^t \|(X-Y)Y^{i-1}A_{\gamma}K\|$. Again, $\|Y^{i-1}A_{\gamma}K\|_{\ell^4(\nu_k^{\gamma})}=O_T(1)_{N\To+\infty,\gamma}$ for each $i$ and the claim follows.
\end{proof}
 
 \bigskip
 
In sums like \eqref{e:sum}, we can make packets of size $2t$, and we have for all $m$ and for any $t$
\begin{multline}\label{e:packets}
\left|\sum_{r=0}^{t-1}\la B_\gamma K, \cS_{u^\gamma}^{2r+m} B_\gamma K\ra_{\nu} -\sum_{r=0}^{t-1}(\overline{s_1s_2})^r\la B_\gamma K,  e^{i \theta_0}\cS_{\gamma}^{ 2r} e^{-i \theta_0}\cS_{u^\gamma}^{m}  B_\gamma K\ra_{\nu} \right|
\\ \leq t^2\left( c_{\cs,\beta} \left[\varepsilon^{1/2} + \eta_1 O(1)_{N\To +\infty, \gamma}\right] +C'_{N,\cs }\right)^{1/4}O_T(1)_{N\To+\infty,\gamma}\,. 
\end{multline}
As we will see below, the size $2t$ of packets should be chosen so that $t (c_{\cs,\beta} \varepsilon^{1/2})^{1/4}$ is small as $M$ gets large. Remembering that $c_{\cs,\beta}\leq f(D, \beta)M^3$ and $\varepsilon=M^{-8}$,
we take $t=M^{\alpha}$ with $0<\alpha< 1/4$. We then group the sum \eqref{e:sum} into packets and write
\begin{multline*}
\Big|\sum_{r'\leq r\leq n}\la \cS_{u^\gamma}^{r-r'}B_\gamma K,B_\gamma K\ra_{\nu}\Big|
=\Big|\sum_{r'=1}^n\sum_{r=0}^{n-r'}\la \cS_{u^\gamma}^{r}B_\gamma K,B_\gamma K\ra_{\nu}\Big|\\
\le \Big|\sum_{r'=1}^n \sum_{a=0}^{\lfloor\frac{n-r'}{2t}\rfloor-2}\sum_{r=2ta}^{2t(a+1)-1}\la \cS_{u^\gamma}^{r}B_\gamma K,B_\gamma K\ra_{\nu}\Big| + 4nt\, \norm{B_\gamma K}^2_{\nu} \, ,
 \end{multline*}
where we estimated $|\sum_{r'=1}^n \sum_{r=2t(\lfloor\frac{n-r'}{2t}\rfloor-1)}^{n-r'}\la \cS_{u^\gamma}^{r}B_\gamma K,B_\gamma K\ra_{\nu}| \le 4nt \norm{B_\gamma K}^2_{\nu}$. Note that $\sum_{r=2ta}^{2t(a+1)-1}\la \cS_{u^\gamma}^{r}\cdot,\cdot\ra = \sum_{r=0}^{t-1}\la \cS_{u^\gamma}^{2r+2ta}\cdot,\cdot\ra + \sum_{r=0}^{t-1}\la \cS_{u^\gamma}^{2r+1+2ta}\cdot,\cdot\ra$. So using \eqref{e:packets},
\begin{multline} \label{e:gathering}
\Big|\sum_{r'=0}^n \sum_{a=0}^{\lfloor\frac{n-r'}{2t}\rfloor-2}\sum_{r=2ta}^{2t(a+1)-1}\la \cS_{u^\gamma}^{r}B_\gamma K,B_\gamma K\ra_{\nu}\Big|\\
\le \Big|\sum_{r'=0}^n \sum_{a=0}^{\lfloor\frac{n-r'}{2t}\rfloor-2}\sum_{r=0}^{t-1}(\overline{s_1s_2})^r\left(\la B_\gamma K,  e^{i \theta_0}\cS_{\gamma}^{ 2r} e^{-i \theta_0}(\cS_{u^\gamma}^{2ta}  +\cS_{u^\gamma}^{2ta+1})  B_\gamma K\ra_{\nu}\right) \Big|
\\ + n\cdot \frac{n}t \cdot t^2\left( c_{\cs,\beta} \left[\varepsilon^{1/2} +\eta_1 O(1)_{N\To +\infty, \gamma}\right] +C'_{N,\cs }\right)^{1/4}O_T(1)_{N\To+\infty,\gamma} \, .
\end{multline} 
 
\begin{lem}
Let $\|K\|_{\infty}\le 1$. For $A_\gamma K=\cS_{u^\gamma}^{2ta} B_\gamma K$ or $\cS_{u^\gamma}^{2ta+1}  B_\gamma K$ we have for any $L$
\begin{multline*}
\left|\sum_{r=0}^{t-1}(\overline{s_1s_2})^r\la B_\gamma K,  e^{i \theta_0}\cS_{\gamma}^{ 2r} e^{-i \theta_0}A_\gamma K\ra_{\nu}\right| \le \frac{L^2c_k}{c(D, \beta)} O_T(1)_{N\To \infty, \gamma}+   t\,O_T(L^{-\infty})_{N\To \infty, \gamma} \\
 + \eta_1 O_{M,T}(1)_{N\To +\infty, \gamma} +\frac1{|s_1s_2 -1|}O_T(1)_{N\To \infty, \gamma}  \, .
\end{multline*}
\end{lem}
 
%and 
%\begin{multline*}\Big|\frac{1}{n^2}\sum_{r'\leq r\leq n}\la \cS_{u^\gamma}^{\gamma\,\,{r-r'}}B^\gamma K,B^\gamma K\ra_{\ell^2(\nu^\gamma_k)}-
%\\  \frac{1}{n^2}\sum_{r'\leq r\leq n}(s_1s_2)^{r-r'}\la B_\gamma K,  e^{-i \theta^o}\cS^{\gamma 2k} e^{i \theta^o}B_\gamma K\ra_{\nu^\gamma}\Big|
%\\ \leq 2n\left( c_{\cs,\beta} \left[\varepsilon^{1/2} + \eta \cdot \|Z\|_{\ell^2(\nu_1^{\gamma})}+ \eta^2 \cdot \|Z\|^2_{\ell^2(\nu_1^{\gamma})}\right] +C'_{N,\cs }\right)\norm{B_\gamma K}_{\ell^4(\nu^\gamma)}\norm{B_\gamma K}_{\ell^4(\nu^\gamma)}
%\\ \leq 2n\left( f(D, \beta) M^{3} \left[M^{-4} + \eta \cdot \|Z\|_{\ell^2(\nu_1^{\gamma})}+ \eta^2 \cdot \|Z\|^2_{\ell^2(\nu_1^{\gamma})}\right] +C'_{N,\cs }\right)\norm{B_\gamma K}_{\ell^4(\nu^\gamma)}\norm{B_\gamma K}_{\ell^4(\nu^\gamma)}
%\leq 
%\end{multline*}

\begin{proof}
First assume $k=1$. We decompose $ e^{-i \theta_0}A_\gamma K = C\mathbf{1} +  f$ where $f\perp \mathbf{1} $ in $\ell^2(\nu^\gamma_1)$. So $\cS_{\gamma}^{ 2r} e^{-i \theta_0}A_\gamma K = C\cS_{\gamma}^{ 2r}\mathbf{1} +  \cS_{\gamma}^{ 2r} f$.

For the term $\cS_{\gamma}^{ 2r}f$ we use Proposition \ref{p:iter}, which yields, for any $L$,
\begin{align*}
\| \mathcal{S}_{\gamma}^{2r }f \|_{\nu} & \le \left(1-L^{-2}c(D,\beta)\right)^{r/2}  \|f\|_{\nu} +  \sum_{l=0}^{r-1}C_{N,L,l,2}( e^{-i \theta_0}A_\gamma)^{1/2}  + 2\eta_1  \sum_{l=1}^{r-1} \| \mathcal{Z}_{l} f\|_{\nu}\, .
\end{align*}
By Corollary~\ref{c:Minfty} (recalling that $r\leq t\le M^{\alpha}$), we have $\sum_{l=0}^{r-1}C_{N,L,l,2}( e^{-i \theta_0}A_\gamma)^{1/2}=t\,O_T(L^{-\infty})_{N\To +\infty, \gamma} $. Indeed, the term $e^{-i\theta_0}$ has no impact, as it can be bounded by $1$ in the proof of Proposition~\ref{p:lesscrude}. We also have $\|f\|_{\nu}\le \|A_{\gamma}K\|_{\nu}\le \|B_{\gamma}K\|_{\nu}=O_T(1)_{N\To \infty, \gamma}$, and $\| \mathcal{Z}_{l} f\|_{\nu}=O_{l,T}(1)_{N\To \infty, \gamma}$ by Remark~\ref{r:bigonotation}. Thus,
\begin{multline*}
\left|\sum_{r=0}^{t-1}(\overline{s_1s_2})^r\la B_\gamma K,  e^{i \theta_0}\cS_{\gamma}^{ 2r}f\ra_{\nu}\right| \\
\le \frac{2L^2}{c(D, \beta)} O_T(1)_{N\To \infty, \gamma} +  t\,O_T(L^{-\infty})_{N\To \infty, \gamma}  + \eta_1 O_{M,T}(1)_{N\To \infty, \gamma}  \, .
\end{multline*}
For the term $C\cS_{\gamma}^{ 2r} \mathbf{1}$, we have $\mathcal{S}_{\gamma}^{l} \mathbf{1} = \mathbf{1} - \eta_1 \sum_{s=0}^{l-1} \mathcal{S}_{\gamma}^{\,s}\iota\xi^{\gamma}= \mathbf{1}+ \eta_1 O_l(1)_{N\To \infty, \gamma}$ by \eqref{e:sumzeta2}. Thus,
\begin{multline*}
\Big|\sum_{r=0}^{t-1}(\overline{s_1s_2})^r\la B_\gamma K,  e^{i \theta_0}\cS_{\gamma}^{ 2r}   \mathbf{1}\ra_{\nu}\Big| \le \Big|\sum_{r=0}^{t-1}(\overline{s_1s_2})^r\la B_\gamma K,  e^{i \theta_0}    \mathbf{1}\ra_{\nu}\Big| +\eta_1 O_M(1)_{N\To \infty, \gamma}\|B_{\gamma}K\|_{\nu}\\
= \Big|\frac{(\overline{s_1s_2})^t-1}{\overline{s_1s_2} -1} \la B_\gamma K,  e^{i \theta_0}   \mathbf{1}\ra_{\nu} \Big| +\eta_1O_M(1)_{N\To \infty, \gamma}\|B_{\gamma}K\|_{\nu}
\\\le \bigg(\frac2{|s_1s_2 -1|}  +\eta_1O_M(1)_{N\To \infty, \gamma}\bigg)\|B_{\gamma}K\|_{\nu} \, .
\end{multline*}
Since $|C|\le \|A_{\gamma}K\|_{\nu}\le \|B_{\gamma}K\|_{\nu}$, this completes the proof for $k=1$.

For higher $k$, as in Remark~\ref{rem:Sunorm}, we have $\|\cS_{\gamma}^{2r}f\|_{\nu_k} \le \sqrt{\frac{\mu_1^{\gamma}(B)}{\mu_k^{\gamma}(B_k)}}\|\cS_{\gamma}^{2r-k+1}\phi_f\|_{\nu_1}$, where now $\phi_f(x_0,x_1) = (\cS_{\gamma}^{k-1}f)(x_0;x_k)$. We then note that $f\perp \mathbf{1}$ in $\ell^2(\nu_k^{\gamma})$ iff $\phi_f \perp \mathbf{1}$ in $\ell^2(\nu_1^{\gamma})$. Indeed, $\langle \mathbf{1},\phi_f\rangle_{\nu_1} = \frac{\mu_k^{\gamma}(B_k)}{\mu_1^{\gamma}(B)}\langle \mathbf{1},f\rangle_{\nu_k}$, since $\langle \mathbf{1},\phi_f\rangle_{\nu_1} = \sum_{(x_0,x_1)}\nu_1(x_0,x_1)(\cS_{\gamma}^{k-1}f)(x_0;x_k)$, so applying \eqref{e:Sgamma}, \eqref{e:muk} and \eqref{eq:mv}, the claim follows. Hence, $\|\cS_{\gamma}^{2r-k+1}\phi_f\|_{\nu_1} \lesssim c(1-L^{-2}C)^{r/2}\|\phi_f\|_{\nu_1}$, where $c=\frac{1}{(1-L^{-2})^{(k+3)/4}} \le 2^{k+1}$ for large $L$. The error terms are the same, this time with $\|\cZ_l\phi_f\|_{\nu_1}=O_{l,T}(1)_{N\To \infty, \gamma}$. Finally, $\|\phi_f\|_{\nu_1} \le \sqrt{\frac{\mu_k^{\gamma}(B_k)}{\mu_1^{\gamma}(B)}}\|f\|_{\nu_k}$.
\end{proof}
 
Starting from \eqref{e:gathering} and applying the lemma, we obtain for $\|K\|_{\infty}\le 1$,
\begin{multline}\label{e:gather2}
\frac{1}{n^2}\left| \sum_{ r'\leq n}\sum_{r\geq r'}\la \cS_{u^\gamma}^{r-r'}B_\gamma K,B_\gamma K\ra_{\nu}\right| \le  \frac{1}t \bigg[\frac{2L^2}{c(D, \beta)} O_T(1)_{N\To \infty, \gamma}
 +   t\,O_T(L^{-\infty})_{N\To \infty, \gamma}  \\
  + \eta_1 O_{M,T}(1)_{N\To +\infty, \gamma} +\frac1{|s_1s_2 -1|}O_T(1 )_{N\To +\infty, \gamma}  \bigg] \\
 +  t\left( c_{\cs,\beta} \left[\varepsilon^{1/2} +\eta_1 O(1)_{N\To +\infty, \gamma}\right] + O_T(M^{-\infty})_{N\To \infty, \gamma}\right)^{1/4}O_T(1)_{N\To+\infty,\gamma} \\
 +4n^{-1}t \,\norm{B_\gamma K}^2_{\nu} \, .
\end{multline} 
Remember that $n=M^9$ and $t=M^\alpha$ with $0<\alpha<1/4$. For the term $\frac{1}t \frac{2L^2}{c(D, \beta)}$ to be small, we choose $L=M^{\alpha'}$ with $0<2\alpha'<\alpha$. For instance, take $\alpha=3/16$ and $\alpha'=1/16$. For the other terms, we have $t (c_{\cs,\beta} \varepsilon^{1/2})^{1/4} = O(M^{\alpha-1/4})$ and $n^{-1}t=M^{-9+\alpha}$. The terms $\eta_1 O_{M,T}(1)_{N\To +\infty, \gamma}$ tend to $0$ as $\eta_1=\eta_0+\eta\To 0$, $M$ and $T$ being fixed. Finally, $\|B_{\gamma}K\|_{\nu}^2=O_T(1)_{N\To +\infty, \gamma}$ assuming $\|K\|_{\infty}\le 1$.

We can gather the first and second alternative into one statement~:
\begin{prp}     \label{prp:bothalt}
Let $A>0$.

For all $M$, for all $\gamma$ that fall either into the first alternative or into the second one with $|s^\gamma_1(N)s^\gamma_2(N)-1|\geq A$, we have for $\|K\|_{\infty}\le 1$ and for $n=M^9$,
\begin{multline*}
\Big\|\frac{1}{n}\sum_{r=1}^n\mathcal{R}_{n,r}^{\gamma}\cF_{\gamma}K\Big\|_{\gamma}^2 \le \frac{1}{M^{3/16}} \bigg[\frac{2M^{1/8}}{c(D, \beta)}  O_T(1)_{N\To \infty, \gamma} +   O_T(M^{-\infty})_{N\To \infty, \gamma}   \\
  + \eta_1 O_{M,T}(1)_{N\To +\infty, \gamma} +\frac1{A}O_T(1 )_{N\To +\infty, \gamma} \bigg] \\
 + O_T(M^{-1/16})_{N\To +\infty, \gamma} +\eta^{1/4}_1 O_{M,T}(1)_{N\To +\infty, \gamma}  \,.
\end{multline*}
\end{prp}
\begin{proof}
The arguments in the proof of \eqref{e:nu-2} readily show that $\frac{1}{n^2}\sum_{r,r'=1}^n\mathbf{E}_{n,r,r'}(\eta_1,F_{\gamma}K) = \eta_1\,O_{n,T}(1)_{N\To \infty, \gamma}$. The assertion follows from \eqref{e:sum}, \eqref{e:1stalt} and \eqref{e:gather2}.
\end{proof}

\begin{prp}\label{p:a}
Let $I\subset I_1$ with $\bar{I}\subset I_1$. There exists $a_0$ such that, if $a\leq a_0$, $M$ is large enough, $\eta_1$ is small enough \emph{($M\geq M(a)$, $\eta_1\leq \eta(a)$)}, and $N$ is large enough~:

For any $\gamma$ falling into the second alternative on $G_N$, the sequence $s^\gamma(N)=s^\gamma_1(N)s^\gamma_2(N)$ must satisfy $|s^\gamma(N)-1| > a^{13}$, if $\gamma$ stays in a set of the form
\[
A_{a, \eta_1}=\left\{\gamma : \Re \gamma\in I, \Im\gamma=\eta_1, \IP(|\cW(o)-\gamma|< a)\leq 1-a\right\} \,.
\]
\end{prp}
Before proving the proposition, let us finally give the

\begin{proof}[Proof of Theorem \ref{t:thm4}]
We apply Proposition~\ref{p:mainbound2} and use Proposition~\ref{p:a} to show that we are in the framework of Proposition~\ref{prp:bothalt}.
 
Two cases may happen. Either $\cW(o)$ is deterministic~: there exists $E_0$ such that $\IP(\cW(o)=E_0)=1$. In that case, we fix a small $a>0$, let $J_1 = I \setminus  [E_0-2a, E_0+2a]$ and $J_2 = I\cap [E_0-2a, E_0+2a]$. We then write $\varnbi(\cF_\gamma K)=\varnbia(\cF_\gamma K)+\varnbib(\cF_\gamma K)$. For $\Re \gamma \in J_1$, we have $|\gamma - E_0|>2a$, so $\IP(|\cW(o)-\gamma|< a)=0$ and Proposition \ref{p:a} applies with $a$ arbitrarily small. Proposition~\ref{prp:bothalt}, applied with $A=a^{13}$, thus allows to control $\varnbia(\cF_{\gamma}K)$, while $\varnbib(\cF_{\gamma}K)=O_T(a)$.

If $\cW(o)$ is not deterministic, there exists $a$ such that for all $E\in\IR$,  $\IP(|\cW(o)- E|< a)\leq 1-a$. Thus, for any complex $\gamma$, $\IP(|\cW(o)- \gamma|< a)\leq 1-a$. In this case Proposition \ref{p:a} may be applied with the fixed value $A=a^{13}$ and all $\gamma$.

Either way, we showed that there exists $a_0$ such that, for all $a\leq a_0$, $M \ge M(a)$,  we have for any $s$ and $T$,
\begin{multline}   \label{e:varnbtozero}
\lim_{\eta_0 \downarrow 0}\limsup_{N\to \infty}\varnbi(\cF_\gamma K)^2 \leq |I|^2 \frac{1}{M^{3/16}} \bigg[\frac{2M^{1/8}}{c(D, \beta)}C_T + C_{s,T}M^{-s}  + \frac{C_T}{a^{13}}    \bigg] \\
+ |I|^2C_TM^{-1/16} + a C_T.
 \end{multline}
Taking $M\to \infty$ followed by $a\downarrow 0$, this completes the proof of Theorem \ref{t:thm4}.
\end{proof}

We conclude the section with the
 
\begin{proof}[Proof of Proposition \ref{p:a}]
We will use the following consequences of \textbf{(Green)}~:
\begin{itemize}
\item There exists $0<c_0<\infty$ such that for all $\gamma\in\C^+$, $\Re \gamma\in I_1$,
\begin{equation}\label{e:inegalmu}
 \IE\left(\sum_{y\sim o}\hat{\mu}_1^{\gamma}(o, y)\right)\leq c_0\,, \quad 
\IE\left(\sum_{y\sim o}(\hat{\mu}_1^{\gamma}(o, y))^{-1}\right)\leq c_0\,, \quad \IE\left(\sum_{y\sim o} |\hat{\zeta}^{\gamma}_y(o)|^{-2}\right)\leq c_0\,.
\end{equation}
In fact, $\hat{\mu}_1^{\gamma}(o, y) = \frac{|\Im \hat{\zeta}^{\gamma}_y(o)\Im\hat{\zeta}^{\gamma}_o(y)|}{|\hat{m}_y^{\gamma}\hat{\zeta}_o^{\gamma}(y)|^2}$, so this follows from \textbf{(Green)} and its consequences \eqref{e:nonzerogreen} and \eqref{e:newgreeen}.
\item There exists $0<c_1<\infty$, such that for all $\gamma\in\C^+$, $\Re\gamma\in I_1$,
\begin{subequations}
\begin{align}
\IP\left(|2\Im \hat{m}^{\gamma}_o|\geq 2 r \mbox{ and } |2 \hat{m}^{\gamma}_o|\leq \frac{1}{2}r^{-1}\right)\geq 1-c_1r \,, \label{e:probam} \\
\IP\left(\sum_{y\sim o}|\hat{\zeta}^{\gamma}_o(y)|\le \frac{1}{2} r^{-1}\right)\geq 1 - c_1r\,.\label{e:probaz}
\end{align}
\end{subequations}
In fact, $\expect(|2\Im \hat{m}_o|^{-1}) + \expect(|2\hat{m}_o^{\gamma}|) \le c_1/2$ by \eqref{e:nonzerogreen}, so the first claim follows by Markov's inequality. The second one follows similarly from \eqref{e:newgreeen}.
\end{itemize}

We may now begin the proof. If $\gamma$ falls into the second alternative, then
\begin{align}       \label{e:enfin}
& \|u_{x_0}^{\gamma}(x_1)u_{x_1}^{\gamma}(x_0) - s^\gamma(N)\|_{\nu}^2 \\
& \qquad \le 4 f(\beta, D)M^3 \left[ M^{-4} + \eta_1 O(1)_{N\To +\infty, \gamma}\right] + 4C'_{N,\cs} \,.  \nonumber
\end{align}

Let $a_0 = (2c_0)^{-2}(6+3c_1)^{-12}$; this choice will become clear later on. Take $a\le a_0$. There exist $M(a)$, $\eta(a)$ and $N(a)$ such that if $M\geq M(a)$, $\eta_1\leq \eta(a)$ and $N\ge N(a)$, then the RHS side in \eqref{e:enfin} is $\leq a^{26}$. We fix $\rho \ge a^{26}$.

So take any $a \le a_0$, $M\ge M(a)$, $\eta_1 \le \eta(a)$, and assume towards a contradiction that we can find a subsequence $N_k=N_k(\eta_1)\To +\infty$ and a sequence $\gamma_k\in A_{a, \eta_1}$, each falling into the second alternative on $G_{N_k}$, such that $|s^{\gamma_k}(N_k)-1|^2\leq \rho$. After extracting further subsequences, let $  \lim_{N_k\to +\infty} 
s^{\gamma_k}(N_k)=s$ and $\gamma_0= \lim_{N_k\to +\infty} \gamma_k\in \C$. Then $|s-1|^2\leq \rho$, $\Re \gamma_0 \in I_1, \Im \gamma_0=\eta_1$, and by \eqref{e:enfin} and Remark~\ref{rem:IDS1}
\[
\IE\left(\sum_{y\sim o}|\hat{u}_{o}^{\gamma_0}(y)\hat{u}_{y}^{\gamma_0}(o) - s|^2\hat{\mu}_1^{\gamma_0}(o, y)\right)\leq \rho\, \IE\left(\sum_{y\sim o}\hat{\mu}_1^{\gamma_0}(o, y)\right) \, ,
\]
which implies, using \eqref{e:inegalmu},
\[
\IE\left(\sum_{y\sim o}|\hat{u}_{o}^{\gamma_0}(y)\hat{u}_{y}^{\gamma_0}(o) - 1|^2\hat{\mu}_1^{\gamma_0}(o, y)\right)\leq 4 \rho\, \IE\left(\sum_{y\sim o}\hat{\mu}_1^{\gamma_0}(o, y)\right)\leq 4 c_0  \rho \, .
\]
By the Cauchy-Schwarz inequality, 
\[
\IE\left(\sum_{y\sim o}|\hat{u}_{o}^{\gamma_0}(y)\hat{u}_{y}^{\gamma_0}(o) - 1|^2\hat{\mu}_1^{\gamma_0}(o, y)\right)^{1/2}\geq \frac{\IE\left(\sum_{y\sim o}|\hat{u}_{o}^{\gamma_0}(y)\hat{u}_{y}^{\gamma_0}(o) - 1|\right)}{\IE\left(\sum_{y\sim o}(\hat{\mu}_1^{\gamma_0}(o, y))^{-1}\right)^{1/2}}
\]
and thus, by \eqref{e:inegalmu},
\begin{equation}\label{e:uu1}
\IE\left(\sum_{y\sim o}|\hat{u}_{o}^{\gamma_0}(y)\hat{u}_{y}^{\gamma_0}(o) - 1|\right)\leq \left(4 c_0\rho\,\IE\left(\sum_{y\sim o}(\hat{\mu}_1^{\gamma_0}(o, y))^{-1}\right)\right)^{1/2}
\leq 2c_0  \rho^{1/2} \, .
\end{equation}
Since the value of $\gamma_0$ is now fixed, let us omit it from the notation.

Let us write $\hat{\zeta}^{\gamma_0}_o(y)=\hat{\zeta}_o(y)=r(o, y) e^{-i\theta(o, y)}$ with $r\in \IR_+$ and $\theta\in \IR$. This implies $\hat{u}_{o}(y)=e^{2i\theta(o, y)}$ and $|\hat{u}_{o} (y)\hat{u}_{y} (o) - 1|=
| (e^{i\theta(y, o)}+ e^{-i\theta(o, y)}) (e^{i\theta(y, o)}- e^{-i\theta(o, y)})|$.

Now \eqref{e:uu1} implies that
\begin{equation}
\IE\left(\sum_{y\sim o}\min_{\eps=\pm 1} |e^{i\theta(y, o)}-\eps e^{-i\theta(o, y)}|^2 \right)\leq 2c_0 \rho^{1/2} \, .\label{e:yo}
\end{equation}
Let us call $\eps(o,y)$ the value of $\eps$ achieving the min. By \eqref{eq:mv} we have
\[
2\hat{m}_o=\hat{\zeta}_y(o)^{-1}-\hat{\zeta}_o(y)=r(y, o)^{-1} e^{i\theta(y, o)}-r(o, y) e^{-i\theta(o, y)}
\]
for all $y\sim o$. Thus, using \eqref{e:inegalmu},
\begin{multline}
\IE\left(\sum_{y\sim o} \left|e^{-i\theta(o, y)}\left(\eps(o,y)  r(y, o)^{-1}-r(o, y)\right) -2\hat{m}_o\right|\right) \\
= \IE\left(\sum_{y\sim o}   \left|\left(e^{i\theta(y, o)}-\eps(o,y) e^{-i\theta(o, y)}\right) r(y, o)^{-1}\right|\right) \\
\leq \sqrt{2c_0}  \rho^{1/4}\IE\left(\sum_{y\sim o} r(y, o)^{-2}\right)^{1/2}\leq 2c_0 \rho^{1/4}=:r^6 \,.
 \label{e:a8}
\end{multline}
Denote $t_{o,y}=\eps(o,y)  r(y, o)^{-1}-r(o, y)\in \R$. It follows by Markov's inequality that
\begin{equation}\label{e:cheb}
\sum_{y\sim o} \left|t_{o,y}e^{-i\theta(o, y)}-2\hat{m}_o\right|\leq r^5
\end{equation}
with probability $\ge 1-r$.

The probability that $|2\Im \hat{m}_o|\geq 2 r$ and $|2 \hat{m}_o|\leq \frac{1}{2}r^{-1}$ is at least $1-c_1r$ by \eqref{e:probam}. Thus, \eqref{e:cheb} implies that with probability $\geq 1-r-c_1r$, we have for any $y\sim o$
\begin{equation}   \label{e:rrmoins}
r \leq  |t_{o,y}| \leq r^{-1} \, .
\end{equation}
Combining \eqref{e:cheb} and \eqref{e:rrmoins}, we see that for any $y, y'\sim o$,
\begin{equation}\label{e:yy'}
\left|e^{-i\theta(o, y)} - t_{o,y'}t_{o,y}^{-1}e^{-i\theta(o, y')}\right|\leq r^{4}\,.
\end{equation}
%The previous identities imply that with probability $\ge 1-r-c_1r$,
%\begin{equation}\label{e:yy'}
%|e^{-i\theta(o, y)} - e^{-i\theta(o, y')} |\leq 2  r^{4}.
%\end{equation}
Now \eqref{eq:green3} says that
\[
\gamma_0= \cW(o) + \sum_{y \sim o} \zeta_o (y)+2\hat{m}_o  = \cW(o) + \sum_{y \sim o} r(o, y) e^{-i\theta(o, y)} +2\hat{m}_o\,.
\]
Using \eqref{e:cheb} and \eqref{e:yy'}, we get for any fixed $y'\sim o$,
\begin{multline}\label{e:rr}
\Big|\gamma_0-\mathcal{W}(o)-\bigg(t_{o,y'}+ \sum_{y\sim o} r(o,y)t_{o,y'}t_{o,y}^{-1}\bigg)e^{-i\theta(o,y')}\Big| \\
\le \left|2\hat{m}_o-t_{o,y'}e^{-i\theta(o,y')}\right| + \left|\sum_{y\sim o}r(o,y)\left(e^{-i\theta(o,y)}-t_{o,y'}t_{o,y}^{-1}e^{-i\theta(o,y')}\right)\right|\\
\le r^5 + r^4\sum_{y\sim o}r(o,y) \le 2r^3
\end{multline}
with probability $\ge 1-r-2c_1r$. Here we used that $\sum_{y\sim o}r(o, y)\le \frac{1}{2} r^{-1}$ with probability $\geq 1-c_1r$, see \eqref{e:probaz}. Since $|\gamma_0-\cW(o)|\geq a $ with probability $\geq a$, it follows that
\[
\Big|  t_{o,y'} + \sum_{y \sim o } r(o, y) t_{o,y'}t_{o,y}^{-1}\Big|\geq a-2r^3
\]
with probability $\geq 1-r-2c_1r-(1-a)$. Recall that $r(o,y)$ and $t_{o,u}$ are real. Taking the imaginary part in \eqref{e:rr}, we thus get $|\Im  e^{-i\theta(o, y')}| \le \frac{2r^3+\eta_1}{a-2r^3}$. Assume $\eta_1\le r^3$. Then if $r<a/5$, we get $|\Im  e^{-i\theta(o, y')}| < r^2$. Hence, $\prob(|\Im e^{-i\theta(o,y')}| \ge r^2) \le (2c_1+1)r+1-a$. But we know that $|2\Im \hat{m}_o| \ge 2r$, so taking the imaginary part in \eqref{e:cheb} and using \eqref{e:rrmoins}, we also have that $|\Im e^{-i\theta(o, y')}|\geq r^2$ with probability $\ge 1-r-c_1r$. If $(2+3c_1)r<a$, this will give a contradiction.

To prove the proposition, we take $r=\frac{a}{6+3c_1}$ and choose $a_0 \le (2c_0)^{-2}(6+3c_1)^{-12}$. Recalling that $2c_0\rho^{1/4}=r^6$, we get $\rho^{1/2} = (2c_0)^{-2}(\frac{a}{6+3c_1})^{12}\ge a^{13}$ for $a\le a_0$, as required. We also take $M>M(a)$, and $\eta_1\leq \min (r^3, \eta(a))$.
\end{proof}

\section{Step 5~: Back to the original eigenfunctions}             \label{sec:retour}

In this section, we show that it suffices to consider the non-backtracking quantum variance in order to prove quantum ergodicity; in other words Theorem \ref{t:thm4} can be retrieved from Theorem \ref{thm:1}. This part may be read before or after the others.

Given $K\in\mathscr{H}_k$, we define the \emph{quantum variance} by
\begin{equation}\label{e:usualvar}
\mathop{\mathrm{Var^I}}(K) = \frac{1}{N}\sum_{\lambda_j\in I} \left|\left\langle \psi_j, K_G \psi_j \right\rangle\right| \, ,
\end{equation}
where $K_G$ is as in Section~\ref{sec:basic}.

More generally, fix $\eta_0>0$ and suppose $K^\gamma \in \mathscr{H}_k$ satisfies conditions \textbf{(Hol)}. We denote
\[
\vari(K^{\gamma}) = \frac{1}{N}\sum_{\lambda_j\in I} \left|\left\langle \psi_j, K_G^{\lambda_j+i\eta_0} \psi_j \right\rangle\right|\,,
\]
where the subscript $\eta_0$ indicates that inside the variance, $\Im \gamma$ is fixed and equal to $\eta_0$.
Denote $\gamma_j=\lambda_j+i\eta_0$, and define
\[ g_j(x_0,x_1) = \overline{\zeta_{x_0}^{\gamma_j}}({x}_1)^{-1} \psi_j(x_1) -\psi_j(x_0) \quad \text{and }
g_j^{\ast}(x_0,x_1) = \overline{\zeta_{x_1}^{\gamma_j}}({x}_0)^{-1} \psi_j(x_0) -\psi_j(x_1)\, ,
\]
so $g_j^{\ast}$ and $g_j$ are defined like $f_j^{\ast}$ and $f_j$ (Section \ref{s:nb}), respectively, with $\zeta$ replaced by $\overline{\zeta}$. Put
\[
\widetilde{\varnbi}(K^{\gamma}) = \frac{1}{N} \sum_{\lambda_j \in I} \left|\left\langle g_j^{\ast}, K_B^{\gamma_j} g_j\right\rangle \right| \, .
\]
Next, given $\gamma\in\IC^+$, define the function $N_\gamma:V\To \R_+$ by
\begin{equation}\label{e:Ngamma}
N_{\gamma}(x) = \Im \tilde g^\gamma (\tilde{x},\tilde{x}) \, ,
\end{equation}
where $\tilde x$ is a point in $\tilG$ projecting down to $G=\Gamma\backslash \tilG$. Recall the Laplacian $P$ defined in \eqref{e:lapl}. We next introduce the operators $P_\gamma,\mathcal{S}_{T,\gamma} , \widetilde{\mathcal{S}}_{T,\gamma} : \IC^V\to \IC^V$ defined by
\begin{equation}\label{e:ST}
P_{\gamma} = \frac{d}{N_{\gamma}} P \frac{N_{\gamma}}{d} \, , \qquad \mathcal{S}_{T,\gamma} = \frac{1}{T} \sum_{s=0}^{T-1} (T-s) P_{\gamma}^s \qquad \text{and} \qquad \widetilde{\mathcal{S}}_{T,\gamma} = \frac{1}{T} \sum_{s=1}^T P_{\gamma}^s  \, ,
\end{equation}
for $T\in \IN^{\ast}$, and the operators $\mathcal{L}^{\gamma},\widetilde{\mathcal{L}}^{\gamma} : \IC^V \to \IC^B$ defined by
\begin{align}\label{e:cL}
(\mathcal{L}^{\gamma}J)(x_0,x_1) & = \frac{|\zeta_{x_0}^{\gamma}(x_1)|^2}{|2m_{x_0}^{\gamma} |^2} \left(\frac{ J(x_0)}{N_{\gamma}(x_1)} -\frac{J(x_1)}{\overline{\zeta_{x_0}^{\gamma}(x_1)}\zeta_{x_1}^{\gamma}(x_0)N_{\gamma}(x_0)}\right) \, , \\
(\widetilde{\mathcal{L}}^{\gamma}J)(x_0,x_1) & = \frac{|\zeta_{x_0}^{\gamma}(x_1)|^2}{ |2m_{x_0}^{\gamma} |^2} \left(\frac{J(x_0)}{N_{\gamma}(x_1)}  - \frac{J(x_1)}{\zeta_{x_0}^{\gamma}(x_1)\overline{\zeta_{x_1}^{\gamma}(x_0)}N_{\gamma}(x_0)}\right) \, . \nonumber
\end{align}
Finally, denote $\vari(K-\langle K \rangle_{\gamma} ) :=\vari(K-\langle K \rangle_{\gamma} \mathbf{1})$ where $\mathbf{1}\in \mathscr{H}_0$ is the constant function equal to $1$ (so that, with the notation of Section \ref{sec:basic}, $\widehat{\mathbf{1}}$ is the identity operator).

\begin{prp}               \label{lem:pasinitial}
Fix $\eta_0>0$ and $T\in \N^{\ast}$. For any $J\in \mathscr{H}_0$, we have
\begin{multline*}
\vari\left(J-\langle J\rangle_{\gamma}  \right) \le \varnbi\left(\mathcal{L}^{\gamma}d^{-1}\mathcal{S}_{T,\gamma}\left(J-\langle J\rangle_{\gamma}\right)\right) \\
+ \widetilde{\varnbi}\left(\widetilde{\mathcal{L}}^{\gamma}d^{-1}\mathcal{S}_{T,\gamma}\left(J-\langle J\rangle_{\gamma}\right)\right) + \vari\left(\widetilde{\mathcal{S}}_{T,\gamma}\left(J - \langle J\rangle_{\gamma}\right) \right) \, .
\end{multline*}
\end{prp}
\begin{proof}
We have
\begin{multline}    \label{eq:fjlfj}
\langle f_j^{\ast}, (\mathcal{L}^{\gamma_j}J)_B f_j\rangle = \sum_{(x_0,x_1)\in B} \left(\frac{(\mathcal{L}^{\gamma_j}J)(x_0,x_1)}{\overline{\zeta_{x_1}^{\gamma_j}(x_0)}\zeta_{x_0}^{\gamma_j}(x_1)} + (\mathcal{L}^{\gamma_j}J)(x_1,x_0)\right) \overline{\psi_j(x_0)}\psi_j(x_1) \\
 -  \sum_{(x_0,x_1)\in B} (\mathcal{L}^{\gamma_j}J)(x_0,x_1)\left(\frac{|\psi_j(x_0)|^2}{\overline{\zeta_{x_1}^{\gamma_j}(x_0)}}+ \frac{|\psi_j(x_1)|^2}{\zeta_{x_0}^{\gamma_j}(x_1)} \right) \, .
\end{multline}
We calculate $\langle g_j^{\ast}, (\widetilde{\mathcal{L}}^{\gamma_j}J)_B g_j\rangle$ similarly. We then note that
\[
\frac{(\mathcal{L}^{\gamma_j}J)(x_0,x_1)}{\overline{\zeta_{x_1}^{\gamma_j}(x_0)}\zeta_{x_0}^{\gamma_j}(x_1)} + (\mathcal{L}^{\gamma_j}J)(x_1,x_0) - \frac{(\widetilde{\mathcal{L}}^{\gamma_j}J)(x_0,x_1)}{\zeta_{x_1}^{\gamma_j}(x_0) \overline{\zeta_{x_0}^{\gamma_j}(x_1)}} - (\widetilde{\mathcal{L}}^{\gamma_j}J)(x_1,x_0) =0 \, ,
\]
using that $\frac{|\zeta_{x_1}^{\gamma}(x_0)|^2}{|m_{x_1}^{\gamma}|^2} = \frac{|\zeta_{x_0}^{\gamma}(x_1)|^2}{|m_{x_0}^{\gamma}|^2}$, by (\ref{eq:mv}). Hence,
\begin{multline*}
\langle f_j^{\ast}, (\mathcal{L}^{\gamma_j}J)_B f_j\rangle - \langle g_j^{\ast}, (\widetilde{\mathcal{L}}^{\gamma_j}J)_B g_j\rangle = \sum_{(x_0,x_1)\in B} (\widetilde{\mathcal{L}}^{\gamma_j}J)(x_0,x_1)\left(\frac{|\psi_j(x_0)|^2}{\zeta_{x_1}^{\gamma_j}(x_0)}+ \frac{|\psi_j(x_1)|^2}{\overline{\zeta_{x_0}^{\gamma_j}(x_1)}} \right)  \\
- \sum_{(x_0,x_1)\in B} (\mathcal{L}^{\gamma_j}J)(x_0,x_1)\left(\frac{|\psi_j(x_0)|^2}{\overline{\zeta_{x_1}^{\gamma_j}(x_0)}}+ \frac{|\psi_j(x_1)|^2}{\zeta_{x_0}^{\gamma_j}(x_1)} \right)  \, .
\end{multline*}
Let $\alpha_{x_0}^{x_1} = \frac{|\zeta_{x_0}^{\gamma}(x_1)|^2}{|2m_{x_0}^{\gamma}|^2N_{\gamma}(x_1)}$, and note that $\alpha_{x_1}^{x_0} = \frac{|\zeta_{x_0}^{\gamma}(x_1)|^2}{|2m_{x_0}^{\gamma}|^2N_{\gamma}(x_0)}$ by (\ref{eq:mv}). Then
\[
\frac{(\widetilde{\mathcal{L}}^{\gamma_j}J)(x_0,x_1)}{\zeta_{x_1}^{\gamma_j}(x_0)} - \frac{(\mathcal{L}^{\gamma_j}J)(x_0,x_1)}{\overline{\zeta_{x_1}^{\gamma_j}(x_0)}} = -2i \left[ \frac{\Im \zeta_{x_1}^{\gamma_j}(x_0)}{|\zeta_{x_1}^{\gamma_j}(x_0)|^2}\alpha_{x_0}^{x_1} J(x_0) - \frac{\Im \zeta_{x_0}^{\gamma_j}(x_1)}{|\zeta_{x_1}^{\gamma_j}(x_0)\zeta_{x_0}^{\gamma_j}(x_1)|^2} \alpha_{x_1}^{x_0}J(x_1)\right]
\]
and
\[
\frac{(\widetilde{\mathcal{L}}^{\gamma_j}J)(x_0,x_1)}{\overline{\zeta_{x_0}^{\gamma_j}(x_1)}} - \frac{(\mathcal{L}^{\gamma_j}J)(x_0,x_1)}{\zeta_{x_0}^{\gamma_j}(x_1)} = 2i \left[ \frac{\Im \zeta_{x_0}^{\gamma_j}(x_1)}{|\zeta_{x_0}^{\gamma_j}(x_1)|^2}\alpha_{x_0}^{x_1} J(x_0) - \frac{\Im \zeta_{x_1}^{\gamma_j}(x_0)}{|\zeta_{x_1}^{\gamma_j}(x_0)\zeta_{x_0}^{\gamma_j}(x_1)|^2}\alpha_{x_1}^{x_0}J(x_1) \right] \, .
\]
Hence,
\begin{align*}
& \langle f_j^{\ast}, (\mathcal{L}^{\gamma_j}J)_B f_j\rangle - \langle g_j^{\ast}, (\widetilde{\mathcal{L}}^{\gamma_j}J)_B g_j\rangle \\
& \qquad = -2i \sum_{x_0\in V} |\psi_j(x_0)|^2J(x_0) \sum_{x_1\sim x_0} \left(\frac{\Im \zeta_{x_1}^{\gamma_j}(x_0)}{|\zeta_{x_1}^{\gamma_j}(x_0)|^2}\alpha_{x_0}^{x_1} +  \frac{\Im \zeta_{x_0}^{\gamma_j}(x_1)}{|\zeta_{x_0}^{\gamma_j}(x_1)\zeta_{x_1}^{\gamma_j}(x_0)|^2}\alpha_{x_0}^{x_1}\right) \\
& \qquad \qquad + 2i \sum_{x_0\in V} |\psi_j(x_0)|^2\sum_{x_1\sim x_0} \left(\frac{\Im \zeta_{x_0}^{\gamma_j}(x_1)}{|\zeta_{x_1}^{\gamma_j}(x_0)\zeta_{x_0}^{\gamma_j}(x_1)|^2} \alpha_{x_1}^{x_0} +  \frac{\Im \zeta_{x_1}^{\gamma_j}(x_0)}{|\zeta_{x_1}^{\gamma_j}(x_0)|^2}\alpha_{x_1}^{x_0} \right)J(x_1)\,.
\end{align*}
Now $\Im \zeta_{x_0}^{\gamma}(x_1) + \Im \zeta_{x_1}^{\gamma}(x_0)\cdot |\zeta_{x_0}^{\gamma}(x_1)|^2 = |\zeta_{x_0}^{\gamma}(x_1)|^2\big[\frac{\Im \zeta_{x_0}^{\gamma}(x_1)}{|\zeta_{x_0}^{\gamma}(x_1)|^2} + \Im \zeta_{x_1}^{\gamma}(x_0)\big] = -2\Im m_{x_1}^{\gamma}\cdot |\zeta_{x_0}^{\gamma}(x_1)|^2$ by (\ref{eq:mv}). Since $2\Im m_{x_1}^{\gamma} = N_{\gamma}(x_1)|2m_{x_1}^{\gamma}|^2$, we get $\frac{\Im \zeta_{x_0}^{\gamma}(x_1) + \Im \zeta_{x_1}^{\gamma}(x_0)|\zeta_{x_0}^{\gamma}(x_1)|^2}{|\zeta_{x_0}^{\gamma}(x_1)\zeta_{x_1}^{\gamma}(x_0)|^2} = \frac{-N_{\gamma}(x_1)|2m_{x_1}^{\gamma}|^2}{|\zeta_{x_1}^{\gamma_j}(x_0)|^2}$. Since $\alpha_{x_0}^{x_1} = \frac{|\zeta_{x_1}^{\gamma}(x_0)|^2}{N_{\gamma}(x_1)|2m_{x_1}^{\gamma}|^2}$ and $\alpha_{x_1}^{x_0} = \frac{|\zeta_{x_1}^{\gamma}(x_0)|^2}{N_{\gamma}(x_0)|2m_{x_1}^{\gamma}|^2}$ by (\ref{eq:mv}), we thus have
\begin{align*}
& \langle f_j^{\ast}, (\mathcal{L}^{\gamma_j}J)_B f_j\rangle - \langle g_j^{\ast}, (\widetilde{\mathcal{L}}^{\gamma_j}J)_B g_j\rangle \\
& \qquad = 2i \sum_{x_0\in V} |\psi_j(x_0)|^2d(x_0)J(x_0) - 2i \sum_{x_0\in V} |\psi_j(x_0)|^2\frac{1}{ N_{\gamma}(x_0)} \sum_{x_1\sim x_0} N_{\gamma}(x_1)J(x_1) \\
& \qquad = 2i\, \langle \psi_j, [(I-P_{\gamma_j})dJ]_G\psi_j\rangle\,.
\end{align*}
Hence,
\[
\vari[(I-P_{\gamma})J] \le \varnbi(\mathcal{L}^{\gamma}d^{-1}J) + \widetilde{\mathrm{Var_{nb}^I}}(\widetilde{\mathcal{L}}^{\gamma}d^{-1}J) \, .
\]
Now note that $P_{\gamma}(\mathcal{S}_{T,\gamma}K) = \frac{1}{T} \sum_{s=1}^T(T-s+1)P_{\gamma}^s K = \mathcal{S}_{T,\gamma}K - K + \widetilde{\mathcal{S}}_{T,\gamma}K$. Hence,
\[
K = (I-P_{\gamma})\mathcal{S}_{T,\gamma}K + \widetilde{\mathcal{S}}_{T,\gamma} K
\]
for any $K\in \mathscr{H}_0$. Taking $K_{\gamma}=J-\langle J\rangle_{\gamma}$, we thus get
\begin{multline*}
\vari(K_{\gamma} ) \le  \vari[(I-P_{\gamma})\mathcal{S}_{T,\gamma}K_{\gamma}] + \vari(\widetilde{\mathcal{S}}_{T,\gamma}K_{\gamma} )\\
 \le \varnbi(\mathcal{L}^{\gamma}d^{-1}\mathcal{S}_{T,\gamma}K_{\gamma}) + \widetilde{\mathrm{Var_{nb}^I}}(\widetilde{\mathcal{L}}^{\gamma}d^{-1}\mathcal{S}_{T,\gamma}K_{\gamma}) + \vari(\widetilde{\mathcal{S}}_{T,\gamma}K_{\gamma} ) \, . \qedhere
\end{multline*}
\end{proof}

We now consider $K\in \mathscr{H}_m$ for $m>0$. Define $\mathcal{T}^{\gamma}:\mathscr{H}_1\to\mathscr{H}_1$ and $\mathcal{O}^{\gamma}_1:\mathscr{H}_1\to\mathscr{H}_0$ by
\begin{equation}\label{e:T}
(\cT^{\gamma}K)(x_0,x_1) = \frac{\overline{\zeta_{x_1}^{\gamma}(x_0)}\zeta_{x_0}^{\gamma}(x_1)}{\overline{\zeta_{x_1}^{\gamma}(x_0)}\zeta_{x_0}^{\gamma}(x_1)+1} K(x_0,x_1)
\end{equation} 
\begin{equation}\label{e:O1}
(\mathcal{O}_1^{\gamma}K)(x_0) = \sum_{x_{-1}\sim x_0} \frac{(\mathcal{T}^{\gamma}K)(x_{-1},x_0)}{\zeta_{x_{-1}}^{\gamma}(x_0)} + \sum_{x_1\sim x_0} \frac{(\mathcal{T}^{\gamma}K)(x_0,x_1)}{\overline{\zeta_{x_1}^{\gamma}(x_0)}} \, .
\end{equation}
For $m\ge 2$, define $\cU^\gamma_m :\mathscr{H}_m\to\mathscr{H}_m$, $\mathcal{O}_m^{\gamma}:\mathscr{H}_m\to\mathscr{H}_{m-1}$ and $\mathcal{P}_m^{\gamma}:\mathscr{H}_m \to \mathscr{H}_{m-2}$ by
\begin{equation}\label{e:Um}
(\cU^{\gamma}_mK)(x_0;x_m) = \overline{\zeta_{x_1}^{\gamma}(x_0)}\zeta_{x_{m-1}}^{\gamma}(x_m)K(x_0;x_m) \, ,
\end{equation}
\begin{multline}\label{e:Om}
(\mathcal{O}_m^{\gamma}K)(x_0;x_{m-1})  = \sum_{x_{-1} \in \mathcal{N}_{x_0} \setminus \{x_1\}} \overline{\zeta_{x_0}^{\gamma}(x_{-1})}K(x_{-1};x_{m-1}) \\
 + \sum_{x_m \in \mathcal{N}_{x_{m-1}} \setminus \{x_{m-2}\}} K(x_0;x_m) \zeta_{x_{m-1}}^{\gamma}(x_m) \, .
\end{multline}
\begin{equation}\label{e:Pm}
(\mathcal{P}_m^{\gamma}K)(x_1;x_{m-1}) = \sum_{x_0 \in \mathcal{N}_{x_1}\setminus \{x_2\}, x_m\in \mathcal{N}_{x_{m-1}}\setminus \{x_{m-2}\}} \overline{\zeta_{x_1}^{\gamma}(x_0)}K(x_0;x_m) \zeta_{x_{m-1}}^{\gamma}(x_m) \, .
\end{equation}

\medskip
\begin{prp}                 \label{prp:passuperieur}
Fix $\eta_0>0$. Suppose $\overline{\psi_j(x_0)}\psi_j(x_1)\in \R$ for any $j=1,\dots,N$ and $(x_0,x_1)\in B$. Then for any $K\in \mathscr{H}_1$, we have
\[
\vari(K-\langle K \rangle_{\gamma} ) \le \varnbi(\mathcal{T}^{\gamma}K) + \vari(\mathcal{O}_1^{\gamma}K - \langle \mathcal{O}_1^{\gamma}K\rangle_{\gamma} ) \, ,
\]
%Moreover, $\langle K\rangle_{\lambda +i\eta_0}  = \langle \mathcal{M}_1^{\gamma}K\rangle_{\lambda +i\eta_0} $ for $\gamma=\lambda+i\eta_0$.
and for any $K \in \mathscr{H}_m$, $m\ge 2$, we have
\[
\vari(K-\langle K \rangle_{\gamma} ) \le \varnbi(\cU^{\gamma}_mK) + \vari(\mathcal{O}^{\gamma}_mK  - \langle \mathcal{O}_m^{\gamma}K\rangle_{\gamma}  ) + \vari( \mathcal{P}^{\gamma}_mK - \langle\mathcal{P}_m^{\gamma}K\rangle_{\gamma} ) \, .
\]
%Moreover, $\langle K \rangle_{\lambda +i\eta_0}  = \langle \mathcal{M}_k^{\gamma}K - \mathcal{P}_k^{\gamma}K\rangle_{\lambda +i\eta_0} $ for $\gamma=\lambda+i\eta_0$.
\end{prp}

\medskip

\begin{proof}
Let $K\in \mathscr{H}_1$. Since $\overline{\psi_j(x_0)}\psi_j(x_1)\in \R$ for all $(x_0,x_1)$, we have
\begin{multline*}
\langle f_j^{\ast},(\cT^{\gamma_j}K)_Bf_j\rangle = \sum_{(x_0,x_1)} \overline{\psi_j(x_0)}\psi_j(x_1) \Big(\frac{1}{\overline{\zeta_{x_1}^{\gamma_j}(x_0)}\zeta_{x_0}^{\gamma_j}(x_1)}+1\Big)\cT^{\gamma_j}(x_0,x_1) \\
-\sum_{(x_0,x_1)} (\cT^{\gamma_j}K)(x_0,x_1)\Big(\frac{|\psi_j(x_0)|^2}{\overline{\zeta_{x_1}^{\gamma_j}(x_0)}} + \frac{|\psi_j(x_1)|^2}{\zeta_{x_0}^{\gamma_j}(x_1)}\Big) \, .
\end{multline*}
By definition of $\cT^{\gamma}$ and $\cO_1^{\gamma}$, this implies
\[
\langle f_j^{\ast}, (\mathcal{T}^{\gamma_j}K)_B f_j\rangle = \langle \psi_j, K_G\psi_j \rangle -\langle \psi_j, (\mathcal{O}^{\gamma_j}_1K)_G \psi_j\rangle
\]
and thus
\[
\vari(K - \langle K\rangle_{\gamma} ) \le \varnbi(\mathcal{T}^{\gamma} K) + \vari(\mathcal{O}_1^{\gamma}K - \langle K\rangle_{\gamma} ) \, .
\]
Recall the definition of $\langle K \rangle_{\gamma}$ in \eqref{e:Klambda}. We claim that
\begin{equation}\label{e:checkO1}
\langle \mathcal{O}^{\gamma}_1K\rangle_{\gamma}  = \langle K\rangle_{\gamma}  \, .
\end{equation}
Indeed, we have $\langle K\rangle_{\gamma}  = \sum_{(x_0,x_1)\in B} K(x_0,x_1)\Phi_{\gamma}(x_0,x_1)$. On the other hand,
\[
\langle \mathcal{O}_1^{\gamma}K\rangle_{\gamma}  = \sum_{(x_0,x_1)\in B} \frac{(\mathcal{T}^{\gamma}K)(x_0,x_1)\Phi_{\gamma}(x_1,x_1)}{\zeta_{x_0}^{\gamma}(x_1)} + \sum_{(x_0,x_1)\in B} \frac{(\mathcal{T}^{\gamma}K)(x_0,x_1)\Phi_{\gamma}(x_0,x_0)}{\overline{\zeta_{x_1}^{\gamma}(x_0)}} \,.
\]
But $\frac{\Phi_{\gamma}(x_1,x_1)}{\zeta_{x_0}^{\gamma}(x_1)} + \frac{\Phi_{\gamma}(x_0,x_0)}{\overline{\zeta_{x_1}^{\gamma}(x_0)}} = \frac{1+\overline{\zeta_{x_1}^{\gamma}(x_0)}\zeta_{x_0}^{\gamma}(x_1)}{\zeta_{x_0}^{\gamma}(x_1)\overline{\zeta_{x_1}^{\gamma}(x_0)}} \Phi_{\gamma}(x_0,x_1)$ by \eqref{eq:psiiden2} and the fact that $\Psi_{\gamma,x}(y)=\Psi_{\gamma,y}(x)$ by \eqref{eq:greensym}, so that $\Phi_{\gamma}(x,y)=\Phi_{\gamma}(y,x)$. Hence,
\[
\langle \mathcal{O}_1^{\gamma}K\rangle_{\gamma}  = \sum_{(x_0,x_1)\in B}  \frac{\overline{\zeta_{x_1}^{\gamma}(x_0)}\zeta_{x_0}^{\gamma}(x_1)}{\overline{\zeta_{x_1}^{\gamma}(x_0)}\zeta_{x_0}^{\gamma}(x_1)+1} K(x_0,x_1) \cdot \frac{1+\overline{\zeta_{x_1}^{\gamma}(x_0)}\zeta_{x_0}^{\gamma}(x_1)}{\zeta_{x_0}^{\gamma}(x_1)\overline{\zeta_{x_1}^{\gamma}(x_0)}} \Phi_{\gamma}(x_0,x_1)= \langle K\rangle_{\gamma}  \, .
\]
This proves the proposition for $m=1$. Now let $m \ge 2$. It is easily checked that
\[
\langle f_j^{\ast},(\cU^{\gamma_j}_mK)_Bf_j\rangle = \langle \psi_j, (K-\mathcal{O}^{\gamma_j}_mK+\mathcal{P}^{\gamma_j}_mK)_G\psi_j\rangle \, .
\]
and thus
\begin{equation}     \label{eq:passuperieur}
\vari(K-\langle K \rangle_{\gamma} ) \le \varnbi(\cU^{\gamma}_mK) + \vari(\mathcal{O}^{\gamma}_mK - \mathcal{P}^{\gamma}_mK - \langle K\rangle_{\gamma} ) \, .
\end{equation}
We now note that
\begin{equation}                  \label{eq:gammav}
\langle K\rangle_{\gamma}  = \langle \mathcal{O}_m^{\gamma}K - \mathcal{P}_m^{\gamma}K\rangle_{\gamma}  \, .
\end{equation}
Indeed, we have
\begin{multline*}
\langle \mathcal{O}_m^{\gamma}K - \mathcal{P}_m^{\gamma}K\rangle_{\gamma}  = \sum_{(x_{-1};x_{m-1})\in B_m} \overline{\zeta_{x_0}^{\gamma}(x_{-1})}K(x_{-1};x_{m-1}) \Phi_{\gamma}(x_0,x_{m-1}) \\
  + \sum_{(x_0;x_m) \in B_m} K(x_0;x_m)\zeta_{x_{m-1}}^{\gamma}(x_m)\Phi_{\gamma}(x_0,x_{m-1})  \\
 - \sum_{(x_0;x_m) \in B_m} \overline{\zeta_{x_1}^{\gamma}(x_0)}K(x_0;x_m)\zeta_{x_{m-1}}^{\gamma}(x_m) \Phi_{\gamma}(x_1,x_{m-1}) \, ,
\end{multline*}
so \eqref{eq:gammav} follows from \eqref{eq:psiiden2}. Using \eqref{eq:passuperieur}, this completes the proof.
\end{proof}

We introduce one last operator $\mathcal{X}_{\gamma}:\mathscr{H}_0\to\mathscr{H}_0$ given by
\[
\mathcal{X}_{\gamma}K = \langle K \rangle_{\gamma} \mathbf{1} \,.
\]
The following corollary then holds assuming all eigenfunctions $\psi_j$ are real. Note that this assumption is not needed in the special case $m=0$, corresponding to Theorem~\ref{thm:2}.

\begin{cor}    \label{cor:recurrence}
Suppose we have shown that $\lim_{\eta_0 \downarrow 0} \limsup_{N\to \infty}\varnbi(\cF_{\gamma}K)=0$, $\lim_{\eta_0 \downarrow 0} \limsup_{N\to \infty}\widetilde{\varnbi}(\widetilde{\cF}_{\gamma}K)=0$
for any $\cF_{\gamma}:\mathscr{H}_m \to \mathscr{H}_k$ that is a polynomial combination of $\mathcal{L}^{\gamma}d^{-1}\mathcal{S}_{T,\gamma}$, $\mathcal{X}_{\gamma}$, $\mathcal{U}^{\gamma}_j$, $\mathcal{T}^{\gamma}$, $\mathcal{O}_j^{\gamma}$ and $\mathcal{P}_j^{\gamma}$ ($T$ fixed), $\widetilde{\cF}_{\gamma}$ the same combination with $\mathcal{L}^{\gamma}$ replaced by $\widetilde{\mathcal{L}}^{\gamma}$,
and that 
\begin{equation}\label{e:Tinfty}
\lim_{T\To +\infty}\lim_{\eta_0\downarrow 0} \limsup_{N\to\infty} \vari\big(\widetilde{\mathcal{S}}_{T,\gamma} (C_{\gamma}K - \langle C_{\gamma}K\rangle_{\gamma}) \big) = 0\,,
\end{equation}
where $C_{\gamma}:\mathscr{H}_m \to \mathscr{H}_0$ is any polynomial combination of $\cU_j^{\gamma}$, $\cT^{\gamma}$, $\cO_j^{\gamma}$ and $\cP_j^{\gamma}$. 

Then it will follow that $\lim_{\eta_0 \downarrow 0} \limsup_{N\to \infty} \vari(K-\langle K\rangle_{\gamma} ) =0$ for any $K \in \mathscr{H}_m$. In other words, Theorem \ref{thm:1} will follow.
\end{cor}

\begin{proof}
The case $m=0$ holds by Proposition~\ref{lem:pasinitial} and the triangle inequality $\varnbi(K-\langle K\rangle_{\gamma}) \le \varnbi(K) + \varnbi(\mathcal{X}_{\gamma}K)$. Here, $\cF_{\gamma}$ has the form $\mathcal{L}^{\gamma}d^{-1}S_{T,\gamma}$, $\mathcal{L}^{\gamma}d^{-1}S_{T,\gamma}\mathcal{X}_{\gamma}$ and $C_{\gamma}=I$.

The result for higher $m$ follows by induction using Proposition~\ref{prp:passuperieur}. For example, for $m=2$, the conclusion is obtained by taking $\cF_{\gamma}$ of the form $\mathcal{U}^{\gamma}_2$, $\mathcal{T}^{\gamma}\mathcal{O}_2^{\gamma}$, $\mathcal{L}^{\gamma}d^{-1}\mathcal{S}_{T,\gamma}\mathcal{O}_1^{\gamma}\mathcal{O}_2^{\gamma}$, $\mathcal{L}^{\gamma}d^{-1}\mathcal{S}_{T,\gamma}\mathcal{X}_{\gamma}\mathcal{O}_1^{\gamma}\mathcal{O}_2^{\gamma}$, $\mathcal{L}^{\gamma}d^{-1}\mathcal{S}_{T,\gamma}\cP_2^{\gamma}$, $\mathcal{L}^{\gamma}d^{-1}\mathcal{S}_{T,\gamma}\mathcal{X}_{\gamma}\cP_2^{\gamma}$, and $C_{\gamma}$ of the form $\mathcal{O}_1^{\gamma}\mathcal{O}_2^{\gamma}$ and $\mathcal{P}_2^{\gamma}$.
\end{proof}

\begin{rem}\label{rem:holverif}
All the operators in Corollary~\ref{cor:recurrence} satisfy the assumptions \textbf{(Hol)} from Definition~\ref{d:hol}. Indeed, the first two points of \textbf{(Hol)} are clear (the derivative of any Green function such as $\zeta^z$ or $G^z$ may be assessed for example using the resolvent equation, yielding $|\partial_z \zeta^z| \le (\Im z)^{-2}$).

For the third point, we should estimate $\frac{1}{N}\sum_{\omega\in B_k} |\cF_{\gamma}K(\omega)|^s$. Assume first that $\cX_{\gamma}$ is not contained in $\cF_{\gamma}$. Then assuming $\|K\|_{\infty}\le 1$, we write
\[
|\cF_{\gamma}K(\omega)| = \left|\sum_{\omega'\in B_m} \cF_{\gamma}(\omega,\omega')K(\omega')\right| \le \sum_{\omega'\in B_m} |\cF_{\gamma}(\omega,\omega')| \,.
\]
Now $\cF_{\gamma}=A^{(1)}\cdots A^{(\ell)}$ is a composition of operators $A^{(r)}$, each of which is either a multiplication or of nearest-neighbour type (with $\cS_{T,\gamma}$ a composition of Laplacians). So the sum $\sum_{\omega'} A^{(r)}(\omega,\omega')$ reduces to $\sum_{\omega'\approx \omega} A^{(r)}(\omega,\omega')$, where depending on the operator, $\omega'\approx \omega$ means $\omega'=\omega$, $\omega'\sim \omega$, $\omega'\in \{o_{\omega},t_{\omega}\}$ (origin and terminus of $\omega$), $\omega'\in \{(x,\omega),(\omega,y):x\sim o_{\omega},y \sim t_{\omega}\}$ or $\omega'\in \{(x,\omega,y):x\sim o_{\omega},y\sim t_{\omega}\}$. In any case, $\# \{\omega'\approx \omega\} \le 2D$. So $\cF_{\gamma}(\omega,\omega') = \sum_{\omega_1\approx \omega}\dots\sum_{\omega_{\ell-1}\approx \omega_{\ell-2}} A^{(1)}(\omega,\omega_1)\dots A^{(\ell)}(\omega_{\ell-1},\omega')$ and thus $\sum_{\omega'\in B_m} |\cF_{\gamma}(\omega,\omega')| \le \sum_{\omega_1\approx \omega} \dots \sum_{\omega_\ell\approx \omega_{\ell-1}} |A^{(1)}(\omega,\omega_1)\dots A^{(\ell)}(\omega_{\ell-1},\omega_\ell)|$. It follows that $|\cF_{\gamma}K(\omega)|^s \le (2\ell D)^{s-1} \sum_{\omega_1\approx \omega}\dots\sum_{\omega_\ell\approx \omega_{\ell-1}} |A^{(1)}(\omega,\omega_1)\dots A^{(\ell)}(\omega_{\ell-1},\omega_\ell)|^s$. Using H\"older's inequality, if $\sum_{r=1}^{\ell}\frac{1}{p_r} =1$, we get using Remark~\ref{rem:IDS1} that
\begin{multline}\label{e:genboun}
\frac{1}{N}\sum_{\omega\in B_k} |\cF_{\gamma}K(\omega)|^s \le C_{D,\ell,k,s} \prod_{r=1}^{\ell} \left(\frac{1}{N}\sum_{\omega_{r-1}}\sum_{\omega_r\approx \omega_{r-1}} |A^{(r)}(\omega_{r-1},\omega_r)|^{sp_r}\right)^{1/p_r}\\
\Lim_{N\To +\infty } C_{D,\ell,k,s}\prod_{r=1}^{\ell}\expect\left[\sum_{\omega_{r-1}\,:\,o_{\omega_{r-1}}=o}\sum_{\omega_r\approx \omega_{r-1}} |\hat{A}^{(r)}(\omega_{r-1},\omega_r)|^{sp_r}\right]^{1/p_r}
\end{multline}
uniformly in $\lambda$. Here, $\ell$ may depend on $T$. By definition, all $\hat{A}^{(r)}(\omega,\omega')$ are well-behaved functions of $\hat{\zeta}$ and $\cG^z$, so the previous expression is finite using Remark~\ref{rem:IDS2}. For example, if $\cF^{\gamma}=\cT^{\gamma}$, we are reduced to estimating $\expect\big(\sum_{o'\sim o} |\frac{\overline{\zeta^{\gamma}_{o'}(o)}\hat{\zeta}^{\gamma}_o(o')}{\overline{\hat{\zeta}_{o'}^{\gamma}(o)}\hat{\zeta}_o^{\gamma}(o')+1}|^s\big)$. Using \eqref{eq:mv}, we observe that $\frac{|\hat{\zeta}_o^{\gamma}(o')|}{|\hat{\zeta}_o^{\gamma}(o')+\overline{\hat{\zeta}_{o'}^{\gamma}(o)}^{-1}|} = \frac{|\hat{\zeta}_o^{\gamma}(o')|}{|2\Re \hat{\zeta}_o^{\gamma}(o')+\overline{2\hat{m}_o^{\gamma}}|} \le \frac{|\hat{\zeta}_o^{\gamma}(o')|}{2\Im \hat{m}_o^{\gamma}}$, and we know from Remark~\ref{rem:IDS2} that $\sup_{\gamma}\expect\big(\sum_{o'\sim o} \frac{|\hat{\zeta}_o^{\gamma}(o')|^s}{(2\Im \hat{m}_o^{\gamma})^s}\big)<\infty$. Similarly, if $\cF^{\gamma}=\cL^{\gamma}d^{-1}\cS_{T,\gamma}$, then $|(\cF^{\gamma}K)(e)| \le \frac{|\zeta_{o_e}^{\gamma}(t_e)|^2}{|m_{o_e}^{\gamma}|^2}\frac{1}{N_{\gamma}(o_e)N_{\gamma}(t_e)}\sum_{r=0}^{T-1}[|(P^rd^{-1}N_{\gamma}K)(o_e)| + \frac{|(P^rd^{-1}N_{\gamma}K)(t_e)|}{|\zeta^{\gamma}_{o_e}(t_e)\zeta^{\gamma}_{t_e}(o_e)|}]$, so \eqref{e:genboun} reduces to
\begin{multline*}
C_{D,T,s}\expect\left(\sum_{o'\sim o} \Big(\frac{|\hat{\zeta}_o^{\gamma}(o')|^2}{|\hat{m}_o^{\gamma}|^2\hat{N}_{\gamma}(o)\hat{N}_{\gamma}(o')}\Big)^{p_1 s}\right)^{1/p_1}\expect\left(\hat{N}_{\gamma}(o)^{p_2 s}\right)^{1/p_2} \\ 
+ C_{D,T,s}\expect\left(\sum_{o'\sim o}\Big(\frac{|\hat{\zeta}_o^{\gamma}(o')|}{|\hat{m}_o^{\gamma}|^2\hat{N}_{\gamma}(o)\hat{N}_{\gamma}(o')|\hat{\zeta}_{o'}^{\gamma}(o)|}\Big)^{p_1s}\right)^{1/p_1}\expect\left(\hat{N}_{\gamma}(o)^{p_2s}\right)^{1/p_2}
\end{multline*}
for some $p_1,p_2$.

The previous discussion was under the assumption $A^{(r)}\neq \cX_{\gamma}$. If $\cF_{\gamma}=F_1^{\gamma} \cX_{\gamma} F_2^{\gamma}$  with $F_1^{\gamma}$ and $F_2^{\gamma}$ as in the previous paragraph, we write $\cF_{\gamma} K(\omega) = \sum_{\omega'} F_1^{\gamma}(\omega,\omega') \langle F_2^{\gamma}K\rangle_{\gamma}$, with $|\langle F_2^{\gamma} K\rangle_{\gamma}| = |\frac{\sum_x N_{\gamma}(x) (F_2^{\gamma}K)(x)}{\sum_x N_{\gamma}(x)}| = |\frac{\sum_x\sum_w N_{\gamma}(x)F_2^{\gamma}(x,w)K(w)}{\sum_x N_{\gamma}(x)}| \le \frac{\sum_x \sum_w N_{\gamma}(x)|F_2^{\gamma}(x,w)|}{\sum_x N_{\gamma}(x)}$. Hence, $|\cF_{\gamma}K(\omega)| \le \sum_{\omega'}|F_1^{\gamma}(\omega,\omega')| \cdot \frac{N}{\sum_x N_{\gamma}(x)} \cdot \frac{1}{N}\sum_x \sum_w N_{\gamma}(x)|F_2^{\gamma}(x,w)|$. Applying H\"older's inequality to $\frac{1}{N}\sum_{\omega\in B_k}(\sum_{\omega'}|F_1^{\gamma}(\omega,\omega')|)^s$ and $(\frac{1}{N}\sum_x N_{\gamma}(x)\sum_w|F_2^{\gamma}(x,w)|)^s$ and taking the limit, we obtain a uniform control as before. Thus, all points of \textbf{(Hol)} are satisfied.
\end{rem}

In view of Remark~\ref{rem:holverif}, we may use Theorem \ref{t:thm4} to conclude that for the $\cF_{\gamma}$ in Corollary~\ref{cor:recurrence}, we have $\lim_{\eta_0 \downarrow 0} \limsup_{N\to \infty}\varnbi(\cF_{\gamma}K)=0$.

Since $\widetilde{\varnbi}(\widetilde{\cF}_{\gamma}K)$ is defined exactly like $\varnbi(\cF_{\gamma}K)$ except that $\zeta$ is replaced by $\overline{\zeta}$, it is clear that it can be shown to vanish asymptotically using the same arguments, simply replacing $\zeta$ by $\overline{\zeta}$ when necessary. By Corollary \ref{cor:recurrence}, to finish the proof of Theorem \ref{thm:1}, it suffices to show \eqref{e:Tinfty}. This is what we do now.

Recall that we introduced $\|K\|_{\gamma}$ for $K\in \mathscr{H}_k$, $k \ge 1$, in \eqref{eq:normgamma}. For $K\in \mathscr{H}_0$, we let
\[
\|K\|_{\gamma}^2 = \|N_{\gamma}K\|_{\mathscr{H}_0}^2 = \frac{1}{N}\sum_{x\in V} N_{\gamma}^2(x)|K(x)|^2 \, .
\]
We also define $(Y_{\gamma} K)(x) = \frac{d(x)}{N_{\gamma}(x)} \cdot \frac{\sum_{y\in V} N_{\gamma}(y)K(y)}{\sum_{y\in V}d(y)}$. Denoting $\langle J\rangle_U := \frac{1}{N}\sum_{x\in V} J(y)$ the uniform average of $J$, we have $Y_{\gamma}K = \frac{\langle N_{\gamma}K\rangle_U}{\langle d\rangle_U} \cdot \frac{d}{N_{\gamma}} $. Fix $I=(a,b)\subset I_1$ as in Section~\ref{s:proof1}.

\begin{prp}
Under assumptions \emph{\textbf{(BSCT)}}, \emph{\textbf{(Green)}}, if $K^\gamma\in \mathscr{H}_0$ satisfies the set of assumptions \emph{\textbf{(Hol)}}, then for any interval $I=(a,b)$ as above,
\begin{multline*}
\lim_{\eta_0\downarrow 0}\limsup_{N\to +\infty}\vari(\widetilde{\mathcal{S}}_{T,\gamma}K^{\gamma}-Y_{\gamma}K^{\gamma})^2 \\
\leq \frac{D\,|I|}{\beta^2 T^2}\lim_{\eta_0 \downarrow 0}\lim_{\eta\downarrow 0}\limsup_{N\to \infty} \int_{a-2\eta}^{b+2\eta} \|K^{\lambda+i(\eta^4+\eta_0)}-Y_{\lambda+i(\eta^4+\eta_0)}K^{\lambda+i(\eta^4+\eta_0)}\|_{\lambda+i(\eta^4+\eta_0)}^2\, \dd \lambda \, .
\end{multline*}
\end{prp}
\begin{proof}
We follow the steps in the proof of Theorem~\ref{thm:upvar}. Let $J^{\gamma} = (\widetilde{\mathcal{S}}_{T,\gamma}-Y_{\gamma})K^{\gamma}$ and $\alpha_{\gamma_j}(x) = N_{\gamma_j}^{1/2}(x)$. Then $\vari(J^{\gamma})^2 \le (\frac{1}{N}\sum_{\lambda_j\in I} \|\alpha_{\gamma_j}^{-1}\psi_j\|^2)(\frac{1}{N}\sum_{\lambda_j\in I} \|\alpha_{\gamma_j}J_G^{\gamma_j}\psi_j\|^2)$. As in the proof of \eqref{e:flim'}, $\frac{1}{N}\sum_{\lambda_j\in I} \|\alpha_{\gamma_j}^{-1}\psi_j\|^2 \lesssim \frac{3}{\pi N} \int_{a-2\eta}^{b+2\eta} \sum_{\rho_G(x)\ge d_{R,\eta}} \frac{\Psi_{z+i\eta_0,\tilde{x}}(\tilde{x})}{N_{\lambda+i\eta_0}(x)}\,\dd \lambda \le \frac{3(|I|+4\eta)}{\pi}$ for any small $\eta>0$, since $N_{\gamma}(x)=\Psi_{\gamma,\tilde{x}}(\tilde{x})$.

Hence, $\lim_{\eta_0\downarrow 0}\limsup_{N\to \infty} \vari(J^{\gamma})^2 \le \frac{3|I|}{\pi} \lim_{\eta_0\downarrow 0}\limsup_{N\to \infty} \frac{1}{N}\sum_{\lambda_j\in I} \|\alpha_{\gamma_j}J_G^{\gamma_j}\psi_j\|^2$. Now $\|\alpha_{\gamma_j} J_G^{\gamma_j}\psi_j\|^2 = \sum_{x\in V} N_{\gamma_j}(x)|J^{\gamma_j}(x)|^2|\psi_j(x)|^2$. Arguing as in Section~\ref{s:proof1}, we get
\[
\frac{1}{N}\sum_{\lambda_j\in I} \|\alpha_{\gamma_j}J_G^{\gamma_j}\psi_j\|^2 \lesssim \frac{3}{\pi N} \int_{a-2\eta}^{b+2\eta} \sum_{\rho_G(x)\ge d_{R,\eta}} \chi(\lambda) N_{z+i\eta_0}(x) |J^{z+i\eta_0}(x)|^2 \Psi_{z+i\eta_0,\tilde{x}}(\tilde{x})\,\dd \lambda \, ,
\]
where $z:=\lambda+i\eta^4$. This is bounded by $\frac{3}{\pi} \int_{a-2\eta}^{b+2\eta} \|J^{z+i\eta_0}\|_{z+i\eta_0}^2\,\dd \lambda$, since $\Psi_{\gamma,\tilde{x}}(\tilde{x})=N_{\gamma}(x)$ and $\chi(\lambda)\le 1$ on $\R$.

Summarizing, we have $\lim_{\eta_0\downarrow 0} \limsup_{N\to \infty} \vari(J^{\gamma})^2 \le \frac{9|I|}{\pi^2} \int_{a-2\eta}^{b+2\eta} \|J^{z+i\eta_0}\|_{z+i\eta_0}^2\,\dd \lambda$.

Now recall that $\widetilde{\mathcal{S}}_{T,\gamma} = \frac{1}{T}\sum_{s=1}^T P_{\gamma}^s$, and $P_{\gamma} = \frac{d}{N_{\gamma}}P\frac{N_{\gamma}}{d}$, so that $P_{\gamma}^s = \frac{d}{N_{\gamma}} P^s \frac{N_{\gamma}}{d}$. Moreover, $Y_{\gamma}K = \frac{d}{N_{\gamma}} \frac{\langle N_{\gamma}K \rangle_U}{\langle d\rangle_U}$. So denoting $\gamma=z+i\eta_0$, $\|K\|_{d}^2 = \frac{1}{N}\sum_{x\in V} d(x)|K(x)|^2$, we have
\begin{align*}
 & \|J^{\gamma}\|_{\gamma}^2 = \|N_{\gamma}J^{\gamma}\|_{\mathscr{H}_0}^2 = \frac{1}{N} \sum_{x\in V} \Big|\frac{1}{T} \sum_{s=1}^T d(x)\Big(P^s\frac{N_{\gamma}K^{\gamma}}{d}\Big)(x) - \frac{\langle N_{\gamma}K^{\gamma}\rangle_U}{\langle d\rangle_U} d(x)\Big|^2 \\
& \quad \le  D\cdot \Big\| \frac{1}{T} \sum_{s=1}^T P^s\Big(\frac{N_{\gamma}K^{\gamma}}{d} -\frac{\langle N_{\gamma}K^{\gamma}\rangle_U}{\langle d\rangle_U} \mathbf{1}\Big)\Big\|_{d}^2 \\
&\quad \le \frac{D}{T^2}\bigg(\sum_{s=1}^T(1-\beta)^s \Big\|\frac{N_{\gamma}K^{\gamma}}{d} - \frac{\langle N_{\gamma}K^{\gamma}\rangle_U}{\langle d\rangle_U}\mathbf{1}\Big\|_{d}\bigg)^2 \le \frac{D}{\beta^2T^2} \,\Big\|\frac{N_{\gamma}K^{\gamma}}{d}- \frac{\langle N_{\gamma}K^{\gamma}\rangle_U}{\langle d\rangle_U} \mathbf{1}\Big\|_{d}^2 \, .
\end{align*}
Here we used \textbf{(EXP)} and the fact that $\frac{N_{\gamma}K^{\gamma}}{d} - \frac{\langle N_{\gamma}K^{\gamma}\rangle_U}{\langle d\rangle_U}\mathbf{1}$ is orthogonal to the constants in $\ell^2(V,d)$. Indeed, the orthogonal projector onto $\mathbf{1}$ in $\ell^2(V,d)$ is $P_{\mathbf{1},d}J = \frac{\langle \mathbf{1},J\rangle_{d}}{\langle \mathbf{1},\mathbf{1}\rangle_{d}}\mathbf{1}=\frac{\langle dJ\rangle_U}{\langle d\rangle_U}\mathbf{1}$. Since $\frac{\langle N_{\gamma}K^{\gamma}\rangle_U}{\langle d\rangle_U}\mathbf{1}= \frac{N_{\gamma}Y_{\gamma}K^{\gamma}}{d}$ and $\frac{1}{d} \le 1$, the proposition follows.
\end{proof}

\begin{cor}    \label{cor:varreste}
For any $C_{\gamma}:\mathscr{H}_m\to\mathscr{H}_0$ as in Corollary~\ref{cor:recurrence} and $\bar{I}\subset I_1$, $\|K\|_{\infty}\le 1$,
\[
\lim_{\eta_0\downarrow 0}\limsup_{N\to +\infty}\vari\left(\widetilde{\cS}_{T,\gamma}\left(C_{\gamma}K-\langle C_{\gamma}K\rangle_{\gamma}\right) \right)^2 \le \frac{c\,|I|^2}{\beta^2 T^2}\,.
\]
\end{cor}
\begin{proof}
Let $K_{\gamma}' = C_{\gamma}K - \langle C_{\gamma}K\rangle_{\gamma}  \mathbf{1}$. Then $Y_{\gamma}K_{\gamma}'=0$, since $Y_{\gamma}C_{\gamma}K = \frac{d}{N_{\gamma}}\frac{\langle N_{\gamma}C_{\gamma}K\rangle_U}{\langle d\rangle_U}$ and $\langle C_{\gamma}K\rangle_{\gamma}  Y_{\gamma}\mathbf{1}= \frac{\langle N_{\gamma}C_{\gamma}K\rangle_U}{\langle N_{\gamma}\rangle_U} \frac{d}{N_{\gamma}} \frac{\langle N_{\gamma}\rangle_U}{\langle d\rangle_U}$. Hence, denoting $z=\lambda+i(\eta^4+\eta_0)$,
\begin{multline*}
\lim_{\eta_0\downarrow 0}\limsup_{N\to +\infty}\vari\left(\widetilde{\mathcal{S}}_{T,\gamma}\left(C_{\gamma}K-\langle C_{\gamma}K\rangle_{\gamma}\right) \right)^2 \\
\leq \frac{D\,|I|}{\beta^2 T^2}\lim_{\eta_0 \downarrow 0}\lim_{\eta\downarrow 0}\limsup_{N\to \infty} \int_{a-2\eta}^{b+2\eta} \|C_zK-\langle C_zK\rangle_z\|_{z}^2\, \dd \lambda \, .
\end{multline*}
Now $\|C_zK\|_z^2 = \frac{1}{N}\sum_{x\in V}N_z^2(x)|(C_z K)(x)|^2 \le \frac{1}{N}\sum_{x\in V} N_z^2(x)[\sum_{w\in B_m} |C_{z}(x,w)|]^2$. Similarly, $|\langle C_{z}K \rangle_{\lambda}| \le \frac{1}{\sum_x N_{z}(x)} \sum_x N_{z}(x) \sum_w |C_{z}(x,w)|$. For our operators $C_{z}$, we thus get $\|C_{z}K\|_{z}^2 = O(1)_{N\To +\infty, z}$ and $|\langle C_{z}K\rangle_{z}| = O(1)_{N\To +\infty, z}$, as in Corollary~\ref{c:Minfty}.
\end{proof}

This proves \eqref{e:Tinfty} and ends the proof of Theorem~\ref{thm:1} on the interval $I$.

\smallskip

Suppose further that $\rho(\partial I_1)=0$. As $I_1$ is open, we have $I_1 = \cup_{j\in \IN} J_j$ for open intervals $J_j=(a_j,b_j)$. Let $J_j^{\varsigma}=(a_j+\varsigma,b_j-\varsigma)$ with $\varsigma>0$ small. Then $\overline{J_j^{\varsigma}}\subset I_1$, so using \eqref{e:varnbtozero} and Corollary~\ref{cor:varreste}, we get $\lim_{\eta_0\downarrow 0}\limsup_{N\to\infty} \varisigma(K-\langle K\rangle_{\gamma} )=0$. Now $\varib(K') = \sum_{j=1}^M \varisigma(K') +  \varis(K')$ for any given $M$. By \eqref{e:cool3} and \textbf{(Green)}, we have $\varis(K-\langle K\rangle_{\gamma} ) \leq\frac{\sharp\{\lambda_j\in  I_1\setminus \cup_{k=1}^M J_k^{\varsigma} \}}{N} \, O(1)_{N\To +\infty, \gamma}$. By the convergence of empirical spectral measures (Remark~\ref{rem:IDS1}), and using the fact that $\rho(\partial I_1)=0$, we have $\frac{\sharp\{\lambda_j\in  I_1\setminus \cup_{k=1}^M J_k^{\varsigma} \}}{N}\to \rho(I_1\setminus \cup_{k=1}^M J_k^{\varsigma})$.
 Finally, $\rho(I_1\setminus \cup_{k=1}^M J_k^{\varsigma}) \to 0$ as $\varsigma \downarrow 0$ and $M\To +\infty$. The conclusion of Theorem~\ref{thm:1} thus holds with $I$ replaced by $I_1$.

\smallskip

Finally, if \textbf{(Green)} holds on $\overline{I_1}$, then  $\rho(\{\lambda\}) = \lim_{\eta \downarrow 0} \eta \Im \expect(\cG^{\lambda+i\eta}(o,o)) = 0$ for any $\lambda\in \overline{I_1}$, since $\sup_{\eta>0} \Im \expect(\cG^{\lambda+i\eta}(o,o)) <\infty$. In particular, $\rho(\partial I_1)=0$.

\appendix

\section{Benjamini--Schramm topology\label{s:BSCT}}

\subsection{Generalities}
In this appendix we collect known facts on the Benjamini-Schramm convergence, we refer the reader to \cite{ATV,AL,BS,B,Salez} for details.

A \emph{coloured rooted graph} $(G,o,W)$ is a graph $G=(V,E)$ with a marked vertex $o\in V$ called the \emph{root}, and a map $W:V\to \R$ which we see as a ``colouring''; it can also be regarded as a potential on $\ell^2(V)$. This is a special case of what is called a \emph{network} in \cite{AL}. All graphs are assumed to be \emph{locally finite}, i.e. each vertex has a finite degree.

If $G$ is connected, we denote by $B_G(x,r)$ the \emph{$r$-ball} $\{y\in V:d_G(x,y) \le r\}$, where $d_G$ is the length of the shortest path between $x$ and $y$ in $G$.

As in \cite{AL}, we define a distance between coloured connected graphs by
\begin{equation}     \label{eq:dloc}
d_{loc}\big((G,o,W),(G',o',W')\big) = \frac{1}{1+\alpha} \, ,
\end{equation}
\begin{align*}
\alpha :=\sup \big\{r>0 & :\exists \text{ graph isomorphism } \phi:B_G(o,\lfloor r\rfloor) \to B_{G'}(o',\lfloor r \rfloor) \text{ with }\\
& \quad  \phi(o)=o' \text{ and } |W'(\phi(v))-W(v)|<1/r \ \forall v\in B_G(o,\lfloor r\rfloor)\big\} \, .
\end{align*}

Two coloured rooted graphs $(G,o,W)$ and $(G',o',W')$ are \emph{equivalent} if there is a graph isomorphism $\phi:G \to G'$ such that $\phi(o)=o'$ and $W' \circ \phi = W$. We denote the equivalence class of $(G,o,W)$ by $[G,o,W]$.

Let $\mathscr{G}_{\ast}$ be the set of equivalence classes of connected coloured rooted graphs. Then $d_{loc}$ turns $\mathscr{G}_{\ast}$ into a separable complete metric space. We may thus consider the set of probability measures on $\mathscr{G}_{\ast}$, denoted by $\mathcal{P}(\mathscr{G}_{\ast})$.

Any finite connected coloured graph $(G,W)$, $G=(V,E)$, defines a probability measure $U_{(G,W)}\in \mathcal{P}(\mathscr{G}_{\ast})$ by choosing the root $x$ uniformly at random in $V$:
\begin{equation}\label{e:U}
U_{(G,W)} = \frac{1}{|V|} \sum_{x\in V} \delta_{[G,x,V]} \, .
\end{equation}
If $(G_n,W_n)$ is a sequence of finite coloured graphs, we say that $\prob\in \mathcal{P}(\mathscr{G}_{\ast})$ is the \emph{local weak limit} of  $(G_n,W_n)$ if $U_{(G_n,W_n)}$ converges weakly-$\ast$ to $\prob$ in $\mathcal{P}(\mathscr{G}_{\ast})$. This notion of convergence was introduced in \cite{BS} and generalized in \cite{AL}. In this case, we also say that $(G_n,W_n)$ converges in the sense of Benjamini-Schramm.

The subset $\mathscr{G}_{\ast}^{D, A}\subset \mathscr{G}_{\ast}$ of equivalence classes $[G, o, W]$ such that $G$ is of degree bounded by $D$, and $W$ takes values in $[-A, A]$, is compact. It follows that $\mathcal{P}(\mathscr{G}_{\ast}^{D,A})$ is compact in the weak-$\ast$ topology.
Hence, if $\mathcal{C}^{D,A}_{\text{fin}}$ denotes the set of finite coloured graphs $(G,W)$, $G=(V,E)$, of degree bounded by $D$ and colouring $W:V\to [-A,A]$, then any sequence $(G_n,W_n)\subset \mathcal{C}^{D,A}_{\text{fin}}$ has a subsequence which converges in the sense of Benjamini-Schramm.

Let $C(\mathscr{G}_{\ast}^{D,A})$ be the set of continuous functions $f:\mathscr{G}_{\ast}^{D,A}\to \R$.

Then a sequence $(G_n,W_n)\subset \mathcal{C}^{D,A}_{\text{fin}}$ has a local weak limit $\prob$ iff there is an algebra $\mathscr{A} \subset C(\mathscr{G}_{\ast}^{D,A})$ which separates points, such that for all $f\in \mathscr{A}$,
\begin{equation}        \label{eq:stonewei}
\lim_{n\to \infty} \frac{1}{|V_n|}\sum_{x\in V_n} f\left([G_n,x,W_n]\right) = \int_{\mathscr{G}_{\ast}^{D,A}} f\left([G,o,W]\right)\dd \prob\left([G,o,W]\right) \, .
\end{equation}
This follows from the compactness of $\mathscr{G}_{\ast}^{D,A}$, see \cite[Chapter 13]{Klenke}.

It may not be very clear how a continuous function on $\mathscr{G}_{\ast}^{D,A}$ looks like, so we give a basic example. If $B_F(o,r)$ is an $r$-ball, the sets $\mathscr{C}_F = \{[G,x,W]: B_G(x,r) \cong B_F(o,r)\}$ turn out to be clopen in $\mathscr{G}_{\ast}^{D,A}$, so the characteristic function $\chi_{\mathscr{C}_F}$ is continuous. Here $B_G(x,r) \cong B_F(o,r)$ means there exists a graph isomorphism $\phi:B_G(x,r)\to B_F(o,r)$ with $\phi(x)=o$, Using (\ref{eq:stonewei}), it can be shown that in the special case where there is no colouring, $(G_n)\subset \mathcal{C}^{D,A}_{\text{fin}}$ has a local weak limit $\prob$ iff
\[
\lim_{n\to \infty} \frac{\#\{x:B_{G_n}(x,r) \cong B_F(o,r)\}}{|V_n|} = \prob(\{[G,x]:B_G(x,r) \cong B_F(o,r)\})
\]
for any $B_F(o,r)$. This was in fact the original criterion in \cite{BS}. Using it, one readily checks that a sequence of $(q+1)$-regular graphs $(G_n)$ satisfies \textbf{(BST)} iff it converges to the $(q+1)$-regular tree $\mathbb{T}_q$ in the sense of Benjamini-Schramm, i.e. iff $(G_n)$ has the local weak limit $\delta_{[\mathbb{T}_q,o]}$, with $o\in\mathbb{T}_q$ arbitrary. More generally, by considering the clopen sets $\mathscr{C}_r = \{[G,x,W]:B_G(x,r) \text{ is not a tree}\}$, one sees that if $(G_n,W_n)\subset \mathcal{C}^{D,A}_{\text{fin}}$ has a local weak limit $\prob$ that is concentrated on the subset $\mathscr{T}_{\ast}^{D,A}\subset \mathscr{G}_{\ast}^{D,A}$ of coloured rooted trees, then $(G_n)$ satisfies \textbf{(BST)}. Conversely, if $(G_n)$ satisfies \textbf{(BST)} and if a subsequence of $(G_n,W_n)$ has a local weak limit $\prob$, then $\prob$ must be concentrated on $\mathscr{T}_{\ast}^{D,A}$.

\subsection{Convergence of empirical spectral measures.} We now show that Benjamini-Schramm convergence implies convergence of the empirical spectral measures. This is already known in some settings \cite{ATV,Salez,SS}. In this paper we need the variant stated as Corollary \ref{cor:spectralmeas}.

Given $[G,o,W]\in \mathscr{G}_{\ast}^{D,A}$, $\gamma\in\C^+=\{z, \Im z>0\}$ and $x\sim y \in G$, we define $\zeta_x^{\gamma}(y)$ as in \S \ref{s:ident}. Like in \S \ref{sec:basic}, $B_k$ is the set of non-backtracking paths of length $k$ on $G$.

Fix $s\in \N$. Let $F:(\C\setminus\{0\})^{2s}\to\C$ be a continuous function and $\gamma\in \IC^+$. Let
\begin{equation}\label{e:Fgammagow}
F_{\gamma}([{G}, {o}, {W}]) = \sum_{(x_0;x_s)\in B_s \,:\, x_0=o} F\left(\zeta_{x_0}^{\gamma}(x_1),\zeta_{x_1}^{\gamma}(x_0),\dots,\zeta_{x_{s-1}}^{\gamma}(x_s),\zeta_{x_s}^{\gamma}(x_{s-1})\right) \, .
\end{equation}
For $s=1$, the sum reduces to $\sum_{x_1: x_1\sim o}$.
One can remark that $F_{\gamma}([{G}, {o}, {W}])=F_{\gamma}([\widetilde{G}, \widetilde{o},\widetilde {W}])$ where $\tilG$ is the universal cover of $G$ and $\widetilde{o},\widetilde {W}$ are lifts of ${o}, {W}$.

Next, given Borel $J \subseteq \R$, we define the measure
\[
\mu_{o,F,\gamma}^{(G,W)}(J) = F_{\gamma}([{G}, {o}, {W}]) \langle \delta_o, \chi_J(H_{G,W})\delta_o\rangle \, .
\]

Fix a compact $I\subset \R$ and fix $\eta\in(0,1)$. 

\begin{lem}           \label{lem:rootspec}
Suppose $(\lambda_n,[G_n,o_n,W_n]) \subset I \times \mathscr{G}_{\ast}^{D,A}$ converges to $(\lambda,[G,o,W])$ in $I\times \mathscr{G}_{\ast}^{D,A}$. Then $\mu_{o_n,F,\lambda_n+i\eta}^{(G_n,W_n)}$ converges weakly-$\ast$ to $\mu_{o,F,\lambda+i\eta}^{(G,W)}$.
\end{lem}
\begin{proof}
Since all operators $H_n=H_{(G_n,W_n)}$ and $H=H_{(G,W)}$ are uniformly bounded by $D+A$, the supports of the spectral measures is compact, so it suffices to show that for any $k\in \N$, $\mu_{o_n,F,\lambda_n+i\eta}^{(G_n,W_n)}(t^k) \to \mu_{o,F,\lambda+i\eta}^{(G,W)}(t^k)$; see \cite[Chapter 13]{Klenke}.

Let $k\in \N$.  Denote $\gamma_n=\lambda_n+i\eta$, $\gamma=\lambda+i\eta$. We have
\[
\left| \mu_{o_n,F, \gamma_n}^{(G_n,W_n)}(t^k) - \mu_{o,F,\gamma}^{(G,W)}(t^k)\right| = \left| F_{\gamma_n}([ {G}_n, {o}_n, {W}_n])\langle \delta_{o_n}, H_n^k \delta_{o_n}\rangle - F_{\gamma}([{G}, {o}, {W}])\langle \delta_o, H^k \delta_o \rangle \right| \, .
\]
We first approximate $F$ by a polynomial.

We have $|\zeta_x^{\lambda+i\eta}(y)|\le \eta^{-1}$ and $|\Im \zeta_x^{\lambda+i\eta}(y)| = \eta \,\|(\tilH^{(\tilde{y}|\tilde{x})}-\lambda-i\eta)^{-1} \delta_{\tilde{y}} \|_{\ell^2(\tilG)}^2$. Since $\|\tilH^{({x}|{y})}- \lambda -i\eta\|_{\ell^2\to\ell^2} \le A+D+c_I+1=:c$ for all $\lambda\in I$ and $\eta\in(0,1)$, we get $|\Im \zeta_x^{\lambda+i\eta}(y)| \ge \eta c^{-2}$.

So let $\mathcal{O}\subset \C$ be the compact region $\{\eta c^{-2}\leq |z|\leq  \eta^{-1}\}$. If $F$ is continuous on $\mathcal{O}^{2s} \subset \C^{2s}$, by Stone-Weierstrass, given $R\in \N^{\ast}$, there is a polynomial $P_R$ of $4s$ variables such that $\sup_{(z_1;z_{2s})\in \mathcal{O}^{2s}} |F(z_1,\dots,z_{2s})-P_R(z_1,\bar{z}_1,\dots,z_{2s},\bar{z}_{2s})| \le \frac{1}{2R}$. Hence, for any $\lambda\in I$ and $(x_0;x_s)$, if $\gamma=\lambda+i\eta$, then
\begin{equation}          \label{eq:approx1}
\left|F\left(\zeta_{x_0}^{\gamma}(x_1),\zeta_{x_1}^{\gamma}(x_0),\dots,\zeta_{x_s}^{\gamma}(x_{s-1})\right)-P_R\left(\zeta_{x_1}^{\gamma}(x_0),\overline{\zeta_{x_1}^{\gamma}(x_0)},\dots,\overline{\zeta_{x_s}^{\gamma}(x_{s-1})}\right)\right|\le \frac{1}{2R} \, .
\end{equation}
Let $h_{\eta}(t) = -(t-i\eta)^{-1}$. Given $\epsilon>0$, we may choose a polynomial $Q_{\epsilon}=Q_{\epsilon}^{\eta}$ such that $\|h_{\eta}-Q_{\epsilon}\|_{\infty} < \epsilon$. It follows that $\|h_{\eta}(H^{(\tilde{x}|\tilde{y})}_{\tilde{G}}-\lambda) - Q_{\epsilon}(H^{(\tilde{x}|\tilde{y})}_{\tilde{G}}-\lambda)\| < \epsilon$. In particular, if $Z_{\epsilon}^{\gamma}(x,y) := Q_{\epsilon}(H^{(\tilde{y}|\tilde{x})}_{\tilde{G}}-\lambda)(\tilde{y},\tilde{y})$, we have for any $\lambda\in I$ and $(x,y)\in B$,
\begin{equation}           \label{eq:zetaz}
|\zeta_x^{\gamma}(y) - Z_{\epsilon}^{\gamma}(x,y)|<\epsilon \,.
\end{equation}
As $P_R$ is Lipschitz-continuous on $\mathcal{O}^{2s}$, we may thus find $C_{R,\eta^{-1}}$ such that
\[
\left|P_R\left(\zeta_{x_0}^{\gamma}(x_1),\dots,\overline{\zeta_{x_s}^{\gamma}(x_{s-1})}\right) - P_R\left(Z_{\epsilon}^{\gamma}(x_0,x_1),\dots,\overline{Z_{\epsilon}^{\gamma}(x_s,x_{s-1})}\right)\right| \le  C_{R,\eta^{-1}} \cdot \epsilon = \frac{1}{2R}
\]
by choosing $\epsilon = \frac{1}{2R}\frac{1}{C_{R,\eta^{-1}}}$. Using (\ref{eq:approx1}), we thus get uniformly in $\lambda\in I$, $(x_0;x_s)$,
\begin{equation}           \label{eq:fzetaprox}
\left| F\left(\zeta_{x_0}^{\gamma}(x_1),\zeta_{x_1}^{\gamma}(x_0),\dots,\zeta_{x_s}^{\gamma}(x_{s-1})\right) - P_R\left(Z_R^{\gamma}(x_0,x_1),\dots,\overline{Z_R^{\gamma}(x_s,x_{s-1})}\right)\right| \le \frac{1}{R} \, ,
\end{equation}
where we now denote $Z_R$ because $\epsilon$ is a function of $R$. Define
\[
P_{\gamma}([{G}, {o}, {W}]) = \sum_{(x_1;x_s),x_0=o} P_R\left(Z_R^{\gamma}(x_0,x_1),\dots,\overline{Z_R^{\gamma}(x_s,x_{s-1})}\right) \, .
\]
Then up to an error $\frac{C_{D,s,A,k}}{R}$, it suffices to consider
\[
\left| P_{\gamma_n}([ {G}_n, {o}_n, {W}_n])\langle \delta_{o_n}, H_n^k \delta_{o_n}\rangle - P_{\gamma}([{G}, {o}, {W}])\langle \delta_o, H^k \delta_o \rangle \right| \, .
\]
Let $d_R$ be the degree of $Q_R$ and choose an arbitrary integer $r\ge d_R+s+k=:d_{R,s,k}$. Then we may find $n_r$ such that for $n \ge n_r$, there exists $\varphi_r : B_{G_n}(o_n,r) \xrightarrow{\sim} B_G(o,r)$ with $\|W\circ \varphi_r - W_n\|_{B_{G_n}(o,r)}<1/r$. Now $\langle \delta_{o_n}, H_n^k \delta_{o_n}\rangle = \sum_{u_0,\dots,u_{k-1}} H_n(o_n,u_0)H_n(u_0,u_1)\dots H_n(u_{k-1},o_n)$
and $H_n(v,w) = \mathcal{A}_n(v, w) + W_n(v) \delta_w(v)$. This only depends on $B_{G_n}(o_n,k)$ and its colouring. Similarly, the quantity $Z_R^{\gamma}(x,y)$ corresponding to $(G_n,o_n,W_n)$ only depends on $B_{G_n}(y,d_R)$ and its colouring. Since $r\ge d_{R,s,k}$ and
$\varphi_r : B_{G_n}(o_n,r) \xrightarrow{\sim} B_G(o,r)$, if we let $\cH_n = \mathcal{A}_G + W_n \circ \varphi_r^{-1}$ on $G$, we get $\langle \delta_{o_n},H_n^k \delta_{o_n}\rangle = \langle \delta_o, \mathcal{H}_n^k \delta_o \rangle$. Similarly, $P_{\gamma_n}([ {G}_n, {o}_n, {W}_n])=P_{\gamma_n}( [{G}, {o}, {W}_n\circ \varphi_r^{-1}])$. Let $ {W}_n'= {W}_n\circ\varphi_r^{-1}$. Then for $n \ge n_r$,
\[
\left|\mu_{o_n,F, \gamma_n}^{(G_n,W_n)}(t^k) - \mu_{o,F,\gamma}^{(G,W)}(t^k)\right| \le \frac{C}{R} +  \left| P_{\gamma_n}([ {G}, {o}, {W}_n'])\langle \delta_o, \mathcal{H}_n^k \delta_o\rangle - P_{\gamma}([{G}, {o}, {W}])\langle \delta_o, H^k \delta_o \rangle \right|  \, .
\]
Writing $\cH_n^k-H^k=\sum_{i=1}^k  \cH_n^{k-i}( \cH_n-H)H^{i-1}$, we have
\[
| \langle \delta_o, ( \cH_n^k-H^k) \delta_o \rangle | \le C'_{k,D,A} \| W_n \circ \varphi_r^{-1}-W\|_{B_G(o,r)} \le \frac{C'_{k,D,A}}{r} \, .
\]
A similar argument yields $|P_{\gamma}([ {G}, {o}, {W}_n']) - P_{\gamma}([{G}, {o}, {W}])| \le \frac{C_{R,D,s,A}}{r}$ and $|P_{\gamma_n}([ {G}, {o}, {W}_n']) -P_{\gamma}([ {G}, {o}, {W}_n']) | \le  C_{R,D,s,A,I} |\lambda_n-\lambda| \le \frac{C_{R,D,s,A,I}}{r}$ for $n\ge n_r'$. We thus showed that for any $r\ge d_{R,s,k}$, there exists $n_r''$ such that if $n\ge n_r''$, then $|\mu_{o_n,F, \gamma_n}^{(G_n,W_n)}(t^k) - \mu_{o,F,\gamma}^{(G,W)}(t^k)| \le \frac{C_{D,s,A,k}}{R} + \frac{C'_{k,D,A}+C_{R,D,s,A} + C_{R,D,s,A,I}}{r}$. It follows that $\limsup_{n\to \infty} |\mu_{o_n,F, \gamma_n}^{(G_n,W_n)}(t^k) - \mu_{o,F,\gamma}^{(G,W)}(t^k)| \le \frac{C_{D,s,A,k}}{R}$. Since $R$ is arbitrary, the proof is complete.
\end{proof}

If $(G,W)\in \mathcal{C}_{\textup{fin}}^{D,A}$, we now define, for $\gamma\in \IC^+$,
\[
\mu^{(G,W)}_{F,\gamma} = \frac{1}{|V|} \sum_{x\in V} \mu_{x,F,\gamma}^{(G,W)} \, .
\]

\begin{cor}          \label{cor:spectralmeas}
Suppose $(G_n,W_n)\subset \mathcal{C}_{\textup{fin}}^{D,A}$ has a local weak limit $\prob$. Fix a compact $I\subset \IR$ and $\eta\in(0,1)$. Then $\mu^{(G_n,W_n)}_{F,\lambda+i\eta}$ converges weakly to $\int_{\mathscr{G}_{\ast}^{D,A}} \mu_{o,F,\lambda+i\eta}^{(G,W)}\,\dd \prob([G,o,W])$, uniformly in $\lambda\in I$. In other words, for any continuous $\varphi:\R\to \R$, we have uniformly in $\lambda\in I$,
\begin{multline*}
\frac{1}{|V_n|} \sum_{x\in V_n}F_{\lambda+i\eta}([G_n,x,W_n])\langle \delta_x,\varphi(H_{(G_n,W_n)})\delta_x\rangle \\
\Lim_{N\To +\infty }  \int_{\mathscr{G}_{\ast}^{D,A}} F_{\lambda+i\eta}([{G}, {o}, {W}])\langle \delta_o, \varphi(H_{(G,W)})\delta_o \rangle \,\dd \prob([G,o,W]) \, .
\end{multline*}
\end{cor}
\begin{proof}
Given continuous $\varphi:\R\to \R$, define $\widehat{\varphi} :I \times \mathscr{G}_{\ast}^{D,A} \to \R$ by $\widehat{\varphi}(\lambda,[G,o,W]) = \int \varphi(t)\,\dd \mu_{o,F,\lambda+i\eta}^{(G,W)}(t)$. Lemma~\ref{lem:rootspec} states $\widehat{\varphi}$ is continuous on $I \times \mathscr{G}_{\ast}^{D,A}$ -- hence, uniformly continuous.
 Let $\widehat{\varphi}_\lambda([G,o,W]) = \widehat{\varphi}(\lambda,[G,o,W])$.
Local convergence means that the measures $U_{(G_n,W_n)}$ (defined in \eqref{e:U}) converge weakly to $\prob$. Thus, for any $\lambda\in I$, $\int \widehat{\varphi}_\lambda\,\dd U_{(G_n,W_n)}\to \int \widehat{\varphi}_\lambda\,\dd \rho$, i.e. $\frac{1}{|V_n|}\sum_{x\in V_n} \widehat{\varphi}_\lambda([G_n,x,W_n]) \to \int \widehat{\varphi}_\lambda([G,o,W])\,\dd \prob([G,o,W])$, which is the statement of the lemma for fixed $\lambda\in I$.

Uniformity in $\lambda$ comes from the uniform continuity of $\widehat{\varphi}$, which implies that the maps $\lambda\mapsto \int \widehat{\varphi}_\lambda\,\dd U_{(G_n,W_n)}$ form a uniformly equicontinuous family.
\end{proof}

\begin{rem}               \label{rem:IDS1}
Taking $F \equiv 1$, we get in particular the convergence of empirical spectral measures. On the other hand, when $\varphi \equiv 1$, we get in particular that under assumption \textbf{(BSCT)}, if $I\subset \R$ is compact and $\eta\in (0,1)$ is fixed, then uniformly in $\lambda\in I$,
\begin{multline}              \label{p:limBSCT}
 \frac{1}{N} \sum_{(x_0;x_s)\in B_s} F\left(\zeta_{x_0}^{\lambda +i\eta}(x_1),\zeta_{x_1}^{\lambda +i\eta}(x_0),\dots,\zeta_{x_{s-1}}^{\lambda +i\eta}(x_s),\zeta_{x_s}^{\lambda +i\eta}(x_{s-1})\right)  \\
\Lim_{N\To +\infty } \expect\left[\sum_{(v_0;v_s)\in B_s:v_0=o}F\left(\hat{\zeta}_{v_0}^{\lambda +i\eta}(v_1),\hat{\zeta}_{v_1}^{\lambda +i\eta}(v_0),\dots,\hat{\zeta}_{v_{s-1}}^{\lambda +i\eta}(v_s),\hat{\zeta}_{v_s}^{\lambda +i\eta}(v_{s-1})\right)\right] \, .
\end{multline}
In the paper, we often encounter expressions of the form $\vartheta_{\gamma}(x_0,x_1) = F(\zeta_{x_0}^{\gamma}(x_1),\zeta_{x_1}^{\gamma}(x_0))$ in the LHS of \eqref{p:limBSCT}. In this case, we write $\hat{\vartheta}_{\gamma}(v_0,v_1) := F(\hat{\zeta}_{v_0}^{\gamma}(v_1),\hat{\zeta}_{v_1}^{\gamma}(v_0))$ for the object defined similarly at the limit. For instance, $\hat{\mu}_1^{\gamma}$ is defined like ${\mu}_1^{\gamma}$ but on the limiting tree $(\cT,\cW)$. In the particular case of $m^\gamma$, we have $\hat{m}_o^{\gamma} = \frac{-1}{2\cG^{\gamma}(o,o)}$.

It is worth noting that $\expect[\sum_{o'\sim o} {F}(\hat\zeta_o^{\gamma}(o'))] = \expect[\sum_{o'\sim o}{F}(\hat\zeta_{o'}^{\gamma}(o))]$. This holds because $\frac{1}{N}\sum_{(x_0,x_1)}F(\zeta_{x_0}^{\gamma}(x_1)) = \frac{1}{N}\sum_{(x_0,x_1)} F(\zeta_{x_1}^{\gamma}(x_0))$.
%
%Finally, the previous results continue to hold for variants of the function $F_{\gamma}[G,o,W]$ defined in \eqref{e:Fgammagow}, as long as the functions take arguments in some ball $B_G(o,s)$ around $o$. For example, if $G^{\gamma}(x):=G^{\gamma}(x,x)$ and $P$ is the Laplacian \eqref{e:lapl}, then for any $s\in \N$, we have uniformly in $\lambda\in I$,
%\[
%\frac{1}{N}\sum_{(x_0,x_1)\in B} (P^s G^{\lambda+i\eta_1})(x_0) F\left(\zeta_{x_0}^{\lambda+i\eta}(x_1)\right) \Lim_{N\To +\infty } \expect\left[\sum_{o'\sim o} (P^s \cG^{\lambda+i\eta_1})(o)F\left(\hat{\zeta}_{o}^{\lambda+i\eta_1}(o')\right)\right].
%\]
\end{rem}

\begin{rem}               \label{rem:IDS2}

Using \eqref{eq:green3'}, we have $|\hat{\zeta}_{o'}^{\gamma}(o)|^s \le |\Im \hat{\zeta}_o^{\gamma}(u)|^{-s}$ for any $u\in \mathcal{N}_o\setminus \{o'\}$. In particular, $|\hat{\zeta}_{o'}^{\gamma}(o)|^s \le \sum_{o''\sim o}|\Im \hat{\zeta}_{o}^{\gamma}(o'')|^{-s}$. We thus see by \textbf{(Green)} that for any $s>0$,
\begin{equation}\label{e:nonzerogreen}
\sup_{\lambda\in I_1, \eta\in (0,1)}\IE( |\Im \cG^{\lambda+i\eta}(o,o)|^{-s}) <\infty \, , \qquad \sup_{\lambda\in I_1, \eta\in (0,1)}\IE( |\cG^{\lambda+i\eta}(o,o)|^{s}) <\infty \, ,
\end{equation}
\begin{equation} \label{e:newgreeen}
\sup_{\lambda\in I_1, \eta\in(0,1)}\IE\Big(\sum_{y\sim o} |\hat{\zeta}^{\lambda+i\eta}_y(o)|^{s}\Big) <\infty \, , \qquad \sup_{\lambda\in I_1, \eta\in (0,1)}\IE\Big(\sum_{y\sim o} |\hat{\zeta}^{\lambda+i\eta}_o(y)|^{s}\Big) <\infty \, ,
\end{equation}
\[
\sup_{\lambda\in I_1, \eta\in (0,1)}\IE\Big(\sum_{y\sim o} |\Im \hat{\zeta}^{\lambda+i\eta}_y(o)|^{-s}\Big) <\infty \, .
\]
We also have
\[
\sup_{\lambda\in I_1, \eta\in (0,1)} \expect\left[\sum_{(v_0;v_t)\in B_t:v_0=o}\left|\hat{\zeta}_{v_0}^{\lambda +i\eta}(v_1)\hat{\zeta}_{v_1}^{\lambda +i\eta}(v_0)\cdots\hat{\zeta}_{v_{t-1}}^{\lambda +i\eta}(v_t)\hat{\zeta}_{v_t}^{\lambda +i\eta}(v_{t-1})\right|^s\right] <\infty \, .
\]
To see this, consider for simplicity $\expect[\sum_{(v_0;v_2),v_0=o}|\hat{\zeta}_{v_0}^{\gamma}(v_1)\hat{\zeta}_{v_1}^{\gamma}(v_2)|^s]$. This is the limit of $\frac{1}{N}\sum_{(x_0;x_2)\in B_2} |\zeta_{x_0}^{\gamma}(x_1) \zeta_{x_1}^{\gamma}(x_2)|^s$. This sum is bounded by $(\frac{1}{N}\sum_{(x_0;x_2)\in B_2} |\zeta_{x_0}^{\gamma}(x_1)|^{2s})^{1/2}\cdot (\frac{1}{N}\sum_{(x_0;x_2)\in B_2}|\zeta_{x_1}^{\gamma}(x_2)|^{2s})^{1/2}$ for any $N$. Using $|\mathcal{N}_{x_1}|-1\le D$ and taking $N\to \infty$, we see the limit is bounded by $D\expect(\sum_{o'\sim o}|\hat{\zeta}_{o}^{\gamma}(o')|^{2s})^{1/2}\expect(\sum_{o'\sim o} |\hat{\zeta}_{o}^{\gamma}(o')|^{2s})^{1/2} \le DC_s$ by \eqref{e:newgreeen}, for any $\lambda\in I_1$ and $\eta>0$. Hence, $\sup_{\lambda\in I_1,\eta>0}\expect[\sum_{(v_0;v_2),v_0=o}|\hat{\zeta}_{v_0}^{\gamma}(v_1)\hat{\zeta}_{v_1}^{\gamma}(v_2)|^s] \le DC_s$.

\end{rem}
\begin{rem}               \label{rem:IDS3}

Let us now look at the quantity $\frac{1}{N}\sum_{(x_0,x_1)}\sum_{(x_2;x_k) ,(y_2;y_k) }|\tilg^{\gamma}(\tilde{x}_k,\tilde{y}_k)|^s$, which we had to control in Section~\ref{s:proof1}.

Let $x_k\wedge y_k$ be the vertex of maximal length in $(x_0;x_k) \cap (x_0;y_k)$, so $x_k\wedge y_k = x_t$ for some $1\le t \le k$. Then $\tilg^{\gamma}(\tilde{x}_k,\tilde{y}_k)=\frac{-\prod_{l=0}^{k-t-1} \zeta_{x_{k-l}}^{\gamma}(x_{k-l-1})\cdot\zeta_{x_t}^{\gamma}(y_{t+1})\prod_{l=t+1}^{k-1} \zeta_{y_l}^{\gamma}(y_{l+1})}{2m_{x_k}^{\gamma}}$. We then write $\frac{1}{N}\sum_{(x_0,x_1)}\sum_{(x_2;x_k) ,(y_2;y_k) }  = \frac{1}{N}\sum_{(x_0,x_1)} \sum_{t=1}^k\sum_{(x_2;x_k) ,(y_2;y_k) ,x_k\wedge y_k=x_t}$, use H\"older's inequality, and take $N\to\infty$ to get a uniform bound involving $\expect[\sum_{o'\sim o} |\hat{\zeta}_o^{\gamma}(o')|^{s_2}]$ and $\expect[ |2\hat{m}_o|^{-s_1}]$, both of which are finite. Hence, $\frac{1}{N}\sum_{(x_0,x_1)}\sum_{(x_2;x_k) ,(y_2;y_k) }|\tilg^{\gamma}(\tilde{x}_k,\tilde{y}_k)|^s$ is uniformly bounded as $N\to\infty$.
\end{rem}

\subsection{Proofs of auxiliary results}        \label{sec:bsctaux}

We now turn to the proofs of some claims in Section~\ref{sec:introd}. In what follows, $\eta_0\in (0,1)$ is fixed.

{\bf{Claim \eqref{e:conv2}}.} Let $\chi : \mathscr{G}_{\ast}^{D,A}\to \IR$ and $F:\C \to \R$ be continuous. Then under \textbf{(BSCT)},
\begin{equation}        \label{eq:joint}
\frac1N \sum_{x\in V_N} \chi([G_N, x]) \sum_{y, d(y, x)=k} F(\tilde g^{\lambda+i\eta_0}_N(\tilde{x}, \tilde{y})) \Lim_{N\To +\infty} \IE\Big( \chi((\cT, o)) \sum_{v, d(v, o)=k} F(\cG^{\lambda +i\eta_0}(o, v))\Big)
\end{equation}
uniformly in $\lambda\in I_0$. This is a variant of Corollary~\ref{cor:spectralmeas} when one considers $F_{\gamma, \chi}:(\lambda,[G,x,W]) \mapsto \chi([G,x]) \sum_{y , d(y,x)=k} F(\tilde{g}^{\gamma}(x,y))$ instead of $F_\gamma$. In particular, taking $k=0$ and $\chi=1$, we obtain (\ref{e:conv2}).

\medskip

{\bf{Claim \eqref{eq:joint2}}.} We may assume $F$ is compactly supported (cf. Lemma~\ref{lem:rootspec}), hence uniformly continuous. Let $h_N(t)=\frac{1}{N}\sum_{x\in V_N}\chi([G_N, x]) \sum_{y, d(y, x)=k} F(t \Im \tilde g^{\lambda +i\eta_0}_N(x, y))$, $h(t) =  \IE\big( \chi((\cT, o)) \sum_{v, d(v, o)=k} F(t \Im \cG^{\lambda +i\eta_0}(o, v))\big)$, let $c_N(\lambda) = \frac{N}{\sum_{\tilde{x}\in \mathcal{D}_N} \Im \tilde g^{\lambda +i\eta_0}_N(\tilde{x},\tilde{x})}$ and $c(\lambda) = \frac{1}{\expect(\Im \mathcal{G}^{\lambda +i\eta_0}(o,o))}$.
The family $h_N$ is uniformly equicontinuous, and as in (\ref{eq:joint}) it converges uniformly to $h$. 
By (\ref{e:conv2}), $c_N(\lambda)\to c(\lambda)$ uniformly in $\lambda$. So $|h_N(c_N(\lambda)-h(c(\lambda))|\to 0$ uniformly in $\lambda$. This proves (\ref{eq:joint2}).

\medskip

We now turn to the proof of {\bf{Claim \eqref{eq:weightedav}}.} Consider the set of (double)-coloured rooted graphs $(G,o,W,a)$, where now $W:V\To \IR$ and $a:V\to \{0,1\}$.
We say $(G,o,W,a)$ and $(G',o',W',a')$ are equivalent if there is $\phi:G\to G'$ with $\phi(o)=o'$, $W'\circ \phi=W$ and $a'\circ \phi=a$. We let $\widehat{\mathscr{G}}_{\ast}^{D,A}$ be the corresponding set of equivalence classes and endow it with a metric $d_{loc}$ defined similarly to (\ref{eq:dloc}).
This amounts to the same definition as before, except that the colourings now take values in $\IR\times \{0, 1\}$ instead of $\IR$. The notion of local weak limit may obviously be extended to this situation.

Assuming that \textbf{(BSCT)} holds, then up to passing to a subsequence, $(G_N,W_N,\bbbone_{\Lambda_N})$ will have a local weak limit $\hat{\prob}$ concentrated on $\{[\cT, o, \cW, a]\}$, whose marginals on $\mathscr{T}_{\ast}^{D,A}$ coincides with $\prob$. The fact that $|\Lambda_N| \geq \alpha N$ implies $\hat\IP(a(o)=1) \geq \alpha$, since $\{a(o)=1\}$ is clopen in $\widehat{\mathscr{G}}_{\ast}^{D,A}$. We claim that
\begin{equation}      \label{e:cool1}
\lim_{N\To +\infty }\la \bbbone_{\Lambda_N} \ra_{\lambda+i\eta_0} = \frac{ \hat\IE\left(a(o) \Im \cG^{\lambda+i\eta_0}(o, o) \right)}{\IE\left({\Im \cG^{\lambda+i\eta_0}}(o, o)\right)}
\end{equation}
uniformly in $\lambda\in I_0$. Indeed, as in Lemma~\ref{lem:rootspec}, if $F : I_0 \times \widehat{\mathscr{G}}_{\ast}^{D,A}\to \C$ is given by $F(\lambda,[G,x,W,a]) = a(x) \Im \tilde{g}^{\lambda +i\eta_0}(x,x)$, then $F$ is continuous. So $\int F_{\lambda}\,\dd U_{G_N,W_N,\bbbone_{\Lambda_N}} \to \int F_{\lambda} \,\dd \hat{\prob}$ uniformly in $\lambda$ as in Corollary~\ref{cor:spectralmeas}. Combined with (\ref{e:conv2}), this yields (\ref{e:cool1}). We next note that for any $\alpha>0$,
\begin{equation}       \label{e:cool2}
\inf_{\lambda\in I_1, \eta_0\in (0,1)}\inf_{a,\hat\IP(a(o)=1)\geq \alpha } \frac{ \hat\IE\left(a(o)
\Im \cG^{\lambda+i\eta_0}(o, o) \right)}{\IE\left({\Im \cG^{\lambda+i\eta_0}}(o, o)\right)} >0 \, .
\end{equation}
In fact, suppose on the opposite that for all $\eps>0$, we can find $\lambda\in I_1, \eta_0\in (0,1)$ and $a$ such that $\hat\IP(a(o)=1)\geq \alpha$ and $\hat \IE\left(a(o)
\Im \cG^{\lambda+i\eta_0}(o, o) \right) \leq \eps$.
 The latter implies
\[
\hat\IP\left(a(o)=1, \Im \cG^{\lambda+i\eta_0}(o, o)\geq \eps^{1/2}  \right)\leq \eps^{1/2} \, .
\]
 On the other hand, since $a$ takes only the values 0 and 1,
\[
\hat\IP\left(a(o)=1, \Im \cG^{\lambda+i\eta_0}(o, o)\geq \eps^{1/2}  \right)\geq \hat\IP( \Im\cG^{\lambda+i\eta_0}(o, o)\geq \eps^{1/2}) -  \hat\IP(a(o)=0) \, .
\]
Thus,
\[
\hat\IP(\Im \cG^{\lambda+i\eta_0}(o, o)\geq \eps^{1/2}) -  \hat\IP(a(o)=0)\leq \eps^{1/2} \, .
\]
Equation \eqref{e:nonzerogreen} with $s=2$ implies that $\hat\IP( \Im \cG^{\lambda+i\eta_0}(o, o)<\eps^{1/2})\leq C \eps$, for some constant $C<\infty$ independent of $\lambda, \eta_0$. So $\hat\IP( \Im\cG^{\lambda+i\eta_0}(o, o)\geq \eps^{1/2})\geq 1-C\eps$. By assumption, $\hat\IP(a(o)=0)\leq 1-\alpha$.
Taking $\eps\to 0$ we would obtain $\alpha \leq 0$, a contradiction. We thus proved (\ref{e:cool2}). Since (\ref{e:cool1}) holds uniformly in $\lambda$, we get (\ref{eq:weightedav}).

Finally, as in the proof of \eqref{e:cool1}, we may consider the set of double-coloured rooted graphs $(G,o,W,K)$, where $K$ is a colouring of pairs of vertices $x,y\in G$, $d_G(x,y)\le R$, with values in $\{|z|\le 1\}\subset \C$. Assuming \textbf{(BSCT)} holds, up to passing to a subsequence, $(G_N,W_N,K_N)$ will have a local weak limit $\hat{\prob}$ concentrated on $\{[\cT,o,\cW,\cK]\}$ whose marginals on $\mathscr{T}_{\ast}^{D,A}$ coincides with $\prob$. We then deduce as before that uniformly in $\lambda\in I_0$,
\begin{equation}      \label{e:cool3}
\lim_{N\To +\infty }\left\la \mathbf{K}_N \right\ra_{\lambda+i\eta_0} = \frac{ \hat\IE\left(\sum_{y:d(y,o)\le R} \cK(o,y) \Im \cG^{\lambda+i\eta_0}(o, y) \right)}{\IE\left({\Im \cG^{\lambda+i\eta_0}}(o, o)\right)} \,.
\end{equation}

\medskip

{\bf{Acknowledgements~:}} This material is based upon work supported by the Agence Nationale de la Recherche under grant No.ANR-13-BS01-0007-01, by the Labex IRMIA and the Institute of Advance Study of Universit\'e de Strasbourg, and by Institut Universitaire de France.

\bibliographystyle{plain}
%\bibliography{biblio-lemasson} 

\providecommand{\bysame}{\leavevmode\hbox to3em{\hrulefill}\thinspace}
\providecommand{\MR}{\relax\ifhmode\unskip\space\fi MR }
\providecommand{\MRhref}[2]{%
  \href{http://www.ams.org/mathscinet-getitem?mr=#1}{#2}
}
\providecommand{\href}[2]{#2}

\end{document}

%% file: def.tex
\newcommand{\nwc}{\newcommand}
\nwc{\nwt}{\newtheorem}
\nwt{coro}{Corollary}
\nwt{ex}{Example}
\nwt{prop}{Proposition}
\nwt{defin}{Definition}

\nwc{\mf}{\mathbf} %Latex (as in \bf not tilted math letters)
\nwc{\blds}{\boldsymbol} %Latex 
\nwc{\ml}{\mathcal} %Latex

\nwc{\lam}{\lambda}
\nwc{\del}{\delta}
\nwc{\Del}{\Delta}
\nwc{\Lam}{\Lambda}
\nwc{\elll}{\ell}
\nwc{\IA}{\mathbb{A}} %algebraic
\nwc{\IB}{\mathbb{B}} %ball
\nwc{\IC}{\mathbb{C}} %complex
\nwc{\ID}{\mathbb{D}} %Dedekind
\nwc{\IE}{\mathbb{E}} %Euklides
\nwc{\IF}{\mathbb{F}} %finite field
\nwc{\IG}{\mathbb{G}} %Gauss
\nwc{\IH}{\mathbb{H}} %Hilbert\N-subgroup
\nwc{\IN}{\mathbb{N}} %natural
\nwc{\IP}{\mathbb{P}} %prime
\nwc{\IQ}{\mathbb{Q}} %rational
\nwc{\IR}{\mathbb{R}} %real
\nwc{\IS}{\mathbb{S}} %sphere
\nwc{\IT}{\mathbb{T}} %torus
\nwc{\IZ}{\mathbb{Z}} %integers
\def\bbbone{{\mathchoice {1\mskip-4mu {\rm{l}}} {1\mskip-4mu {\rm{l}}}
{ 1\mskip-4.5mu {\rm{l}}} { 1\mskip-5mu {\rm{l}}}}}
\def\bbleft{{\mathchoice {[\mskip-3mu {[}} {[\mskip-3mu {[}}{[\mskip-4mu {[}}{[\mskip-5mu {[}}}}
\def\bbright{{\mathchoice {]\mskip-3mu {]}} {]\mskip-3mu {]}}{]\mskip-4mu {]}}{]\mskip-5mu {]}}}}
\nwc{\setK}{\bbleft 1,K \bbright}
\nwc{\setN}{\bbleft 1,\cN \bbright}
 \newcommand{\Lim}{\mathop{\longrightarrow}\limits}
\nwc{\va}{{\bf a}}
\nwc{\vb}{{\bf b}}
\nwc{\vc}{{\bf c}}
\nwc{\vd}{{\bf d}}
\nwc{\ve}{{\bf e}}
\nwc{\vf}{{\bf f}}
\nwc{\vg}{{\bf g}}
\nwc{\vh}{{\bf h}}
\nwc{\vi}{{\bf i}}
\nwc{\vI}{{\bf I}}
\nwc{\vj}{{\bf j}}
\nwc{\vk}{{\bf k}}
\nwc{\vl}{{\bf l}}
\nwc{\vm}{{\bf m}}
\nwc{\vM}{{\bf M}}
\nwc{\vn}{{\bf n}}
\nwc{\vo}{{\it o}}
\nwc{\vp}{{\bf p}}
\nwc{\vq}{{\bf q}}
\nwc{\vr}{{\bf r}}
\nwc{\vs}{{\bf s}}
\nwc{\vt}{{\bf t}}
\nwc{\vu}{{\bf u}}
\nwc{\vv}{{\bf v}}
\nwc{\vw}{{\bf w}}
\nwc{\vx}{{\bf x}}
\nwc{\vy}{{\bf y}}
\nwc{\vz}{{\bf z}}
\nwc{\bal}{\blds{\alpha}}
\nwc{\bep}{\blds{\epsilon}}
\nwc{\barbep}{\overline{\blds{\epsilon}}}
\nwc{\bnu}{\blds{\nu}}
\nwc{\bmu}{\blds{\mu}}
\nwc{\bet}{\blds{\eta}}

\nwc{\bk}{\blds{k}}
\nwc{\bm}{\blds{m}}
\nwc{\bM}{\blds{M}}
\nwc{\bp}{\blds{p}}
\nwc{\bq}{\blds{q}}
\nwc{\bn}{\blds{n}}
\nwc{\bv}{\blds{v}}
\nwc{\bw}{\blds{w}}
\nwc{\bx}{\blds{x}}
\nwc{\bxi}{\blds{\xi}}
\nwc{\by}{\blds{y}}
\nwc{\bz}{\blds{z}}

\nwc{\cA}{\ml{A}}
\nwc{\cB}{\ml{B}}
\nwc{\cC}{\ml{C}}
\nwc{\cD}{\ml{D}}
\nwc{\cE}{\ml{E}}
\nwc{\cF}{\ml{F}}
\nwc{\cG}{\ml{G}}
\nwc{\cH}{\ml{H}}
\nwc{\cI}{\ml{I}}
\nwc{\cJ}{\ml{J}}
\nwc{\cK}{\ml{K}}
\nwc{\cL}{\ml{L}}
\nwc{\cM}{\ml{M}}
\nwc{\cN}{\ml{N}}
\nwc{\cO}{\ml{O}}
\nwc{\cP}{\ml{P}}
\nwc{\cQ}{\ml{Q}}
\nwc{\cR}{\ml{R}}
\nwc{\cS}{\ml{S}}
\nwc{\cT}{\ml{T}}
\nwc{\cU}{\ml{U}}
\nwc{\cV}{\ml{V}}
\nwc{\cW}{\ml{W}}
\nwc{\cX}{\ml{X}}
\nwc{\cY}{\ml{Y}}
\nwc{\cZ}{\ml{Z}}

\nwc{\fA}{\mathfrak{a}}
\nwc{\fB}{\mathfrak{b}}
\nwc{\fC}{\mathfrak{c}}
\nwc{\fD}{\mathfrak{d}}
\nwc{\fE}{\mathfrak{e}}
\nwc{\fF}{\mathfrak{f}}
\nwc{\fG}{\mathfrak{g}}
\nwc{\fH}{\mathfrak{h}}
\nwc{\fI}{\mathfrak{i}}
\nwc{\fJ}{\mathfrak{j}}
\nwc{\fK}{\mathfrak{k}}
\nwc{\fL}{\mathfrak{l}}
\nwc{\fM}{\mathfrak{m}}
\nwc{\fN}{\mathfrak{n}}
\nwc{\fO}{\mathfrak{o}}
\nwc{\fP}{\mathfrak{p}}
\nwc{\fQ}{\mathfrak{q}}
\nwc{\fR}{\mathfrak{r}}
\nwc{\fS}{\mathfrak{s}}
\nwc{\fT}{\mathfrak{t}}
\nwc{\fU}{\mathfrak{u}}
\nwc{\fV}{\mathfrak{v}}
\nwc{\fW}{\mathfrak{w}}
\nwc{\fX}{\mathfrak{x}}
\nwc{\fY}{\mathfrak{y}}
\nwc{\fZ}{\mathfrak{z}}

\nwc{\tA}{\widetilde{A}}
\nwc{\tB}{\widetilde{B}}
\nwc{\tE}{E^{\vareps}}
\nwc{\tk}{\tilde k}
\nwc{\tN}{\tilde N}
\nwc{\tP}{\widetilde{P}}
\nwc{\tQ}{\widetilde{Q}}
\nwc{\tR}{\widetilde{R}}
\nwc{\tV}{\widetilde{V}}
\nwc{\tW}{\widetilde{W}}
\nwc{\ty}{\tilde y}
\nwc{\teta}{\tilde \eta}
\nwc{\tdelta}{\tilde \delta}
\nwc{\tlambda}{\tilde \lambda}
\nwc{\ttheta}{\tilde \theta}
\nwc{\tvartheta}{\tilde \vartheta}
\nwc{\tPhi}{\widetilde \Phi}
\nwc{\tpsi}{\tilde \psi}
\nwc{\tmu}{\tilde \mu}

\nwc{\To}{\longrightarrow} %limits

\nwc{\ad}{\rm ad}
\nwc{\eps}{\epsilon}
\nwc{\ep}{\epsilon}
\nwc{\vareps}{\varepsilon}

\def\ep{\epsilon}

\def\sq2{\sqrt{2}}

\def\t2{{\mathbb T}^2}
\def\s2{{\mathbb S}^2}

\def\N{\mathbb{N}}

\def\R{\mathbb{R}}

\def\C{\mathbb{C}}

\nwc{\lap}{\bigtriangleup}
\nwc{\rest}{\restriction}
\nwc{\Diff}{\operatorname{Diff}}
\nwc{\diam}{\operatorname{diam}}
\nwc{\Res}{\operatorname{Res}}
\nwc{\Spec}{\operatorname{Spec}}
\nwc{\Vol}{\operatorname{Vol}}
\nwc{\Op}{\operatorname{Op}}
\nwc{\supp}{\operatorname{supp}}
\nwc{\Span}{\operatorname{span}}

\nwc{\dia}{\varepsilon}
\nwc{\cut}{f}
\nwc{\qm}{u_\hbar}

\def\hto0{\xrightarrow{\hbar\to 0}}

\def\rto0{\xrightarrow{r\to 0}}

\providecommand{\norm}[1]{\lVert#1\rVert}

\nwc{\la}{\langle}
\nwc{\ra}{\rangle}
\nwc{\lp}{\left(}
\nwc{\rp}{\right)}

\nwc{\bequ}{\begin{equation}}
\nwc{\be}{\begin{equation}}
\nwc{\ben}{\begin{equation*}}
\nwc{\bea}{\begin{eqnarray}}
\nwc{\bean}{\begin{eqnarray*}}
\nwc{\bit}{\begin{itemize}}
\nwc{\bver}{\begin{verbatim}}

\nwc{\eequ}{\end{equation}}
\nwc{\ee}{\end{equation}}
\nwc{\een}{\end{equation*}}
\nwc{\eea}{\end{eqnarray}}
\nwc{\eean}{\end{eqnarray*}}
\nwc{\eit}{\end{itemize}}
\nwc{\ever}{\end{verbatim}}

%% file: 181105QE_revised.bbl
\begin{thebibliography}{10} 

\bibitem{ATV}
M.~Ab\'ert, A.~Thom and B.~Vir\'ag, \emph{Benjamini-Schramm convergence and pointwise convergence of the spectral measure}, author homepage.

\bibitem{Alt1}
A.~De~Luca, B.~L.~Altshuler, V.~E.~Kravtsov, and A.~Scardicchio, \emph{Anderson Localization on the Bethe Lattice: Nonergodicity of Extended States}, Phys. Rev. Lett. \textbf{113} (2014) 046806.

\bibitem{Alt2}
A.~De~Luca, B.~L.~Altshuler, V.~E.~Kravtsov, and A.~Scardicchio, \emph{Support set of random wave-functions on the Bethe lattice}, arXiv 2013.

\bibitem{AW2}
M.~Aizenman, S.~Warzel, \emph{Absolutely continuous spectrum implies ballistic transport for quantum particles in a random potential on tree graphs}, J. Math. Phys. \textbf{53} (2012) 095205, 15.

\bibitem{ASW}
M.~Aizenman, M.~Shamis and S.~Warzel, \emph{Resonances and partial delocalization on the complete graph}, Ann. Henri Poincar\'e \textbf{16} (2015), 1969--2003.

\bibitem{AL}
D.~Aldous, R.~Lyons, \emph{Processes on unimodular random networks}, Electron. J. Probab. \textbf{12} (2007) 1454--1508.

\bibitem{A}
N.~Anantharaman, \emph{Quantum ergodicity on regular graphs}. To appear in Comm. Math. Phys.

\bibitem{A17}
N.~Anantharaman, \emph{Some relations between the spectra of simple and non-backtracking random walks}, preprint arXiv:1703.03852 (2017).

\bibitem{ALM}
N.~Anantharaman, E.~Le~Masson, \emph{Quantum ergodicity on large regular graphs}, Duke Math. Jour. \textbf{164} (2015) 723--765.

\bibitem{AS}
N.~Anantharaman, M.~Sabri, \emph{Poisson kernel expansions for Schr\"odinger operators on trees}, preprint arXiv:1610.05907.

\bibitem{AS2}
N.~Anantharaman, M.~Sabri, \emph{Quantum ergodicity for the Anderson model on regular graphs}, preprint (2017).

\bibitem{BaSz16}
A.~Backhausz, B.~Szegedy, \emph{On the almost eigenvectors of random regular graphs}, preprint arXiv : 1607.04785.

\bibitem{BKY}
R.~Bauerschmidt, A.~Knowles, H.T.~Yau, \emph{Local semicircle law for random regular graphs}, preprint arXiv:1503.08702.

\bibitem{BHKY}
R.~Bauerschmidt, J.~Huang, A.~Knowles, H.T.~Yau, \emph{Bulk eigenvalue statistics for random regular graphs}, preprint arXiv:1505.06700.

\bibitem{BHY}
R.~Bauerschmidt, J.~Huang, H.T.~Yau, \emph{Local Kesten--McKay law for random regular graphs}, preprint arXiv:1609.09052.

\bibitem{BS}
I.~Benjamini, O.~Schramm, \emph{Recurrence of distributional limits of finite planar graphs}, Electron. J. Probab. \textbf{6} (2001) 13 pp.

\bibitem{B}
I.~Benjamini, \emph{Coarse Geometry and Randomness. \'Ecole d'\'Et\'e de Probabilit\'es de Saint-Flour XLI-2011}, Springer 2013.

\bibitem{BT}
M.~Berry, M.~Tabor, \emph{Level clustering in the regular spectrum,}
Proc. Royal Soc.
A 356 (1977), 375--394.

\bibitem{BGS1}
O. Bohigas, M. J. Giannoni, and C. Schmit. \emph{Characterization of chaotic quantum spectra and universality of level fluctuation laws.} Phys. Rev. Lett., 52:1--4, 1984. 

\bibitem{BGS2}
O. Bohigas, M. J. Giannoni, and C. Schmit. \emph{Spectral properties of the laplacian and random matrix theories}. Journal De Physique Lettres, 45:1015--1022, 1984. 

\bibitem{Bor15}
C.~Bordenave, \emph{On Quantum Percolation in Finite Regular Graphs}, Ann. Henri Poincar\'e \textbf{16} (2015), 2465--2497.

\bibitem{BourgadeYau13}
P.~Bourgade, H.-T~.Yau, \emph{ Eigenvector Moment Flow and Local Quantum Unique Ergodicity}, Comm. Math. Phys. \textbf{350} (2017) 231--278.

\bibitem{BL}
S.~Brooks, E.~Lindenstrauss, \emph{Non-localization of eigenfunctions on large regular graphs}, Israel J. Math. \textbf{193} (2013), 1--14.
%
\bibitem{BLML}
S.~Brooks, E.~Le Masson, E.~Lindenstrauss, \emph{Quantum ergodicity and averaging operators on the sphere}, Int. Math. Res. Not. \textbf{19} (2016) 6034--6064.

\bibitem{CdV85}
Y.~Colin~de~Verdi\`ere, \emph{Ergodicit\'e et fonctions propres du laplacien}, Comm. Math. Phys. \textbf{102} (1985) 497--502.

\bibitem{Diac91}
P.~Diaconis, D.~Stroock, \emph{Geometric bounds for eigenvalues of {M}arkov chains}, Ann. Appl. Probab. \textbf{1} (1991) 36--61.

\bibitem{Dumitriu}
I.~Dumitriu, S.~Pal, \emph{Sparse regular random graphs: Spectral density and eigenvectors}, Ann. Prob. \textbf{40} (2012) 2197–-2235.

\bibitem{EKYY}
L.~Erd{\H{o}}s, A.~Knowles, H.-T~Yau and J.~Yin, \emph{Spectral statistics of {E}rd{\H o}s-{R}\'enyi graphs {I}: {L}ocal semicircle law}, Ann. Probab, \textbf{41} (2013) 2279--2375.

\bibitem{ESY09}
L.~Erd{\H{o}}s, B.~Schlein and H.-T~Yau, \emph{Local semicircle law and complete delocalization
for Wigner random matrices}, Comm. Math. Phys., \textbf{287} (2009) 641--655.

\bibitem{ESY09-2}
L.~Erd{\H{o}}s, B.~Schlein and H.-T~Yau, \emph{Semicircle law on short scales and delocalization
of eigenvectors for Wigner random matrices}, Ann. Probab., \textbf{37} (2009) 815--852.

\bibitem{Geisinger}
L.~Geisinger, \emph{Convergence of the density of states and delocalization of eigenvectors on random regular graphs}, J. Spectr. Theory \textbf{5} (2015) 783--827.

\bibitem{Keating}
J.~P.~Keating, \emph{Quantum graphs and quantum chaos}, In Analysis on graphs and its applications, vol-
ume 77 of Proc. Sympos. Pure Math., p. 279–-290, Amer. Math. Soc. 2008.

\bibitem{Klein}
A.~Klein, \emph{Extended States in the Anderson Model on the Bethe Lattice}, Adv. Math. \textbf{133} (1998) 163--184.

\bibitem{Klenke}
A.~Klenke, \emph{Probability Theory. A Comprehensive Course}, Second Edition, Springer 2014.

\bibitem{Mark}
J.~Marklof, \emph{Pair correlation densities of inhomogeneous quadratic forms}, Ann. of Math. \textbf{158} (2003) 419--471.

\bibitem{Mey01}
C.~D.~Meyer, Matrix Analysis and Applied Linear Algebra, SIAM 2001.

\bibitem{PP}
W.~Parry, M.~Pollicott, \emph{Zeta functions and the periodic orbit structure of hyperbolic dynamics}, Ast\'erisque No. 187--188 (1990).

\bibitem{Salez}
J.~Salez, \emph{Some implications of local weak convergence for sparse random graphs}, Hal Id: tel-00637130.

\bibitem{SS}
C.~Schumacher, F.~Schwarzenberger, \emph{Approximation of the Integrated Density of States on Sofic Groups}, Ann. Henri Poincar\'e \textbf{16} (2015) 1067--1101.

\bibitem{SarnakSchur}
P.~Sarnak, \emph{Arithmetic quantum chaos}, In The Schur lectures (1992) (Tel Aviv), volume 8 of Israel
Math. Conf. Proc., p. 183--236. Bar-Ilan Univ., Ramat Gan, 1995.

\bibitem{SarnakPoisson}
P.~Sarnak, \emph{Values at integers of binary quadratic forms} In Harmonic analysis and number theory
(Montreal, PQ, 1996), volume 21 of CMS Conf. Proc., p. 181--203. Amer. Math. Soc. 1997.

\bibitem{Smi07}
U.~Smilansky, \emph{Quantum chaos on discrete graphs}, J. Phys. A, \textbf{40} (2007) F621–-F630.

\bibitem{Smi10}
U.~Smilansky, \emph{Discrete graphs - a paradigm model for quantum chaos}, S\'eminaire Poincar\'e, XIV:1--
26, 2010.

\bibitem{Sni}
A.~I.~\v{S}nirel'man, \emph{Ergodic properties of eigenfunctions}, Uspehi Mat. Nauk, \textbf{29} (1974) 181--182.

\bibitem{TMS}
K.~S.~Tikhonov, A.~D.~Mirlin, and M.~A.~Skvortsov, \emph{Anderson localization and ergodicity on random regular graphs}, Phys. Rev. B, \textbf{94} (2016) 220203.

\bibitem{TranVu}
L.~Tran, V.~Vu, and K.~Wang, \emph{Sparse random graphs: eigenvalues and eigenvectors}, Random
Structures Algorithms, \textbf{42} (2013) 110--134.

\bibitem{Zel87}
S.~Zelditch, \emph{Uniform distribution of eigenfunctions on compact hyperbolic surfaces}, Duke Math. J. \textbf{55} (1987) 919--941.

\bibitem{ZelC}
S.~Zelditch, \emph{Quantum ergodicity of $C^{\ast}$ dynamical systems}, Comm. Math. Phys. \textbf{177} (1996) 502--528.

\end{thebibliography}
